\documentclass{mystyle}
\usepackage{graphicx}
\usepackage{amsfonts, amssymb, amsmath, mathrsfs, latexsym}
\usepackage[all]{xy}
\usepackage[dvipsnames]{xcolor}
\usepackage{mathtools}
\usepackage{tikz}
\usetikzlibrary{positioning}
\usepackage[dvipsnames]{xcolor}
\usepackage[normalem]{ulem}
\usepackage{booktabs}

\ifpdf \usepackage[unicode,pdftex]{hyperref} \input glyphtounicode\pdfgentounicode=1
\else\usepackage[unicode,dvipdfm]{hyperref}\fi
\usepackage{pgfplots}
\usepgfplotslibrary{fillbetween}
\pgfplotsset{width=11.5cm, compat=1.8}

\newtheorem{thm}{Theorem}[section]
\newtheorem{prop}[thm]{Proposition}
\newtheorem{lem}[thm]{Lemma}
\newtheorem{cor}[thm]{Corollary}

\theoremstyle{definition}

\newtheorem{deff}[thm]{Definition}

\newtheorem{rmk}[thm]{Remark}
\newtheorem{exam}[thm]{Example}
\newtheorem*{mcs}{Modified Assertion of Severi}

\newcommand{\enm}[1]{\ensuremath{#1}}          %
\newcommand{\op}[1]{\operatorname{#1}}
\newcommand{\cal}[1]{\mathcal{#1}}

\newcommand{\CC}{\enm{\mathbb{C}}}

\newcommand{\NN}{\enm{\mathbb{N}}}

\newcommand{\ZZ}{\enm{\mathbb{Z}}}
\newcommand{\FF}{\enm{\mathbb{F}}}
\renewcommand{\AA}{\enm{\mathbb{A}}}
\newcommand{\PP}{\enm{\mathbb{P}}}

\newcommand{\Dd}{\enm{\cal{D}}}
\newcommand{\Ee}{\enm{\cal{E}}}
\newcommand{\Ff}{\enm{\cal{F}}}
\newcommand{\Gg}{\enm{\cal{G}}}
\newcommand{\Hh}{\enm{\cal{H}}}
\newcommand{\Ii}{\enm{\cal{I}}}

\newcommand{\Ll}{\enm{\cal{L}}}
\newcommand{\Mm}{\enm{\cal{M}}}
\newcommand{\Nn}{\enm{\cal{N}}}
\newcommand{\Oo}{\enm{\cal{O}}}

\newcommand{\Xx}{\enm{\cal{X}}}

\renewcommand{\phi}{\varphi}
\renewcommand{\theta}{\vartheta}
\renewcommand{\epsilon}{\varepsilon}

\newcommand{\Aut}{\op{Aut}}

\newcommand{\dashdownarrow}{\mathrel{\rotatebox[origin=t]{-270}{\reflectbox{$\dashrightarrow$}}}}

\renewcommand{\to}[1][]{\xrightarrow{\ #1\ }}

\newcommand{\GG}{\ensuremath{\mathbb{G}}}
\newcommand{\h}{\ensuremath{\mathcal}}
\newcommand{\HL}{\ensuremath{\mathcal{H}^\mathcal{L}_}}

\newcommand{\w}{\widetilde}
\newcommand{\ce}{C_\mathcal{E}}
\newcommand{\st}{\tilde}
\newcommand{\vni}{\vskip 4pt \noindent}
\newcommand{\al}{\ensuremath{\alpha}}

\newcommand{\uin}{\rotatebox[origin=c]{90}{$\in$}}

\newcommand{\csi}{Castelnuovo-Severi inequality}

\title[Hilbert scheme of linearly normal curves]
{Hilbert scheme of  linearly normal curves in $\mathbb{P}^r$ with index of speciality five and beyond}
\thanks{
{ The author was supported in part by National Research Foundation of South Korea (2022R1I1A1A01055306).} }

\author[Changho Keem]{Changho Keem}
\address{
Department of Mathematics,
Seoul National University\\
Seoul 151-742,  
South Korea}
\email{ckeem1@gmail.com}
\date{\today}
\begin{document}
\begin{abstract}
We study the Hilbert scheme of smooth, irreducible, non-degenerate and linearly normal curves of degree $d$ and genus $g$ in $\mathbb{P}^r$ ($r\ge 3$) whose complete and very ample hyperplane linear series $\mathcal{D}$ have  relatively small index of speciality $i(\mathcal{D})=g-d+r$. In particular we completely determine the existence as well as the non-existence  of Hilbert schemes of linearly normal curves $\mathcal{H}^{\mathcal{L}}_{d,g,r}$ {\bf for every possible triples $(d,g,r)$} with $i(\mathcal{D})=5$ and $r\ge 3$. 
We also determine the irreducibility of the Hilbert scheme $\mathcal{H}^{\mathcal{L}}_{g+r-5,g,r}$ when the genus $g$ is near to the minimal possible value with respect to the dimension of the projective space $\mathbb{P}^r$   for which $\mathcal{H}^{\mathcal{L}}_{g+r-5,g,r}\neq\emptyset$,  say $r+9\le g\le r+11$.
In the course of proofs of key results, we show the existence of linearly normal curves
of degree $d\ge g+1$ with arbitrarily given index of speciality with some mild restriction on 
the genus $g$. 
\end{abstract}
\maketitle		
\section{
An overview,  motivation and preliminary set up}

Let $\h{H}_{d,g,r}$ be the Hilbert scheme of smooth, irreducible and non-degenerate curves of degree $d$ and genus $g$ in $\PP^r$. 
We denote by $\mathcal{H}^\mathcal{L}_{d,g,r}$ the union of components of $\mathcal{H}_{d,g,r}$ whose general element corresponds to a linearly normal curve $C\subset \PP^r$.
By abuse of terminology, we say that a component of the Hilbert scheme $\h{H}_{d,g,r}$ has index of speciality $\alpha$ if 
the hyperplane series $\h{D}=g^r_d$ of a general element of the component has index of speciality $\alpha$, i.e. 
$$h^1(C,\h{D})=g-d+\dim |\h{D}| =\alpha \ge g-d+r.$$
\noindent
It is possible that there may exist a component of the Hilbert scheme $\h{H}_{d,g,r}$ which has index of speciality strictly greater than $g-d+r$; cf. \cite{Keem}. In other words, there may exist a (component of) Hilbert scheme entirely consisting of non-linearly normal curves.
However the index of speciality of any component of $\HL{d,g,r}$ is $g-d+r$, which we define as the {\bf index of speciality of a Hilbert scheme of linearly normal curves}. 

\vni
\noindent
We recall the following modified assertion of Severi (\cite{Sev}), which has been given attention by some authors recently; cf. \cite{JPAA, KK3}.
\begin{mcs}  A nonempty $\mathcal{H}^\mathcal{L}_{d,g,r}$ is irreducible for any triple $(d,g,r)$ in the Brill-Noether range  $$\rho (d,g,r)=g-(r+1)(g-d+r)\ge 0.$$ 
\end{mcs}

\noindent
We note that the modified assertion of Severi makes sense only if the index of speciality of Hilbert scheme of linearly normal curves $\HL{d,g,r}$ is non-negative.
In other words, $\HL{d,g,r}\neq\emptyset$ only when $g-d+r\ge 0$. Through a preliminary attempt to settle down the Modified Assertion of Severi \cite{Sev, KK3}, one has a quite extensive knowledge about certain basics of the Hilbert scheme of linearly normal curves of small index of speciality $\alpha \le 4$, which can be summarized as follows; cf. \cite{lengthy, JPAA, KK3, speciality}.

\begin{rmk}\label{eresult} We assume $r\ge 3$.
\begin{itemize}
\item[(i)]{$\alpha =0$}: $\mathcal{H}^\mathcal{L}_{g+r,g,r}\neq\emptyset$ and is irreducible.
\item[(ii)]{ $\alpha=1$}: $\mathcal{H}^\mathcal{L}_{g+r-1,g,r}\neq\emptyset$ and is irreducible for $g\ge r+1$ and is empty for $g\le r$.
\item[(iii)]{$\alpha=2$}: $\mathcal{H}^\mathcal{L}_{g+r-2,g,r}\neq\emptyset$ and is irreducible for every $g\ge r+3$ and is empty for
 $g\le r+2$.

\item[(iv)]{$\alpha=3$}: 
\begin{enumerate}
\item[(a)] $\h{H}^\h{L}_{g+r-3,g,r}=\emptyset$ for $g\le r+4$ and $r\ge5$.
\item[(b)] For $g\ge r+5$ and $r\ge 5$, $\h{H}^\h{L}_{g+r-3,g,r}\neq\emptyset$ unless $g=r+6$ and $r\ge 10$.
\item[(c)]$\h{H}^\h{L}_{g+r-3,g,r}$ is irreducible for every $g\ge 2r+3$ and is reducible for almost all $g$ in the range $r+5\le g\le 2r+2$.
\item[(d)] Somewhat stronger results hold for curves in $\PP^3$ and $\PP^4$. We refer \cite[Remark 3.1]{lengthy} for details.
\end{enumerate}
\item[(v)]
{$\alpha =4$}: \begin{enumerate}
\item[(a)] $\HL{g+r-4,g,r}=\emptyset$ for $g\le r+6$, $r\ge 5$. 
\item[(b)] $\HL{g+r-4,g,r}\neq\emptyset$
for  $g=r+7$, $r\ge 5$.
\item[(c)] $\HL{g+r-4,g,r}\neq\emptyset$ if and only if $~5\le r\le 8$ for $g=r+8$, $r\ge 5$ .
\item[(d)] $\HL{g+r-4,g,r}\neq\emptyset$ if and only if $~5\le r\le 11$ for $g=r+9$, $r\ge 5$. 
\item[(e)] $\HL{g+r-4,g,r}\neq\emptyset$
for any $g\ge r+10$, $r\ge 5$.
\item[(f)] Furthermore, the irreducibility has been fully determined in all the peculiar cases  in (c) as well as several cases in (d).
\item[(g)] For curves in $\PP^3$ and $\PP^4$ with $\al =4$ -- which are not explicitly mentioned in the above -- we refer \cite[Remark 2.1]{speciality} and references therein for details.
\end{enumerate}

\vni

\end{itemize}
\end{rmk}
\vni Results mentioned in Remark \ref{eresult} indicate that the Modified Assertion of Severi may continue to hold even beyond the Brill-Noether range $\rho (d,g,r)\ge 0$ espeically  when the index of speciality is rather small. For example, in the case $\al=3$, the Brill-Noether range $\rho(g+r-3,g,r)\ge 0$ is $g\ge 3r+3$. However the irreducibility holds for every $g\ge 2r+3$. 
\noindent
\vni
There are a couple of  advantages in studying  the Hilbert scheme of linearly normal curves according to a given fixed  index of speciality {\bf $\alpha$}. One notes that  the residual series of the very ample and complete hyperplane series $g^r_d$ corresponding to a general element of a component of 
$\HL{d,g,r}$ of index of speciality $\al$  has the fixed dimension $\alpha -1$, regardless of the values of the genus $g$ and the 
degree $d$. Therefore one may work  more effectively in exploring out several important properties of Hilbert schemes under consideration by studying the family of curves in a fixed projective space $\PP^{\alpha -1}$ induced by the residual series of the (complete) hyperplane series.  

\vni
To be more specific, in the cases $\alpha =2, 3, 4$, 
the residual series of hyperplane series of a general element of a component of $\HL{d,g,r}$ are pencils, nets or webs. Therefore one may expect  to squeeze out several basic properties of the Hilbert scheme $\HL{d,g,r}$ such as existence, irreducibility, gonality of a general element in each component,  as well as the number of moduli by considering the corresponding property of Hurwitz space,  Severi variety of plane curves or the Severi variety of curves on a surface of low degree in $\PP^3$. 
\vni
There is another advantage in studying the Hilbert scheme $\HL{d,g,r}$ according to its index of speciality. Unless the genus $g$ is fairly high with respect to the degree $d$ of the projective curve in $\PP^r$, a general curve in a component of $\HL{d,g,r}$ defined by a very ample complete hyperplane series $\h{D}$ is by no means extremal or nearly extremal i.e. the genus $g$ is rather far below the maximal possible genus of curves of degree $d$ in $\PP^r$. However the residual series $\h{D}^\vee$ of $\h{D}$ with respect to the canonical series may induce an extremal or a nearly extremal curve in $\PP^{\al -1}$. If this occurs, studying (extrinsic) curves defined by the residual series $\h{D}^\vee$ may become relatively easier than directly handling the original curve in $\PP^r$.  We will encounter this sort of phenomenon several times in this article.

\vni
In this paper we study the Hilbert scheme of linearly normal curves with index of speciality $\alpha =5$. Equivalently we study Hilbert schemes of non-degenerate, smooth and linearly normal curves in $\PP^r$ ($r\ge 3$)  of degree $d=g+r-5$. We completely determine the existence (i.e. non-emptyness)  as well as the non-existence (if any) of $\HL{d,g,r}$ for every triple
$(d,g,r)$ with $g-d+r=5$; cf. Theorem \ref{sub6x} and Figure \ref{fig:existence} in the end of \S 3. In order to show the non-emptiness of $\HL{d,g,r}$ for a given triple $(d,g,r)$, it is sufficient to demonstrate the existence of a complete and very ample linear series $g^r_d$ on a particular curve. We will carry this out on a suitably chosen general $k$-gonal curve of genus $g$. 

\vni
We also determine the irreducibility of $\HL{d,g,r}$ when the genus $g$ is near to the minimal possible value  $r+9$ with respect to $r$ for which $\HL{d,g,r}\neq\emptyset$; cf. Theorem \ref{sub6x}, Proposition \ref{g=r+9}, Theorem \ref{r+10}, Proposition \ref{d13r4}, Theorem \ref{$g=13$}, Theorem \ref{2r+6}, Proposition \ref{r+117}, Theorem \ref{g=r+12} and Figure \ref{fig:irreducibility} in the end of \S 6.
Happily, when $g$ is equal or near to $r+9$, the curve $C\subset\PP^r$ corresponding to a general element of a component of $\HL{d,g,r}$ or its residual curve $C^\vee$ -- which is by definition the curve induced by the residual series of the very ample and complete hyperplane series of $C$ -- necessarily lies on a surface of relatively small degree. This is a consequence of a  classical result which says that an integral projective curve in $\PP^r$ whose (arithmetic) genus $g$  is in between
the first and second Castelnuovo genus bound -- which we call a nearly extremal curve -- necessarily lies on a surface of minimal or near to minimal degree. By analyzing the family of such curves on these special
surfaces, one may determine the irreducibility of $\HL{d,g,r}$ for some particular triples $(d,g,r)$ especially when the genus $g$ is near to the minimal possible value with 
respect to the dimension of the ambient projective space $\PP^r$.  In carrying out this process,  non-trivial results such as the the irreducibility of the Severi
variety of nodal curves on Hirzebruch surfaces plays an important role.
We will be more specific with our road map for the determination of the irreducibility of $\HL{d,g,r}$ in corresponding sections.

{


\vni
For notation and conventions, we  follow those in \cite{ACGH} and \cite{ACGH2}; e.g. $\pi (d,r)$ is the maximal possible arithmetic genus of an irreducible,  non-degenerate and reduced curve of degree $d$ in $\PP^r$ which is usually referred as the first Castelnuovo genus bound. We shall refer to curves $C\subset\PP^r$ of degree $d$ whose (arithmetic) genus equals $\pi(d,r)$ as {\bf extremal curves}. $\pi_1(d,r)$ is the so-called the second Castelnuvo genus bound which is the maximal possible arithmetic genus of  an irreducible, non-degenerate and reduced curve of degree $d$ in $\PP^r$ not lying on a  surface of minimal degree $r-1$; cf. \cite[page 99]{H1}, \cite[page 123]{ACGH}.
We shall call curves $C\subset\PP^r$ of degree $d$ and (arithmetic) genus $g$ such that $\pi_1(d,r)<g\le\pi(d,r)$ {\bf nearly extremal curves}.
\vni
Following classical terminology, a linear series of degree $d$ and dimension $r$ on a smooth curve $C$ is denoted by $g^r_d$.
A base-point-free linear series $g^r_d$ ($r\ge 2$) on a smooth curve $C$ is called {\bf birationally very ample} when the morphism 
$C \rightarrow \mathbb{P}^r$ induced by  the $g^r_d$ is generically one-to-one onto (or is birational to) its image curve.
A base-point-free linear series $g^r_d$ on $C$  is said to be compounded of an involution ({\bf compounded} for short) if the morphism induced by the linear series $g^r_d$ gives rise to a non-trivial covering map $C\rightarrow C'$ of degree $k\ge 2$. 
Throughout we work exclusively over the field of complex numbers.

\vni
The organization of this paper is as follows. In the remainder of this section,  we briefly recall some other terminologies and our preliminary set up.
In the next section we list up several auxiliary results and prove a couple of lemmas which are necessary for our study. In the third section, we determine all the triples $(d,g,r)$ with $r\ge 3$ for which $\mathcal{H}^\mathcal{L}_{g+r-5,g,r}\neq\emptyset$. 

\vni
Subsequent three sections are devoted to  determining the irreducibility of a certain non-empty $\mathcal{H}^\mathcal{L}_{g+r-5,g,r}$. Specifically we determine the
irreducibility of $\HL{2r+4,r+9,r}$ for every $r\ge 3$ in \S 4; cf Proposition \ref{g=r+9}. We stress that $\HL{2r+4,r+9,r}$ is the first non-empty Hilbert scheme of index of speciality $\alpha =5$ with the smallest possible genus $g=r+9$ with respect to the dimension of the ambient projective space $\PP^r$ unless $r=4$. In \S 5,  we determine the irreducibility of $\HL{g+r-5,g,r}$
for $g=r+10$ ($r\ge 3$ and $r\neq 4$); cf. Theorems \ref{r+10} and \ref{$g=13$}).  In \S 6, we also determine the irreducibility of $\HL{g+r-5,g,r}$ for $g=r+11$, $r\ge 7$
(Theorem \ref{2r+6},  Proposition \ref{r+117}) as well as for  $g=r+12$, $9\le r\le 14$ (Theorem \ref{g=r+12}).

\vni
In the final section we state and discuss generalized statements, which partially covers those results for low index of specialities. The results we state there have a certain numerical restriction. However, it extends the existence of linearly normal curves to an arbitrarily given index of speciality $\al$ such that $\al\ge r+1$. We also list up certain of irreducibility results in several  low genus cases which may be obtained by utilizing  similar kind techniques we use for lower index of speciality $\al\le 5$.

\vni
{
 
 \vni
In the appendix, we provide proofs of several results which seem to have been known to people rather widely; however the author could not trace explicit sources in the literature and therefore provide proofs for the convenience of readers.

\noindent
\vni
We now briefly recall several terminologies, fundamental results and basic frameworks for our study which are  well-known; cf. \cite{ACGH2}  or \cite[\S 1 and \S 2]{AC2}.

\vni
Let $\mathcal{M}_g$ be the moduli space of smooth curves of genus $g$. Given an isomorphism class $[C] \in \mathcal{M}_g$ corresponding to a smooth irreducible curve $C$, there exist a neighborhood $U\subset \mathcal{M}_g$ of the class $[C]$ and a smooth connected variety $\mathcal{M}$ which is a finite ramified covering $h:\mathcal{M} \to U$, as well as  varieties $\mathcal{C}$, $\mathcal{W}^r_d$ and $\mathcal{G}^r_d$ proper over $\mathcal{M}$ with the following properties:
\begin{enumerate}
\item[(1)] $\xi:\mathcal{C}\to\mathcal{M}$ is a universal curve, i.e. for every $p\in \mathcal{M}$, $\xi^{-1}(p)$ is a smooth curve of genus $g$ whose isomorphism class is $h(p)$,
\item[(2)] $\mathcal{W}^r_d$ parametrizes the pairs $(p,L)$ where $L$ is a line bundle of degree $d$ and $h^0(L) \ge r+1$ on $\xi^{-1}(p)$,
\item[(3)] $\mathcal{G}^r_d$ parametrizes the couples $(p, \mathcal{D})$, where $\mathcal{D}$ is possibly an incomplete linear series of degree $d$ and dimension $r$ on $\xi^{-1}(p)$.
\end{enumerate}

\noindent
For a complete linear series $\h{E}$ on a smooth curve $C$, the residual series $|K_C-\h{E}|$ is denoted by $\h{E}^\vee$.
{Given an irreducible family $\h{F}\subset\h{G}^r_d$ with some geometric meaning, whose general member is complete,  the closure of the family $\{\h{E}^\vee| \h{E}\in \h{F}, ~\h{E} \textrm{ is complete}\}\subset \h{W}^{g-d+r-1}_{2g-2-d}$ is  denoted by $\h{F}^\vee$.

\noindent
Let $\widetilde{\mathcal{G}}$ ($\widetilde{\mathcal{G}}_\mathcal{L}$,  respectively) be  the union of components of $\mathcal{G}^{r}_{d}$ whose general element $(p,\mathcal{D})$ corresponds to a very ample (very ample and complete, respectively) linear series $\mathcal{D}$ on the curve $C=\xi^{-1}(p)$. By recalling that an open subset of $\mathcal{H}_{d,g,r}$ consisting of elements corresponding to smooth irreducible and non-degenerate curves is a $\mathbb{P}\textrm{GL}(r+1)$-bundle over an open subset of $\widetilde{\mathcal{G}}$, {\bf the irreducibility of $\widetilde{\mathcal{G}}$ guarantees the irreducibility of $\mathcal{H}_{d,g,r}$. Likewise, the irreducibility of $\widetilde{\mathcal{G}}_\mathcal{L}$ ensures the irreducibility of 
 $\mathcal{H}_{d,g,r}^\mathcal{L}$.}

\noindent
\vni
We recall the following  fundamental fact regarding the scheme $\mathcal{G}^{r}_{d}$ which is also well-known; cf. \cite[2.a]{H1} and \cite[Ch. 21, \S 3, 5, 6, 11, 12]{ACGH2}. 
\begin{prop}\label{facts}
For non-negative integers $d$, $g$ and $r$, let $$\rho(d,g,r):=g-(r+1)(g-d+r)$$ be the Brill-Noether number.
The dimension of any component of $\mathcal{G}^{r}_{d}$ is at least $$\lambda(d,g,r):=3g-3+\rho(d,g,r), $$ 
hence the dimension of any component of $\h{H}_{d,g,r}$ is at least
$$\h{X}(d,g,r):=\lambda (d,g,r)+\dim\PP GL(r+1).$$ Moreover, if $\rho(d,g,r)\ge 0$, there exists a unique component $\mathcal{G}_0$ of $\widetilde{\mathcal{G}}$ which dominates $\mathcal{M}$(or $\mathcal{M}_g$).
	\end{prop}

\section{Glossary of relevant generalities and auxiliary results }

\noindent
In this section we prepare a couple of lemmas and collect several facts which are relevant to our study.  Even though almost all the results we quote in this section are quite well known, we list them up for the convenience of readers. We first recall the following generalities regarding the Severi variety of (nodal) curves on a  Hirzebruch surface $\mathbb{F}_e$, i.e. a geometrically ruled surface over $\PP^1$ with invariant $e\ge 0$. 
\begin{deff}
\begin{enumerate}
\item[(i)]
Given a Hirzebruch surface $\mathbb{F}_e=\PP (\h{E})\stackrel{\pi}{\rightarrow}\PP^1$ where $\h{E}=\h{O}_{\PP^1}\oplus\h{O}_{\PP^1}(-e)$, let $C_0$ be the section with $C_0^2=-e\le 0$ and $f$ be a fibre of $\pi$.  Given a very ample linear system $\h{L}=|aC_0+bf|$ on $\mathbb{F}_e$, let
 $$p_a(\h{L})=(a-1)(b-1-\frac{1}{2}ae)$$ be the arithmetic genus of an integral curve belonging to $\h{L}$.

\item[(ii)]
Given an integer $0\le g\le p_a(\h{L})$, we set
 $\delta =p_a(\h{L})-g$. We denote by $$\Sigma_{\h{L}, \delta}\subset \h{L}=\PP(H^0(X,\h{L}))$$ the (equi-singular) Severi variety which is the closure of the locus of integral curves in the linear system $\h{L}$ whose singular locus consists of exactly $\delta$ nodes. 
\item[(iii)]
Let $\Sigma_{\h{L},g}$ be the  (equi-generic) Severi variety which is the closure of the locus of integral curves of geometric genus $g$  in the linear system $\h{L}$. 
\end{enumerate}
\end{deff}

\noindent
\vni
\noindent
We shall make use the following well-known results adopted for our specific situation. Almost all the results quoted below are known to be true in a more general context. Readers are advised to refer \cite{DS}, which provides an excellent treatment on Severi varieties on rational surfaces in general. 

\begin{rmk}\label{Severi}
\begin{enumerate}
\item[(i)] A general member of every irreducible component of the equi-generic Severi variety $\Sigma_{\h{L},g}$ is a nodal curve; cf. \cite[Proposition 2.1]{H2}, \cite[Theorem B2, p.177]{DS} and \cite[pp. 105-117]{H3}.  
\item[(ii)] The equi-singular Severi variety $\Sigma_{\h{L}, \delta}$ is {\bf irreducible} of the expected dimension $$\dim|\h{L}|-\delta=\frac{(a+1)(2b+2-ae)}{2}-1-\delta$$
if non-empty; cf. \cite[Proposition 2.11, Theorem 3.1]{tyomkin}.
\end{enumerate}
\end{rmk}
 \vni
Mainly for dimension count of certain families of curves or linear series under our investigation, we make a note of
 the following facts regarding  a surface $S\subset \PP^{n+1}$ of minimal degree $n\ge 2$.

\begin{rmk}\label{minimal} For an irreducible and  non-degenerate surface $S\subset\PP^{n+1}$ of minimal degree $n$, one of the following  holds; cf. \cite[IV, Exercises pp 53-54]{Beauville}.

\begin{itemize}
\item[(i)]
$S$ is a smooth rational normal scroll, in which case we have;
\begin{equation}\label{lsdimension}
\dim |aH+bL| =\frac{1}{2}a(a+1)n+(a+1)(b+1)-1
\end{equation}
where $H$ (resp. $L$) is the class of a hyperplane section (resp. the class a line of the ruling).
For any reduced and irreducible curve $C\subset\PP^{n+1}$ of degree $d$ contained in the linear system $|aH+bL|$, we have
\begin{equation}\label{scrolldegree}
d=(aH+bL)\cdot H=na+b
\end{equation} and by adjunction formula
\begin{equation}\label{scrollgenus}
p_a(C)
=\frac{1}{2}\,a \left( a-1 \right)\cdot n  + \left( a-1 \right) 
 \left( b-1 \right).
 \end{equation}
Throughout, whenever we deal with a rational normal surface scroll $S\subset\PP^{n+1}$ of minimal degree, $H$ (resp. $L$) always denotes the class of a hyperplane section (resp. the class a line of the ruling). 
\item[(ii)] $S$ is  a Veronese surface in $\PP^5$ in which case any reduced irreducible curve in $S$ has even degree $2\cdot e$ having the arithmetic genus $\frac{(e-1)(e-2)}{2}$. 
\item[(iii)]
$S$ is a cone over a rational normal curve in a hyperplane $\PP^{n}\subset\PP^{n+1}$, which is the image of the Hirzebruch surface $\mathbb{F}_n=\mathbb{P}(\h{O}_{\PP^1}\oplus\h{O}_{\PP^1}(-n))$. We denote by  $h$ (resp. $f$ ) the class in $\textrm{Pic}(\mathbb{F}_n)$ of the tautological bundle $\h{O}_{\mathbb{F}_n}(1)$ (resp. of a fibre); $f^2=0, f\cdot h=1, h^2=n, h=C_0+nf$.  The morphism $\mathbb{F}_n\rightarrow S\subset\PP^{n+1}$ given by the complete linear system $|h|$ is an embedding outside $C_0$ -- the curve with  negative self-intersection -- and contracts $C_0$ to the vertex $P$ of the cone $S$; $$C_0^2=-n, C_0\in |h-nf|, K_{\mathbb{F}_n}=-2h+(n-2)f=-2C_0+(-n-2)f.$$ Let $C\subset S$ be a reduced and non-degenerate curve of degree $d$ and let $\w{C}$ be the strict transformation of $C$ under the desingularization $\widetilde{S}\cong\mathbb{F}_n\rightarrow S$. 
~Setting $k=\w{C}\cdot f$, we have $\w{C}\equiv kh+(d-nk)f$. The adjunction formula gives
\begin{equation}\label{cone}
p_a(\w{C})=1/2\, \left( k-1 \right)  \left( 2\,d-nk-2 \right).
\end{equation}
We  further remark that 
\begin{equation}\label{conevertex}
0\le \w{C}\cdot C_0=\w{C}\cdot (h-nf)=d-nk = m
\end{equation} where $m$ is the multiplicity of $C$ at the vertex  $P$ of the cone $S$.

\end{itemize}
\end{rmk}

\vni Regarding families of curves with high genus -- especially nearly extremal curves -- we quote the following result which we shall use occasionally for determining the irreducibility of a Hilbert scheme; cf. \cite[Corollary 3.16, page 100]{H1}.

\begin{prop}[Harris]\label{nearlyextremal} Suppose $r\neq 3$ or $5$, and $\pi(d,r)\ge g>\pi_1(d,r)$. 
Then $\h{H}_{d,g,r}\neq\emptyset$ if and only if there exist integers $a>0$ and $b$ satisfying 
(\ref{scrolldegree}) and (\ref{scrollgenus}). Moreover, $\h{H}_{d,g,r}$ has exactly one irreducible component for each such pair $(a,b)$. 
\end{prop}

\vni
The following fact should be known to people as a folklore, which the author does not know of any source of a proof in the literature. We provide a proof in Appendix A.

\begin{prop}\label{specialization} Smooth curves in $\PP^{n+1}$  lying on a cone over a rational normal curve in  a hyperplane $H\cong\PP^{n}$ is  a specialization of curves lying on a rational normal surface scroll.
\end{prop}

\begin{rmk}  \label{cones} By Proposition \ref{specialization}, a family of curves in $\PP^r$ lying on  cones over  rational normal curves in
$H\cong\PP^{r-1}$ is in the boundary of the family of curves residing inside smooth surfaces of minimal degree. Consequently, we may assume that a surface $S$ of minimal degree containing a smooth reduced curve in $\PP^r$ is smooth. 
\end{rmk}
\vni
The following lemma -- regarding multiples of the unique pencil $g^1_k$ on a general $k$-gonal curve -- will be used  to show the existence of a complete very ample linear series with a given index of speciality $\alpha$. 
\begin{lem}\cite[Proposition 1.1]{CKM}\label{kveryample} Assume $2k-g-2<0$. Let $C$ be a general $k$-gonal curve of genus $g$, $k\ge 2$, $0\le m$, $n\in\mathbb{Z}$ such that 
\begin{equation}\label{veryamplek}
g\ge 2m+n(k-1)
\end{equation}
 and let $D\in C_m$. Assume that there is no $E\in g^1_k$ with $E\le D.$ Then $\dim|ng^1_k+D|=n$.
\end{lem}
\vni The following inequality -- known as Castelnuovo-Severi inequality -- shall be used occasionally; cf. \cite[Theorem 3.5]{Accola1}.
\begin{rmk}[Castelnuovo-Severi inequality]\label{CS} Let $C$ be a curve of genus $g$ which admits coverings onto curves $E_h$ and $E_q$ with respective genera $h$ and $q$ of degrees $m$ and $n$ such that these two coverings admit no common non-trivial factorization; if $m$ and $n$ are primes this will always be the case. Then
$$g\le mh+nq+(m-1)(n-1).$$ 
\end{rmk}

\medskip
\vni
We fix the following standard notation and convention which we shall use for the existence of linearly normal smooth curves with prescribed degree and genus $(d,g)$ inside a projective space of relatively low dimension. 
\begin{rmk}\label{linearlynormalcriteria}
(a) Let $\PP^2_s$ be the rational surface $\PP^2$ blown up at $s$ general points. Let $e_i$ be the class of the exceptional divisor
$E_i$ and $l$ be the class of a line $L$ in $\PP^2$. For integers  $b_1\ge b_2\ge\cdots\ge b_s$, let $(a;b_1,\cdots, b_i, \cdots,b_s)$ denote the class of the linear system $|aL-\sum b_i E_i|$ on $\PP^2_s$.  By abuse of notation we use the same  expression $(a;b_1,\cdots, b_i, \cdots,b_s)$ for the linear system $|aL-\sum b_i E_i|$ itself or the line bundle $\Oo_{\PP^2_s}(aL-\sum b_i E_i)$. We use the notation  $$(a;b_1^{s_1},\cdots,b_j^{s_j},\cdots,b_t^{s_t}), ~ \sum s_j=s$$ when  $b_j$ appears $s_j$ times consecutively  in the linear system $|aL-\sum b_i E_i|$. 
\vni
(b) We also use the following convention; if the exponent $s_j$  of $b_j$ is zero in the above expression,  then such an entry does not 
  appear in the actual linear system, i.e.  $$(a;b_1^{s_1},\cdots,b_j^{s_j=0},\cdots,b_t^{s_t})=(a;b_1^{s_1},\cdots,\widehat{b_j^{s_j}},\cdots,b_t^{s_t}).$$

\vni

\vni
(c) Assume 
$3\le s\le 6$. Let $C$ be a divisor on $\PP^2_s$ and let $(a;b_1,\cdots ,b_s)$ be the class corresponding to $|\h{O}_S(C)|$. Suppose that $(a;b_1,\cdots ,b_s)$ satisfies the conditions 

(1) $a\ge b_1+b_2+b_3$

(2) $a\ge \textrm{max}(0, b_1)$

(3) $b_1\ge b_2\ge\cdots\ge b_s$.

\vni
Then the followings hold; cf. \cite[Remark 3.1.2, Proposition 3.1.3]{kleppe},\cite[sec. 4 proof of 4.1 and Ex 4.8]{Hartshorne} and \cite[Section 2]{Gruson}.
\vni
\begin{enumerate}
\item[(i)]
 If $C$ is reduced then $b_s\ge -1$. 
 
 \item[(ii)]
 Assume $(a;b_1,\cdots ,b_s)\neq (a;a,0,\cdots , 0)$, $a\ge 2$.  A general member of the complete linear system $|\h{O}_S(C)|$  is smooth and connected if and only if $b_s\ge 0$. 
 
 \item[(iii)]
 Identifying $\PP^2_s$ with the image of the embedding  
  $\PP^2_s\lhook\joinrel\xrightarrow{|-K_{\PP^2_s}|}
 \PP^{9-s}$,
a general member in $|\h{O}_S(C)|$ is linearly normal in $\PP^{9-s}$ if and only if $b_s\ge 1$, provided $$(a;b_1,\cdots ,b_s)\neq
(\lambda +3t,\lambda +t, t, \cdots, t) \textrm{ for some } \lambda \ge 2$$ in which case $h^1(\PP^{9-s},\h{I}_C(1))=0$  if and only if $t\ge 2$.

\item[(iv)] In  Appendix B, we will give a slightly more general criteria covering (iii) for the linear normality of  curves lying on a blown up surface of $\PP^2$,  suitably 
embedded into $\PP^r$  by a certain very ample line bundle.
\end{enumerate}
 
 \vni
 (d)
For the very ampleness of a linear system $|aL-\sum b_i E_i|$,  we frequently use the main result in \cite{sandra}, sometimes without explicit mention. 
\end{rmk}

\noindent
We also make a note of the following lemma
which will be used for the 
non-existence (emptyness) of peculiar Hilbert schemes $\HL{g+r-5,g,r}$ for $g=r+10,~ r\ge 9$ and $g=r+11,~ r\ge12$; Theorem \ref{sub6x} (c),(d).
\begin{lem}\label{easylemma1} Let $\h{E}=g^{\al -1}_e$ ($\al\ge 4$) be a special complete linear series on a smooth curve $C$ of genus $g\ge 5$, with (possibly empty) base locus $\Delta$. Let $\h{E}':=\h{E}-\Delta=g^{\al -1}_{e'}$ be the moving part of $\h{E}$. 
We assume that  either 
\begin{itemize}
\item [(i)]
$\h{E}'$ induces a double covering $C\stackrel{\eta}{\rightarrow} E$ onto a possibly singular curve $E\subset\PP^{\al -1}$ of (geometric) genus $h\ge 1$. Let  $\w{E}\stackrel{\epsilon}{\rightarrow} E\subset\PP^{\al -1}$ be the normalization of the curve $E$. Let 
$C\stackrel{\w{\eta}}{\rightarrow} \w{E}$ be the morphism associated with $\eta$, i.e.  $\eta=\epsilon\circ\w{\eta}$; set
$$\epsilon^*(\h{O}_{\PP^{\al -1}}(1))=g^{\al -1}_{e'/2}, ~~|\eta^*(\h{O}_{\PP^{\al -1}}(1))|=|\w{\eta}^*\epsilon^*(\h{O}_{\PP^{\al -1}}(1))|=\h{E'}.$$ Suppose that the linear series $g^{\al -1}_{e'/2}$ on $\w{E}$ which is pulled back to $\h{E}'$ via $\w{\eta}$, i.e.  $\w{\eta}^*(g^{\al -1}_{e'/2})=\h{E'}$ -- which is necessarily complete by the completeness of $\h{E}'$ -- is {\bf non special} or


\item [(ii)] $\Delta\neq\emptyset$ and $\h{E}$ induces a triple covering onto a rational curve.
\end{itemize}
Then $\h{E}^\vee$ is not very ample. 
\end{lem}
\begin{proof}
(i)  For any $u\in \w{E}$, 
\begin{eqnarray*}
 \dim|\h{E}+\w{\eta}^*(u)|&=&\dim|\h{E}'+\Delta +\w{\eta}^*(u)|=\dim|\w{\eta}^*(g^{\al -1}_{e'/2}+u)+\Delta|\\&=&\dim|\w{\eta}^*(g^\al_{e'/2+1})+\Delta|\ge \dim\h{E}+1
\end{eqnarray*} and hence $\h{E}^\vee$ is not very ample.
\vni
(ii)) Let $g^1_3$ be the (unique) trigonal pencil on $C$. Choose $q\in\Delta$ and let  $q+t+s\in g^1_3$ be the unique trigonal divisor containing $q$. We then have 
\begin{eqnarray*}
\dim|\h{E}+t+s|&=&\dim|(\al -1)g^1_3+\Delta+t+s|\\&=&\dim| (\al -1)g^1_3+(q+t+s)+(\Delta-q)|\\&=&\dim|\al g^1_3+(\Delta -q)|\ge\dim\h{E}+1,
\end{eqnarray*}
hence $\h{E}^\vee$ is not very ample. 
\end{proof}
\vni 
On the contrary, the following lemma will be used for the existence of $\HL{g+r-5,g,r}$ with $g=r+12$ when  $r$ is big enough. The proof is lengthy which we provide in Appendix  C.

\begin{lem}\label{triple0}
 Let $C\stackrel{\eta}{\rightarrow} E$ be a triple covering of an elliptic curve $E$ of genus $g\ge 3\al+7$, $\al\ge3$. 
 For a complete $g^{\al -1}_\al\in W^{\al -1}_\al(E)$, 
 we put $\h{E}:=|\eta^*(g^{\al -1}_\al )|$.
 Then 
 $\dim\h{E}=\al -1$ and $\h{E}^\vee =|K_C-\h{E}|=g^{g-2\al -2}_{2g-2-3\al }$ is very ample.
\end{lem}

 
 \vni 
 We close this section with an immediate consequence of 
 Lemma \ref{triple0}. 
\begin{cor}\label{triple1} For $\al\ge 3$, $r\ge\al+5$ 
$$\HL{2r+\al +2,r+2\al +2,r}\neq\emptyset.$$
\end{cor}
\begin{proof} We take the linear series $\h{E}=g^{\al-1}_{3\al}=\eta^*(g^{\al-1}_\al)$ on a triple covering of an elliptic curve 
$C\stackrel{\eta}{\rightarrow} E$. Since $g=r+2\al +2\ge 3\al+7$, the complete linear series $\h{E}^\vee=g^r_{2r+\al+2}$ is very ample by Lemma \ref{triple0}. 
\end{proof}

\section{Existence \& non-existence of linearly normal curves with the index of speciality $\al =5$}

This section is devoted to determining the exact range of the triples $(d,g,r)$ for which 
$\HL{d,g,r}\neq\emptyset$ with the index of speciality $\al=g-d+r=5, r\ge3$.

\vni
Before proving the first main result of this section, we make a remark concerning the existence of linearly normal smooth curves in $\PP^r$ which directly comes from well-known results in Brill-Noether theory.
\vni
\begin{rmk}\label{general} (i)
In the Brill-Noether range $$\rho (d,g,r)=\rho (d,g,r)=g-(r+1)(g-d+r)\ge 0$$ with $r\ge 3$, the existence of a smooth linearly normal curve in $\PP^r$ of degree $d$ and genus $g$ follows from a theorem due to Eisenbud-Harris \cite[Theorem (1.8)]{H1}; a general element of $W^r_d(C)$ on a general curve $C$ of genus $g$ is very ample, which is also complete by the fact  that no component of $W^r_d(C)$ is entirely contained in $W^{r+1}_d(C)$ (for any $C$) \cite[Lemma (3.5), Chapt. IV]{ACGH}. Hence we see that 
$$\HL{d,g,r}\neq\emptyset \textrm{ if }\rho(d,g,r)\ge 0  \textrm{ and } r\ge 3.$$
\vni
(ii) 
Therefore we may focus on triples $(d,g,r)$ which are outside the Brill-Noether range.
 In our current situation $g-d+r=5$, we are particularly interested in the range
 $$\rho (d,g,r)=g-5(r+1)<0.$$

\end{rmk}
\vni
 The following theorem provides the full list of triples  $(d,g,r)$ with 
 $g-d+r=5$ and $r\ge 3$ for which $\HL{g+r-5,g,r}\neq\emptyset$; cf. Figure \ref{fig:existence}.



\begin{thm}\label{sub6x}
\begin{enumerate}
\item[(a)] 

$
\HL{g+r-5,g,r}=\h{H}_{g+r-5,g,r}=\emptyset\textrm{ for } g\le r+8, r\neq 4
$.


\vni
For $r=4$, 
$\HL{g+r-5,g,r}=\HL{g-1,g,4}=\emptyset$  if $g\le r+7=11$ and 

$\HL{g-1,g,4}\neq\emptyset$ for 
$g=r+8=12$.
\item[(b)] $\HL{g+r-5,g,r}\neq\emptyset$
for  $g=r+9$, $r\ge 3$.

\item[(c)] $\HL{g+r-5,g,r}=\emptyset$ for $g=r+10$ if and only if $r\ge 9$.
\item[(d)] $\HL{g+r-5,g,r}=\emptyset$ for $g=r+11$  if and only if $r\ge 12$. 
\item[(e)] $\HL{g+r-5,g, r}\neq\emptyset$ for $g\ge r+12$, $r\ge 3$. 
\end{enumerate}
\end{thm}
\begin{proof} 
\begin{enumerate}
\item[(a-i)]
$r=3$: If $g\le 10$ one has $\pi (g-2,3)\lneq g$ which is an absurdity. If $g=11$ we have $\pi_1(g-2,3)=10<g\le\pi(9,3)=12$ and hence $C$ is a nearly extremal curve lying on a quadric surface in $\PP^3$. However there is no pair $(a,b)\in\mathbb{N}\times\mathbb{N}$ satisfying 
$g=11=(a-1)(b-1)$ and $d=g-2=9=a+b$. 
\item[(a-ii)] $r=4$:  If $g\le r+7=11$,  
$g\le \pi(g-1,4)\le g-2$  a contradiction. A non-degenerate irreducible curve of degree $d=11$ in $\PP^4$ has maximal arithmetic genus $\pi(11,4)=12=r+8$ and such an extremal curve exists in the linear system $|4H-L|$ on a cubic surface scroll in $\PP^4$; cf. \cite[III, Theorem 2.5]{ACGH}. 
\item[(a-iii)] $r=5$: For $g\le 13=r+8$, one has $\pi(g,5)\le g-1$, a contradiction.
\item[(a-iv)] $r\ge 6$:
One computes $m:=[\frac{g+r-6}{r-1}]\le2$ if $r\ge 6$ and $g\le r+8$. In case $m=2$, one has $$g\le \pi (g+r-5,r)=\frac{m(m-1)}{2}(r-1)+m(g+r-6-2(r-1))=2g-r-9$$ which  is not compatible with the assumption $g\le r+8$.  Same contradiction occurs when $m=1$.
  
\end{enumerate}
\begin{enumerate}
\item[(b-i)] $r=3$:  For $g=r+9=12$, there exists a smooth curve of degree $d=10$ of type $(3,7)$ on a smooth quadric in $\PP^3$. There also exists a smooth curve in $\PP^3$ with $(d,g)=(10,12)$ on a 
non-singular cubic surface;  a proper transformation of a plane curve of degree $9$ with $5$ ordinary triple points and a node on  $\PP^2$ blown up at the $6$ assigned singular points; $C\in (9;3^5,2)$ on $\PP^2_6$. 
\item[(b-ii)] $r=4$: On a general trigonal curve of genus $g=13$, $|K-4g^1_3|=g^4_{12}$ is very ample; we may take $m=2, k=3, n=4$ in Lemma \ref{kveryample}.  Such a curve exists on a smooth cubic surface scroll in $\PP^4$ in the linear system $|3H+3L|$. There is another irreducible family of curves lying on a del Pezzo  surface  $\PP^2_5\subset\PP^4$; $C\in (9;3^5)$. Both curves are linearly normal; if not, $C$ is an image of a  projection from a curve $\tilde{C}\subset\PP^5$ of the same degree $d=12$, contradicting the Castelnuovo's genus bound $\pi(12,5)=10>g=r+9=13$. In the next section we will see that these two types of curves form the only two components of $\HL{12,13,4}$; cf. Proposition \ref{g=r+9} (v).


\item[(b-iii)] $r=5$: Similarly there exists a smooth linearly normal curve of degree $d=g=14$ on a smooth quartic surface scroll $S\subset\PP^5$. We may take $C\in |3H+2L|$ which is trigonal.
Note that by adjunction, $$g^5_{14}=|K_C-4g^1_3|=|K_S+C-4L|_{|C}=|H|_{|C},$$ which is complete and hence $C$ is linearly normal. Alternatively, since $\pi (14,6)=11<g$, every smooth curve $C\subset\PP^5$ with $(d,g)=(14,14)$ is linearly normal. 

\item[(b-iv)] $r\ge 6$: For $g=r+9$ we have $\pi (g+r-5,r)=r+9$ and the curve corresponding to a  general  element of a component of $\h{H}_{2r+4,r+9,r}$ is an extremal curve. An extremal curve is always linearly normal and hence  $$\HL{2r+4,r+9,r}=\h{H}_{2r+4,r+9,r}\neq\emptyset.$$
\end{enumerate}
\vni
\vni
(c) We first show that $\HL{g+r-5, g,r}=\emptyset$ when $g=r+10$ and $r\ge 9$. Let $\h{D}\in\h{G}$ be a complete very ample $g^r_{g+r-5}$ and set $\h{E}:=\h{D}^\vee=g^4_{13}$. Since $$\pi (13,4)=18<19\le r+10=g,$$ 
$\h{E}$ is compounded with non-empty base locus $\Delta$, $\deg\Delta=\delta\ge 1$  inducing a $k$-sheeted map $\eta: C\rightarrow E$ onto a possibly singular non-degenerate curve $E\subset\PP^4$, $\deg E=f$. The following triples $(k,f,\delta )$ are possible subject to the conditions $$k\ge 2, ~\deg E=f\ge \dim\PP^4, ~k\cdot f+\delta=13.$$

\begin{enumerate}
\item[(1)]
$(k,f,\delta )=(2,6,1)$; $C\stackrel{\eta}\rightarrow E$ is a double cover of a curve $E$ of genus $2$ and $\h{E}=\eta^*(g^4_6)+\Delta$.

\item[(2)]
$(k,f,\delta ) =(2,5, 3)$; $C\stackrel{\eta}\rightarrow E$ is a double cover of a curve $E$ of genus $1$ and $\h{E}=\eta^*(g^4_5)+\Delta$.

\item[(3)]
$(k,f,\delta ) =(2,4, 5)$; $C$ is hyperelliptic.

\item[(4)]
$(k,f,\delta ) =(3,4, 1)$; $C$ is trigonal and $\h{E}=4g^1_3+\Delta$. 
\end{enumerate}

\vni
Note that linear series $\w{\h{E}}$ on $E$ (or on the normalization of $E$) such that $\eta^*(\w{\h{E}})+\Delta=\h{E}$ is complete  since $\h{E}$ is complete. $\w{\h{E}}$ is non-special by Clifford's theorem.

\vni In the cases (1), (2), (4), the residual series $\h{D}=\h{E}^\vee$ is not very ample by Lemma \ref{easylemma1}; we may take $\al =5, e=13$ in Lemma \ref{easylemma1}. In the case (3), we recall that a hyperelliptic curve does not carry a special very ample linear series. Hence 
$\HL{g+r-5,r+10,r}=\emptyset$ if $r\ge 9$.

\vni
\vni
We next show $$\HL{2r+5.r+10.r}=\h{H}_{2r+5,r+10,r}\neq\emptyset ~ \textrm{ for } 3\le r\le 8.$$
We consider $\PP^2_{9-r}\subset\PP^r$ which may be realized as a  smooth surface in $\PP^r$ embedded by the anti canonical linear system $|-K_{\PP^2_{9-r}}|=(3,1^{9-r})$.  Set
\begin{equation}\label{octic}\h{L}(\beta,r):=(8;3,2^{8-r-\beta},1^{\beta}); ~~0\le\beta\le 8-r.
\end{equation}
Let $C\in\h{L}(\beta,r)$ be a general member.
 Since $\h{L}(\beta,r)$ is very ample by \cite[Theorem]{sandra} or \cite[V, 4.13]{Hartshorne}, $C\subset \PP^r$ is smooth of degree $d$ and genus $g$ where
\begin{eqnarray}
&d=\h{L}(\beta,r)\cdot (-K_{\PP^2_{9-r}})=(8;3,2^{8-r-\beta},1^{\beta})\cdot (3;1^{9-r})=2r+5+\beta, \label{degree}
\\
&g=\frac{(8-1)(8-2)}{2}-3-(8-r-\beta)=r+10+\beta.\label{genus}
\end{eqnarray}
\vni
We take $\beta=0$, then 
\begin{equation*}\h{L}(0,r)=(8;3,2^{8-r},1^0)=(8;3,2^{8-r}),
\end{equation*}
according to the convention we employed in Remark  \ref{linearlynormalcriteria} (b).

\vni
Note that a smooth curve
of degree $d=2r+5$ and genus $g=r+10$ in $\PP^r$ is linearly normal since 
$$\pi (2r+5,r+1) < g=r+10;$$
i.e. if it is not linearly normal then one may embed the curve in a higher dimensional projective space $\PP^{r+1}$ and up there $g=r+10$ is too big  for the genus of a smooth curve of degree $d=2r+5$.  Alternatively, one may 
use Remark \ref{linearlynormalcriteria} (c)(iii). 

\vni
(d) If $g=r+11$ and $r\ge 12$ -- as in the first half of the previous case (c) ($g=r+10$ and $r\ge 9$) -- 
we have $$g=r+11\ge 23>\pi (14,4)=22$$ and hence $\h{E}$ is compounded with (possibly empty) base locus $\Delta$
inducing a $k^{\ge 2}$ -- sheeted map onto a curve of degree $f$ in $\PP^4$. 
Hence Lemma \ref{easylemma1} applies and we may  conclude $\HL{g+r-5. r+11.r}=\emptyset$ for $r\ge 12$. We omit the details. 







\vni
In what follows, we show $\HL{g+r-5,g,r}=\HL{2r+6, r+11,r}\neq\emptyset$ for $3\le r\le 11$.

\vni
\begin{enumerate}
\item[(d-i)]
We first deal with the cases $3\le r\le 7$ using the same surface $\PP^2_{9-r}\subset\PP^r$ embedded by $|-K_{\PP^2_{9-r}}|$. We take $\beta=1$ in  the linear system (\ref{octic}) $$\h{L}(\beta,r)=(8;3,2^{8-r-\beta},1^{\beta})$$ and obtain 
$$\h{H}_{2r+6,r+11,r}\neq\emptyset ~\mathrm{if }  ~3\le r\le 7. $$ Note that the obvious condition $8-r-\beta\ge 0$ on one of the exponent in $\mathcal{L}(\beta,r)$ 
forces $r\le 7$ if $\beta =1$.
\vni
If $5\le r\le 7$, a smooth curve
of degree $d=2r+6$ and genus $g=r+11$ in $\PP^r$ is linearly normal since 
 $$g=r+11 > \pi (2r+6,r+1)$$
 and hence 
 $$\h{H}_{2r+6,r+11,r}=\HL{2r+6,r+11,r}\neq\emptyset ~\mathrm{ if}~ 5\le r\le 7.$$  For $r=3,4$ we have $$\HL{2r+6,r+11,r}\neq\emptyset$$ by Remark \ref{linearlynormalcriteria} (c)(iii). 
 
 \vni 
\vni

\vni
\item[(d-ii)] For the cases $8\le r\le 11$, we take the non-singular surface $S$ as
 $$\PP^2_{12-r}\lhook\joinrel\xrightarrow{|(4;2,1^{11-r})|} S\subset
 \PP^r,$$
embedded by the very ample\footnote{ $(4;2,1^{11-r})$ is very ample even for $4\le r\le 11.$} 
 linear system $(4;2,1^{11-r})$; cf.  \cite[Theorem]{sandra}. We consider a linear system
 \begin{equation}\label{nonic}
\h{M}(\gamma,r)= (9;4,2^{11-r-\gamma}, 1^{\gamma}), ~~0\le\gamma\le 11-r
 \end{equation}
 and 
let $C\in\h{M}(\gamma,r)$ be a general member.
 Since $\h{M}(\gamma,r)$ is very ample by \cite[Theorem]{sandra}, $C$ is smooth of degree $d$ and genus $g$ in $\PP^r$ where
\begin{eqnarray}
d=\h{M}(\gamma,r)\cdot (4;2,1^{11-r})=
2r+6+\gamma, \label{degree9}
\\
g=\frac{(9-1)(9-2)}{2}-\frac{4\cdot3}{2}-(11-r-\gamma)=r+11+\gamma.\label{genus9}
\end{eqnarray}
\vni
For the particular value $\gamma=0$, we have 
\begin{equation*}\h{M}(0,r)=(9;4,2^{11-r}).
\end{equation*}
~Note that $$g=r+11 > \pi (2r+6,r+1),$$ holds\footnote{This strict inequality holds for $5\le r\le 11$.} for $8\le r\le 11$ and hence a general member in $\h{M}(0,r)$ is linearly normal 
of degree $d=2r+6$ and genus $g=r+11$ in $\PP^r$.
Therefore we may claim 
$$\HL{2r+6,r+11,r}\neq\emptyset \textrm{ for } 8\le r\le 11.$$

\end{enumerate}
\vni

\vni
(e) In order to avoid unnecessarily complicated computation we first show 

(e-i) $\HL{g+r-5,g, r}\neq\emptyset$ for $g=r+12$, $r\ge 3$ and then proceed to show 

(e-ii) $\HL{g+r-5,g, r}\neq\emptyset$ for $g\ge r+13$, $r\ge 3$.

\begin{enumerate}
\item[(e-i-1)] $g=r+12$, $3\le r\le 6$: We take $\beta =2$ in the linear system (\ref{octic}), 
$$\h{L}(2,r)=(8;3,2^{6-r},1^2)$$ containing a smooth curve of degree $d=2r+7$ and genus $g=r+12$ by
(\ref{degree}), (\ref{genus}).
The linear normality follows from Remark \ref{linearlynormalcriteria} (c).

\item[(e-i-2)] $g=r+12, \, 7\le r\le 10$:  We take the linear system $\h{M}(\gamma,r)$ in (\ref{nonic}) on the surface 
$\PP^2_{12-r}\subset\PP^r$ and put $\gamma =1$. 
A general member of the very ample linear system 
 \begin{equation*}
\h{M}(1,r)=(9;4, 2^{10-r},1)
 \end{equation*}
\vni 
 is smooth of degree $d=2r+7$ and genus $g=r+12$ in $\PP^r$ by (\ref{degree9}), (\ref{genus9}). Assume that a general $C\in\h{M}(1,r)$ is not linearly normal. Let $\w{C}\subset\PP^{r+1}$ be a curve of degree $d=2r+7$ which is projected onto $C\subset\PP^r$ from a point outside $\w{C}$.
Note that  $$g=r+12 =\pi (2r+7,r+1)$$ for $7\le r\le 10$ and  hence $\w{C}\subset\PP^{r+1}$ is an extremal curve lying on surface of minimal degree in
$\PP^{r+1}$. In case $\w{C}$ lies on a smooth scroll, $$\w{C}\in |(m+1)H-(r+1-\epsilon-2)L|=|3H-(r-7)L|$$ where
$m:=[\frac{d-1}{r}]=2, \epsilon =d-1-mr$; cf. \cite[Theorem 2.5, p122]{ACGH}. Hence $C\cong\w{C}$ is trigonal with the trigonal pencil cut out by the rulings of the scroll. 
On the other hand, a general $C\in\h{M}(1,r)$ has a base-point-free $g^1_5$ cut out by lines through the $4$-fold point on the plane model of $C$. By the Castelnuovo-Severi inequality, 
$g\le (3-1)\cdot (5-1)=8$, a contradiction.  If $\w{C}\subset\PP^{r+1}$ lies on a rational normal cone, we put $n=r$, $d=2r+7$, $p_a(\w{C})=r+12$ in Remark \ref{minimal} (iii) (in
the genus formula \eqref{cone}) to get $k=3$ hence $\w{C}$ is trigonal leading to the same contradiction.

\item[(e-i-3)]  
$g=r+12, r\ge 10$;  we take $\al =5$ in Corollary \ref{triple1} and may conclude $\HL{2r+7,r+12,r}\neq\emptyset$.
\vni
\vni
\item[(e-ii)] $g\ge r+13$: We are primarily interested in the existence of linearly normal curves in $\PP^r$ with index of speciality $\al=5$. However we will show the existence of linearly normal curves in $\PP^r$ with any given index of speciality $\al \ge 2$ under an additional condition $r\ge \al+1$ together with a mild restriction on the genus $g$. We use Lemma \ref{kveryample} regarding the dimension of the linear series which is a multiple of $g^1_k$ on  general $k$-gonal curves. Our assertion
$$\HL{g+r-5,g, r}\neq\emptyset \textrm {~for~ } g\ge r+13$$
is an immediate consequence of the following claim when $r\ge 6$.

\item[(e-ii-1)]
\vni
{\bf Claim (3.2.A): Fix an  integer $\alpha \ge 3$. For every  $r\ge\al+1$ and $g\ge r+3\al-2$, $$\HL{g+r-\al,g,r}\neq\emptyset.$$ } 
\vni
We set $$e:=g-r+\al -2=\deg |K_C-\h{D}|;  ~\h{D}\in\h{G}\subset\widetilde{\h{G}}_\h{L}\subset\h{G}^r_{g+r-\al}$$ and
$$e\equiv \sigma ~(\textrm{mod}~\al-1),~~ 0\le \sigma\le \al-2.$$

\vni
By the assumption $g\ge r+3\al -2$, we have
\begin{equation}\label{alpha}
e=g-r+\al -2=(\al -1)k+\sigma\ge 4(\al -1)
\end{equation}
 for some $k\ge 4$.  Since $r\ge\al+1$,  we have 
 \begin{align*}
\rho (k,g,1)&=2k-g-2=-(\al-3)k-r-\sigma+\al-4\\
&\le-(\al-3)k-(\al+1)-\sigma+\al-4\\
&=-(\al-3)k-\sigma-5\lneqq 0
\end{align*}
and hence $\h{M}^1_{g,k}\subsetneq\h{M}_g$.

\vni
 Let  $C\in\h{M}^1_{g,k}$ be a general $k$-gonal curve.
 We choose $q_1+\cdots +q_{\sigma}\in C_{\sigma}$ such that no pair $(q_i,q_j)$, $q_i\neq q_j$ are in the same fibre of the $k$-sheeted map onto $\PP^1$ defined by the unique $g^1_k$ on $C$.
For an arbitrary choice $t+s\in C_2$, we take $$D=q_1+\cdots +q_\sigma+t+s\in C_{\sigma +2}, m=\sigma +2 \textrm{ and } n=\al -1$$ in Lemma \ref{kveryample}. By the choice of $q_1+\cdots +q_\sigma\in C_\sigma$,  there is no $E\in g^1_k$ such that $E\le D$ . Furthermore the numerical condition (\ref{veryamplek}) in Lemma \ref{kveryample} is satisfied; via direct computation together with the condition\footnote{The numerical condition (6) in Lemma \ref{kveryample} holds when $r\ge \sigma +3$ which automatically holds if $r\ge\al +1$.} $r\ge \al+1$. Hence $$\dim |(\al -1)g^1_k+D|=\al -1$$ by Lemma \ref{kveryample}, which implies  $$|K_C-(\al -1)g^1_k-(q_1+\cdots +q_\sigma)|=g^{r}_{2g-2-k(\al -1)-\sigma}=g^r_{g+r-\al}$$ is very ample (and complete),  finishing the proof of the claim. 
\vni
\item[(e-ii-2)] $3\le r\le 5$: If $r=3$, the existence of smooth linearly normal curves of degree $d=g+r-5=g-2$, $g\ge r+13=16$ in $\PP^3$ follows from works of Gruson-Peskine \cite{Gruson} and \cite{DP}. The existence of smooth linearly normal curves of degree $d=g+r-5=g-1$ ($d=g$ resp.) with $g\ge r+13=17$ ($g\ge 18$ resp.) in $\PP^4$ ($\PP^5$ resp.) follows from the work of Rathman \cite{rathman}.
 However, we will show the existence and the linear normality of such curves 
through explicit examples we have seen in this proof. 

\vni By Remark \ref{general} (ii) we only need to check the existence of linearly normal curves outside the Brill-Noether range; $\rho (g+r-5, g,r)<0$, $3\le r\le 5$, i.e.
\[
\begin{cases} r+13=16\le g < (r+1)(g-d+r)=20; &r=3\\
r+13=17\le g<(r+1)(g-d+r)=25; &r=4\\
r+13=18\le g<(r+1)(g-d+r)=30;&r=5.
\end{cases}
\]

\vni (i) $r=3$: The very ample linear system  (\ref{octic})
$$\h{L}(\beta,r)=(8;3,2^{8-r-\beta},1^{\beta})$$ with $r=3$ and $3\le\beta\le 5$ on $\PP^2_{9-r}=\PP^2_6\lhook\joinrel\xrightarrow{|(3;1^6|}
 \PP^3$ covers $16\le g=r+10+\beta\le 18$; cf. \eqref{degree}, \eqref{genus}. For the case $g=19$, we use
the very ample linear system $(10;6,2^2,1^3)$ on $\PP^2_{9-r}=\PP^2_6\lhook\joinrel\xrightarrow{|(3;1^6|}
 \PP^3$. Linear normality follows from Remark \ref{linearlynormalcriteria} (c).

\vni (ii) $r=4$: Let $S\subset\PP^4$ be a Bordiga surface in $\PP^4$, which is the rational surface $\PP^2_{10}$ blown up at $10$ general points and embedded into $\PP^4$ by the linear system $(4;1^{10})$. 
Consider $$\h{B}(\delta)=(9;3,2^\delta,1^{9-\delta})=(4;1^{10})+(5;2,1^\delta, 0^{9-\delta}); 1\le\delta\le8, $$
which is a sum of a very ample and a base-point-free linear system. Hence $\h{B}(\delta)$ is very ample whose general member is an irreducible, non-singular curve $C$ of degree $d$ and genus $g$ where
$$d=(9;3,2^\delta,1^{9-\delta})\cdot(4;1^{10})=24-\delta,$$
$$g=\frac{(9-1)(9-2)}{2}-\frac{3\cdot 2}{2}-\delta=25-\delta, ~~1\le\delta\le 8.$$
A general $C\in (9;3,2^\delta,1^{9-\delta})$ is linearly normal by Corollary \ref{appendixb}.
\vni
\vni
(iii) 
$r=5$: We may find a smooth curve of degree $d=g$ in $\PP^5$ on a Bordiga surface $\PP^2_9\stackrel{(4;1^9)}{\hookrightarrow}\PP^5$ as follows.
\begin{itemize} 
\item[{(1)}] $18\le g=d\le 21$: Take $$C\in (10;3^5,2^{\delta -1},1^{5-\delta}), ~1\le\delta\le4$$ on $\PP^2_9$. Note that 
\begin{eqnarray*}
(10;3^5,2^{\delta -1},1^{5-\delta})&=&(4;1^9)+(6;2^5,1^{\delta -1},0^{5-\delta})\\&=&(4;1^9)+(3;1^5,0^4)+(3;1^5,1^{\delta -1}, 0^{5-\delta}),
\end{eqnarray*}
where $(4;1^9)$ is very ample,  both $(3;1^5,0^4)$ and $(3;1^5,1^{\delta -1}, 0^{5-\delta})$ are base-point-free hence $(10;3^5,2^{\delta -1},1^{5-\delta})$ is very ample containing an irreducible and non-singular curve $C$ as a general member by Bertini. 
One computes
\begin{eqnarray*}p_a(C)&=&\frac{1}{2}(C+K_{\PP^2_9})\cdot C+1\\
&=&\frac{1}{2}((10;3^5,2^{\delta -1},1^{5-\delta})+(-3; -1^9))\cdot (10;3^5,2,^{\delta -1},1^{5-\delta})+1\\&=&
\frac{1}{2}(7;2^5,1^{\delta -1},0^{5-\delta})\cdot (10;3^5,2^{\delta -1},1^{5-\delta})+1=22-\delta\end{eqnarray*}
and $$d=(10;3^5,2^{\delta -1},1^{5-\delta})\cdot (4;1^9)=22-\delta.$$ 
 $C$ is linearly normal  by Corollary \ref{appendixb}.
\item[{(2)}] $22\le g=d\le25$: Take 
$$C\in (12;6,4,3^3,2^{4-\delta},1^\delta), 1\le \delta\le 4.$$ By the same reason as the previous one, 
\begin{align*}(12;6,4,3^3,2^{4-\delta},1^\delta)&=(4;1^9)+(8;5,3,2^3,1^{4-\delta},0^{\delta})\\&=(4;1^9)+(5;3,2,1^{7-\delta}, 0^\delta)+(3;2,1^4,0^4)
\end{align*}
 is very ample; $(5;3,2,1^{7-\delta})$ ($(3;2,1^4)$ resp.) is base-point-free on $\PP^2_{9-\delta}$ ($\PP^2_5$ resp.) by \cite{sandra}. Hence $(5;3,2,1^{7-\delta},0^\delta)$ ($(3;2,1^4,0^4)$ resp.) is base-point-free on $\PP^2_{9}$. An easy calculation  yields
 $$d=g=21+\delta, ~1\le\delta\le 4.$$
 
\item[{(3)}] $26\le g=d\le 29$: Taking 
$$C\in (14;8,4^3,3,2^{4-\delta},1^\delta),$$ we have  
\begin{align*}(14;8,4^3,3,2^{4-\delta},1^\delta)&=(4;1^9)+(10;7,3^3,2,1^{4-\delta},0^\delta)\\&=(4;1^9)+(6;4,2^3,1^{5-\delta},0^\delta)+(4;3,1^4,0^4).
\end{align*}
Likewise, the second and third summands are base-point-free on $\PP^2_9$ and hence 
$(14;8,4^3,3,2^{4-\delta},1^\delta)$ is very ample whose general member $C$ has genus $g=25+\delta ~(1\le\delta\le 4)$ and 
degree $d=g$ in $\PP^5$.
\end{itemize}
\end{enumerate} \end{proof}

\vni
Before closing the section, we would like to make the following remark which is worthy of being mentioned.

\begin{rmk} 
\begin{itemize}
\item[(i)]
The peculiar and seemingly artificial linear systems  $\Ll(\beta,r)$ and $\Mm(\gamma,r)$  in (\ref{octic}), (\ref{nonic}) on del Pezzo  surfaces which appear in the proof of Theorem \ref{sub6x} arise naturally by analyzing nearly extremal curves induced by $\h{E}=\h{D}^\vee$ corresponding to a  general element  $\h{D}\in\h{G}\subset\h{G}^r_{g+r-5}$. More systematic treatment appears in the course of the proof of Theorem \ref{r+10} when we check the irreducibility of $\HL{2r+5,r+10,r}$.
\item[(ii)] Hilbert schemes $\HL{2r+10,r+5,r}$ and $\HL{2r+11,r+6,r}$ are non-empty only for low $r$.  For $\al=3, 4$, there are also such Hilbert schemes $\HL{g+r-\al,g,r}$; cf. \cite[Proposition 3.3]{lengthy}, \cite[Theorem 2.4]{speciality}. However this phenomenon is not so unusual under our setting, i.e. studying smooth projective curves in connection with the index of speciality.  We will see this phenomenon in a more general context in the final section; cf. Proposition \ref{r+2a+1}.
\item[(iii)] The following is an illustration of the result we obtained in Theorem \ref{sub6x}.
\end{itemize}
\end{rmk}


\pagebreak
\begin{figure}
\begin{tikzpicture}
\begin{axis}
[grid=both, transpose legend,
legend columns=8,
legend style={at={(0.5,-0.2)},anchor=north,font=\footnotesize},
xlabel=Dimension of the projective space $r$, ylabel=Genus $g$,
ymin=3,ymax=49,enlargelimits=false]

\addplot+  [
sharp plot
] coordinates {
(3,12) (4,12) (4,13) (5,14) (6,15) (6,15) (7,16) (8,17) (9,18) (10,19) (11,20) (12,21) (13,22)
}; \addlegendentry{$\{(r,g) | g=r+9, r\ge 3\}\cup\{(4,12)\}\subset\mathbb{N}\times\mathbb{N}$,
$\mathcal{H}^\mathcal{L}_{d,g,r}\neq\emptyset$}

\addplot+ [
sharp plot,
] coordinates {
(3,11)  (5,13) (6,14) (7,15) (8,16) (9,17) (10,18) (11,19) (12,20) (13,21)
}; \addlegendentry{$\{(r,g) | g=r+8, r\ge 3\}\setminus\{(4,12)\}\subset\mathbb{N}\times\mathbb{N}$,
$\mathcal{H}^\mathcal{L}_{d,g,r}=\emptyset$

}
\addplot+ [
sharp plot,
] coordinates {
(3,13)  (4,14) (5,15) (6,16) (7,17) (8,18)
}; \addlegendentry{$\{(r,g) | g=r+10, r\ge 3\}$,
$\mathcal{H}^\mathcal{L}_{d,g,r}\neq\emptyset \Longleftrightarrow 3\le r\le 8$
}

\addplot+ [
sharp plot,
] coordinates {
(3,14)  (4,15) (5,16) (6,17) (7,18) (8,19) (9,20) (10,21) (11,22)
}; \addlegendentry{
$\{(r,g) | g=r+11, r\ge 3\}$,
$\mathcal{H}^\mathcal{L}_{d,g,r}\neq\emptyset \Longleftrightarrow 3\le r\le 11$
}

\addplot+  [
sharp plot,
] coordinates {
(3,15)  (4,16) (5,17) (6,18) (7,19) (8,20) (9,21) (10,22) (11,23) (12,24) (13,25)
};
\addlegendentry{$\{(r,g) | g=r+12, r\ge 3\}$,
$\mathcal{H}^\mathcal{L}_{d,g,r}\neq\emptyset$

}


\addplot [red!80!black,fill=red,fill
opacity=0.5,
] coordinates {
(3,11) (4,12) (5,13) (6,14)
(7,15) (8,16) (9,17) (10,18)
(11,19) (12,20) (13,21)
}
|- (0,0) -- cycle; \addlegendentry{Pink Zone, $\mathcal{H}^\mathcal{L}_{d,g,r}=\emptyset$}

\addplot [green!20!black] coordinates {
(3,20) (5,30) (13,70)
};

\addplot [red,name path=A, smooth] table {
x y
3 20
13 70
};
\addplot [blue,name path=B, smooth] table {
x y
3 49
13 49
};
\addplot [green!100] fill between [
of=A and B,
soft clip={domain=1:1},
];\addlegendentry{Blue Zone,  Off Brill-Noether  range, $\mathcal{H}^\mathcal{L}_{d,g,r}\neq\emptyset$}
\addplot [blue,name path=C, smooth] 
table 
{x y
3 20
13 70
};
\addplot [blue,name path=D, smooth] 
table {
x y
3 15
13 25
};
\addplot [blue!60] fill between [
of=C and D,
soft clip={domain=1:1},
];\addlegendentry{Green Zone,  Brill-Noether range, $\mathcal{H}^\mathcal{L}_{d,g,r}\neq\emptyset$}
\end{axis}
\end{tikzpicture}
\caption{Existence and non-existence; $\alpha=g-d+r=5$}
\label{fig:existence}
\end{figure}
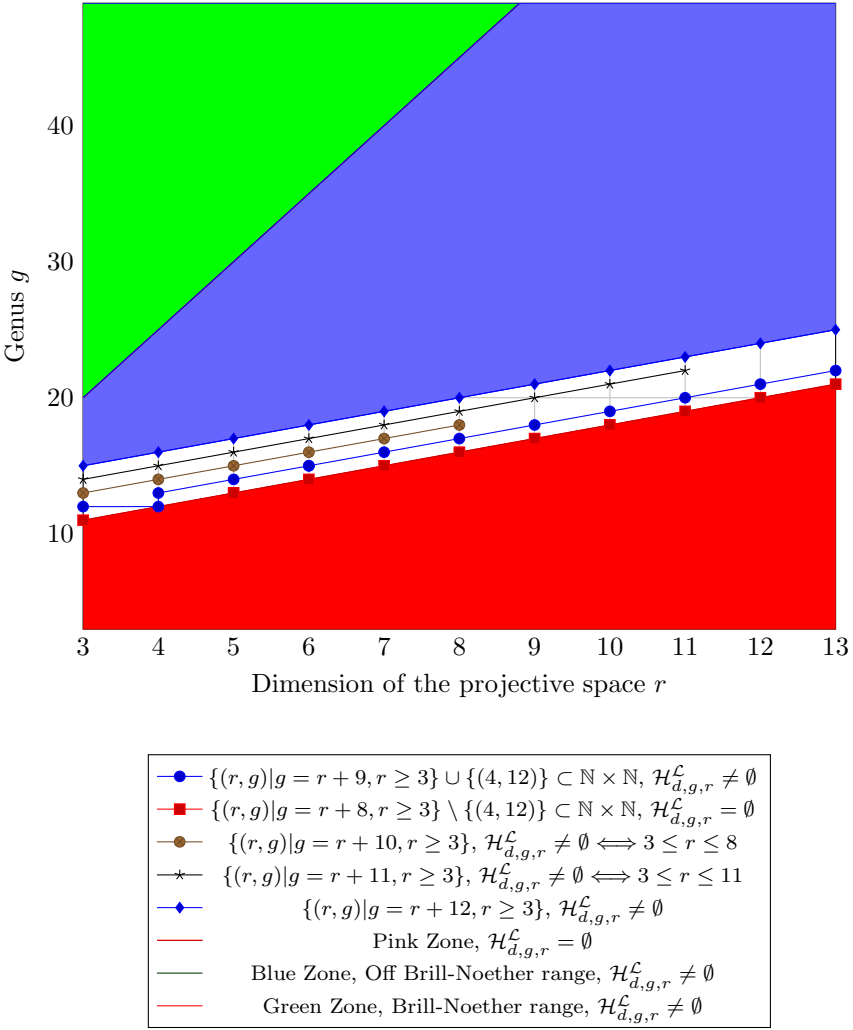

\section{Irreducibility of $\HL{g+r-5, g,r}$ for lowest possible genus $g$}
In this section, we determine the irreducibility of $\HL{g+r-5, g,r}$ when the genus $g$ has the lowest possible 
value with respect to $r$. Recall that by Theorem \ref{sub6x} (a), (b), we know
\begin{itemize} 
\item$
\HL{g+r-5,g,r}=\h{H}_{g+r-5,g,r}=\emptyset\textrm{ for } g\le r+8, r\neq 4
$ and $\HL{g+r-5,g,r}\neq\emptyset$
for  $g=r+9$, $r\ge 3$.
\item For $r=4$, 
$\HL{g+r-5,g,r}=\HL{g-1,g,4}=\emptyset$  if $g\le r+7=11$ whereas 
$\HL{g-1,g,4}\neq\emptyset$ for 
$g=r+8=12$.
\end{itemize}

\vni In general, determining the irreducibility of Hilbert schemes of smooth curves is rather a hard task. However one may obtain the following irreducibility result when the genus $g$ has the minimal possible value
$g=r+9$ ($r\neq 4$) or $g=r+8, r+9$ ($r=4$). 

\begin{prop}\begin{itemize}\label{g=r+9}
\item[(i)] $\HL{g+r-5,g,r}\neq\emptyset$, $g=r+9$   is irreducible for every $r\ge 7$.
\item[(ii)] $\HL{g+r-5,g,r}\neq\emptyset$, $g=r+9$   is {\bf reducible} for $r= 6$.
\item[(iii)] $\HL{g+r-5,g,r}\neq\emptyset$, $g=r+9$ is irreducible for $r=5$.

\item[(iv)] $\HL{g+r-5,g,r}\neq\emptyset$,
$g=r+8$ is irreducible for $r=4$.

\item[(v)] $\HL{g+r-5,g,r}\neq\emptyset$,
$g=r+9$ is {\bf reducible} for $r=4$.
\item[(vi)] $\HL{g+r-5,g,r}\neq\emptyset$, $g=r+9$ is {\bf reducible} for $r=3$.
\end{itemize}
\end{prop}
\begin{proof} 
\hskip3pt  \begin{itemize}
\item[(i),(ii)] If $r\ge 6$ and $g=r+9$, a curve $C\subset \PP^r$ of degree $d=g+r-5=2r+4$ and genus 
$g$  is an extremal curve lying on a surface of minimal degree; $\pi(2r+4,r)=r+9$ if $r\ge 6$. A routine computation shows that if $r\ge 7$  there is only one pair $(a,b)=(3,7-r)\in\mathbb{N}\times\mathbb{Z}$ satisfying the degree and genus formula (\ref{scrolldegree}), (\ref{scrollgenus}) for curves lying on a rational normal surface scroll whereas there are two such pairs $\{(a,b)\, | \,(3,1), (4,-4)\}$ if $r=6$. Hence the result follows
by  Proposition \ref{nearlyextremal}; see also \cite[Theorem 2.5 and Corollary 2.6, page 122]{ACGH}.
\item[(iii)] $r=5$: Let $C\subset \PP^5$ be a smooth curve of genus $g=r+9=14$ and $d=g+r-5=14$.  Note that  
$$\pi_1(14,5)=13\lneq g=14\le\pi (14,5)=15$$ and hence $C$ is a nearly extremal curve lying on a quartic surface $S\subset\PP^5$. If  $S$ is smooth, $S$ is either a rational normal scroll or a Veronese surface. If $S$ is a Veronese surface, $C$ is an isomorphic image a smooth plane curve $C_0$ of degree $t$ with $14=d=2t$. Since $p_a(C)= g=14\neq p_a(C_0)=15$, this does not occur. 
\vni
Suppose that $S$ is a rational normal surface scroll and let $C\in |aH+bL|$. We solve equations (\ref{scrolldegree}) (\ref{scrollgenus}) 
 and get $a=3,\frac{9}{2}$. Hence $C\in |3H+2L|$ is trigonal whose trigonal pencil $g^1_3$ is cut out by rulings of the scroll. Since $$|K_S+C-H|=|(-2H+2L)+C-H|=|4L|$$ the very ample and complete hyperplane series  $g^5_{14}$ of $C\subset\PP^5$ is of the form $|K_C-4g^1_3|$. Conversely, on a trigonal curve $C$  of genus $g\ge12$, the complete linear series $|K_C-4g^1_3|$ is very ample by Lemma \ref{kveryample}. Hence there is an irreducible family $$\h{G}_{3,0}\subset \w{\h{G}}_\h{L}\subset\h{G}^5_{14}$$ sitting over the family of trigonal curves (and dominating the irreducible locus $\h{M}^1_{g,3}$) whose (general) member is of the form $|K_C-4g^1_3|$, which induces an embedding into
 $\PP^5$. Therefore we may deduce  $$\h{G}_{3,0}\stackrel{bir}{\cong} \dim\h{M}^1_{g,3},$$
$$\dim\h{G}_{3,0}=\dim\h{M}^1_{g.3}=2g+1>\lambda (g,g,5)=4g-33,$$ and the irreducible family $\h{H}_{3,0}\subset\HL{g,g,5}$ sitting over $\h{G}_{3,0}$ may constitute a full component. 
 
 \vni It is possible that our smooth linearly normal $C\subset\PP^5$ may lie on a singular quartic surface, which is a cone over a rational normal curve in $\PP^4\subset\PP^5$. Recalling that smooth curves on a cone over a rational normal curve are specializations of curves on a rational normal scroll by Proposition \ref{specialization}, we deduce that such curves  are 
 in the the boundary of the component containing the family $\h{H}_{3,0}$ over $\h{G}_{3,0}$.
\vni
It finally follows that  $\HL{14,14,5}$ is irreducible with only one component whose general element is a trigonal curve lying on a rational normal scroll in $\PP^5$ which we specified above.

\item[(iv)] $r=4, \,g=r+8=12$: Since $\pi(11,4)=g=12$, a smooth curve $C\subset\PP^4$ with $(d,g)=(11,12)$ is an extremal curve lying on a cubic surface. By an easy computation there is only one pair $(a,b)=(4,-1)\in\mathbb{N}\times\mathbb{Z}$ satisfying the degree and genus formula (\ref{scrolldegree}), (\ref{scrollgenus}) for curves lying on a rational normal surface scroll.  Hence the irreducibility follows
by  Proposition \ref{nearlyextremal}.
\vni
\item[(v)] $r=4,g=r+9=13, d=g-1=12$: This case is more involved  and  we need to argue as follows. 
\vni First note that $\pi_1(12,4)=13$. Hence $C$ lies on a cubic  or quartic surface by \cite[Theorem 3.15, page 99]{H1}. 
\vni
Suppose $C$ lies on a smooth cubic surface scroll $S\subset \PP^4$  and let $C\in |aH+bL|$; cf. Remark \ref{cones}. We solve equations (\ref{scrolldegree}) (\ref{scrollgenus}) 
 and get $a=3,\frac{16}{3}$. Hence $C\in |3H+3L|$ is trigonal whose trigonal pencil $g^1_3$ is cut out by rulings of the scroll. Since $$|K_S+C-H|=|(-2H+L)+C-H|=|4L|$$ the very ample and complete hyperplane series  $|H|=g^4_{12}$ of $C\subset\PP^4$ is of the form $|K_C-4g^1_3|$. 
 By the same reasoning as in the case (iii), we may deduce that 
there is an irreducible family 
$$\h{G}_{3,0}\subset\w{\h{G}}_\h{L}\subset\h{G}^4_{14}$$ sitting over the family of trigonal curves (and dominating the irreducible locus $\h{M}^1_{g,3}$) whose (general) member is of the form $|K_C-4g^1_3|$ and induces an embedding into
 $\PP^4$.
 Therefore $$\h{G}_{3,0}\stackrel{bir}{\cong} \dim\h{M}^1_{g,3},$$ 
$$\dim\h{G}_{3,0}=\dim\h{M}^1_{g.3}=2g+1>\lambda (g,g-1,4)=4g-28,$$ and the irreducible family $\h{H}_{3,0}\subset\HL{g,g-1,4}$ sitting over $\h{G}_{3,0}$ has dimension $$\dim\h{H}_{3,0}=\dim\h{M}^1_{g,3}+\dim\PP GL(5)=51.$$

 \vni  We now suppose that a smooth linearly normal curve $C\subset\PP^4$ of degree $d=12$ and genus $g=13$ lies on a quartic surface $S\subset\PP^4$, which is one of the following:
 \begin{itemize}
 \item[(1)] An external projection of a quartic scroll or a Veronese surface in $\PP^5$,
 \item[(2)] a del Pezzo  surface, possibly with finitely many isolated double points,
 \item[(3)] a cone over a elliptic quartic curve in $\PP^3\subset\PP^4$. 
 \end{itemize}
 
 \vni 
 We do not attempt to go through and analyze all the possible cases listed above, which could  be quite cumbersome.
 Instead, we choose a {\bf reasonable} quartic surface $S\subset\PP^4$ and get a numerical description of a smooth curves $C\subset S$ with $(d,g)=(12, 13)$.  We then proceed to compute the normal bundle $N_{C|\PP^4}$, which would eventually justify the existence of an extra component
 of $\HL{12,13,4}$ other than 
the component containing the family $\h{H}_{3,0}$ consisting of trigonal curves which we already obtained. 

\vni
Let $S\subset\PP^4$ be a smooth del Pezzo  which is the isomorphic image of $\PP^2_5 \lhook\joinrel\xrightarrow{|(3;1^5)|} S\subset\PP^4$. Setting $C\in (a;b_1,\cdots,b_5)$,
we have  $$\deg C =3a-\sum b_i=12, C^2=a^2-\sum b_i^2=2g-2-K_S\cdot C=36.$$
By Schwartz's inequality,  one has $(\sum b_i)^2\le 5(\sum b_i^2)$ and substituting 
$\sum b_i=3a-12$, $\sum b_i^2=a^2-36$ we obtain 
$$5\,({a}^{2}-36)- \left( 12-3\,a \right) ^{2}\ge 0 \Longleftrightarrow a=9, b_1=b_2=\cdots b_5=3, $$ and therefore
$$C\in L:=(9;3^5).$$
\vni
By Riemann-Roch on $S$, Serre duality and Kodaira's vanishing theorem, one may easily compute
$$h^0(S,L)=25.$$

\vni
Let  $\mathcal{H}^0_4\subset\mathcal{H}_{12,13,4}$ be the irreducible locus consisting of  curves lying on a smooth del Pezzo  surface. Recall that a smooth del Pezzo  in $\PP^4$ is a complete intersection of two quadric 
hypersurfaces and is completely determined by a pencil of quadrics in 
$\PP^4$; 
\begin{align}\label{del12}
\dim\h{H}^0_4&=\dim\mathbb{G}(1,\PP(H^0(\PP^4,\h{O}(2))))+\dim \PP(H^0(S,L))\\&=
\dim\mathbb{G}(1,14)+24=50> \lambda (12,13,4)+\dim\PP GL(5)=48\nonumber.
\end{align}
The strict inequality in  \eqref{del12}
suggests that  $\mathcal{H}^0_4$ (or its closure) could well be a component other than $\h{H}_{3,0}$. We now verify that the closure of $\h{H}^0_4$ is indeed a component
of dimension
$50$.

\vni
Since 
$$C\in (9;3^5)=3(3;1^5)=3(-K_S),$$ $C\subset S$ is a divisor cut out on $S$ by a cubic 
hypersurface and hence $C$ is a complete intersection of two quadrics and a cubic hypersurface. 
We also note that $$K_S+C\in |(-3;-1^5)+(9;3^5)|= |(6;2^5)|=|2(3;1^5)|,$$ and hence $K_C$ is cut out on $C$ by $\h{O}_S(2)$; $\h{O}_C(K_C)\cong\h{O}_S(2)_{|C}$. In other words, the hyperplane series of $C\subset S\subset\PP^4$ is semi-canonical and $$h^0(C,\h{O}_C(2))=h^0(C,\h{O}_C(K_C))=g=13.$$
We  compute  the dimension  of the tangent space $$T_C{\h{H}_4}\cong H^0(C,N_{C|\PP^4})$$ where $\h{H}_4$ is a component containing the irreducible family $\h{H}^0_4$.
\vni
Since  $C$ is a complete intersection, the normal bundle $N_{C|\PP^4}$ splits; $$N_{C|\PP^4}=\h{O}_C(2)\oplus\h{O}_C(2)\oplus\h{O}_C(3).$$
Noting that $\Oo_C(3)$ is non-special, we have
$$h^0(C,N_{C|\PP^4})=2h^0(C, \h{O}_C(2))+h^0(C,\h{O}_C(3))=50=\dim\h{H}^0_4$$
by \eqref{del12}. Hence the irreducible family  $\h{H}^0_4$ is dense in the component $\h{H}_4$. Since $\dim\h{H}^0_3=51>\dim\h{H}_4$, we have $$\h{H}_4\neq\h{H}_3$$
where $\h{H}_3$ is a component containing the 
irreducible locus $\h{H}^0_3$. The reducibility of $\HL{12,13,4}$ readily follows from this.\footnote{One may  try invoke lower semi-continuity of gonality to ensure that $\h{H}^0_4$ is not in the boundary of $\h{H}^0_3$. However without the computation $\dim H^0(C,N_{C|\PP^4})$, it is presumably possible that the two irreducible families  $\h{H}^0_3$ and $\h{H}^0_4$ may be contained in another component of dimension strictly bigger than 
$\dim\h{H}^0_3$. }

\vni
\item[(vi)] $r=3$, $g=r+9=12$: Let $C\subset\PP^3$ be a smooth curve with $(d,g)=(10,12)$. From the exact sequence $$0\rightarrow\h{I}_C(3)\rightarrow\h{O}_{\PP^3}(3)\rightarrow\h{O}_C(3)\rightarrow 0$$ and Riemann-Roch one computes $H^0(\PP^3,\h{I}_C(3))\neq 0$, hence $C$ lies on a possibly reducible cubic surface.  Indeed there exists a family $\h{I}_2$ consisting of  smooth curves of degree $10$ of type $(3,7)$ on smooth quadrics; $$\dim\h{I}_2=\dim \PP(H^0(\PP^3, \h{O}(2)))+\dim\PP(H^0(\PP^1\times\PP^1\h{O}_{\PP^1\times\PP^1}(3,7)))=40.$$  There also exists a family $\h{I}_3$ consisting of curves on  
non-singular cubics, e.g. a general member in the linear system $(9;3^5,2)$ on $\PP^2_6 \lhook\joinrel\xrightarrow{(3;1^6)} \PP^3$; 
$$\dim\h{I}_3=\dim\PP(H^0(\mathbb{P}^3, \mathcal{O}(3)))+\dim|(9;3^5,2)|=40=4\cdot 10.$$ 
Since $\dim\h{I}_2=\dim\h{I}_3$, one is not in the closure of the other and hence these two families form
two distinct irreducible components of $\HL{g-2,g,3}$ of the same expected dimension; cf. \cite[Theorem 4.1 and Corollary 4.2]{speciality} for other related results. 
\end{itemize}
\end{proof}
\begin{rmk}  It is worthwhile to note that  $\h{H}_{g+r-5,g,r}=\HL{g+r-5,g,r}$ for triples $(d,g,r)$ in Proposition \ref{g=r+9}. This follows from the fact that extremal, nearly extremal  or curves with $g=\pi_1(d,r)$ are always linearly normal; note that
$$\pi(d,r+1)<\pi(d,r+1)+\left[\frac{d}{r}\right]\le \pi_1(d,r)<\pi(d,r).$$
\end{rmk}

\begin{rmk} The following is the current status of our knowledge about the irreducibility of  $\HL{d,g,r}$ with low index of speciality $\al=g-d+r$; cf. \cite{JPAA, lengthy, speciality}.
\begin{itemize}
\item[(i)] If $\al=0$, we have $\rho(d,g,r)=g>0$ and there is no reducible $\HL{d,g,r}$. Hence the Modified Assertion of 
Severi is trivially true in this case. 
\item[(ii)] If $\al=1$, $\rho (d,g,r)=g-(r+1)$.  $\HL{d,g,r}\neq\emptyset$ if and only if $\rho\ge 0$. Moreover, every non-empty $\HL{d,g,r}$ is irreducible and therefore Modified Assertion of Severi also holds in this case.
\item[(iii)] If $\al =2$, $\rho (d,g,r)=g-2(r+1)$. One has $\HL{d,g,r}\neq\emptyset$ if and only if $g\ge r+3$
and every non-empty $\HL{d,g,r}$ is irreducible, which follows from the irreducibility of Hurwitz scheme. Hence the  Modified Assertion of Severi holds everywhere beyond the Brill-Noether range. 
\item[(iv)]  If $\al =3$, $\rho (d,g,r)=g-3(r+1)$. It is known that if $g\ge 2r+3$,  $\HL{d,g,r}$ is irreducible, which basically follows from the irreducibility of Severi variety of plane nodal curves; cf. \cite{JPAA}. Therefore the Modified Assertion of Severi holds far beyond the Brill-Noether range. The lower bound $g\ge 2r+3$ for the irreducibility of $\HL{g+r-3,g,r}$ is sharp for every $r\ge 5$. Indeed  for all the values $g$ in the range $r+5\le g\le 2r+2$ such that $r\ge 5$,  $\HL{d,g,r}$ is reducible. 
Moreover all the triples $(d,g,r)=(g+r-3,g,r), r\ge 3$ for which  $\HL{d,g,r}$ is irreducible have been idetified; cf. \cite[Theorem 1.1, Remark 3.1, Theorem 3.9, Corollary 3.10]{lengthy} and references therein. 

\item[(v)] For $\al =4$, our overall knowledge on the irreducibility of $\HL{d,g,r}$ drastically decreases 
compared with the Hilbert schemes having smaller index of speciality. Unless the genus $g$ is near to
the minimal possible value with respect to $r$, virtually nothing is known about the irreducibility of the $\HL{d,g,r}$
either inside or outside the Brill Noether range. 
\item[(vi)] In the next section, we proceed one step  further to determine the irreducibility of $\HL{d,g,r}$ with $\al =5$ when the genus $g$ is very near to minimal possible value with respect to $r$.
\end{itemize}
\end{rmk}

\section{Irreducibility of $\HL{2r+5,r+10,r}$} 

In Theorem \ref{sub6x}, we saw that if $g=r+10$,  $\HL{g+r-5,g,r}\neq\emptyset$ if and only if $3\le r\le 8$. However when $g=r+9$,  $\HL{g+r-5,g,r}\neq\emptyset$ for every $r\ge 3$. 
A similar phenomenon of this kind occurs for low index of speciality $3\le \al\le 4$; cf. \cite[Proposition 3.3]{lengthy} and \cite[Theorem 2.4]{speciality}. In this section we determine the irreducibility of this peculiar Hilbert scheme $\HL{2r+5,r+10,r}$ for every $r\neq 4$. We start to prove the following.

\begin{thm} \label{r+10}$\HL{2r+5,r+10,r}$ is irreducible if $r=6,7$ and reducible if $r=8$.
\end{thm}
\begin{proof} Given a  component $\h{G}\subset\w{\h{G}}_\h{L}\subset\h{G}^r_{g+r-5}$ whose general element $(p,\mathcal{D})\in\Gg$ corresponds to a very ample and complete linear series $\mathcal{D}$ on the curve $C=\xi^{-1}(p)$,   set $\h{E}=\h{D}^\vee=g^4_{13}$.  Note that $\h{E}=g^4_{13}$ is not compounded; otherwise we may use Lemma \ref{easylemma1} to conclude that $\h{D}$ is not very ample.  We also note that $\h{E}$ is base-point-free since $\pi(12,4)=15<r+10=g$.
In the range $6\le r\le 8$, $$\pi_1(13,4)=15<g=r+10\le \pi (13,4)=18, $$
hence  a curve of (geometric) genus $g=r+10$ and degree $13$ in $\PP^4$ lies on a surface of minimal degree in $\PP^4$; cf. \cite[Theorem 3.15, page 99]{H1}.  Let $\ce\subset\PP^4$ be the image curve of the morphism induced by $\h{E}$ and let $S\subset\PP^4$ be the  cubic surface containing $\ce$.

\vni
{\bf (A)} We first assume that $S$ is a non-singular cubic rational normal surface scroll. Recall that there are obvious isomorphisms 
$$S\cong\mathbb{F}_1\cong\PP^2_1 $$ together with  natural correspondences between generators of Picard groups;

\begin{align}\label{picid}
& \textrm{Pic}S& &\cong &  &\textrm{Pic}(\mathbb{F}_1)& &\cong& &\textrm{Pic}(\PP^2_1)&\nonumber\\
&\hskip 10pt\uin& &{}&  &\hskip 10pt\uin& &{}& &\hskip 10pt\uin&\nonumber\\
&(H,L)&  & \longleftrightarrow& &(e+2f,f)&  &\longleftrightarrow &  &(2l-\tilde{e}, l-\tilde{e}),& 
\end{align}

\vni
$e\in\textrm{Pic}(\mathbb{F}_1)$ is the class of the section with minimal self intersection $e^2=-1$, $f$ the class of fibre of the Hirzebruch surface $\mathbb{F}_1\rightarrow \PP^1$; $l$ the class of a line in $\PP^2_1$, $\w{e}$ the class of the exceptional divisor on $\PP^2_1$.
By abusing notation we shall make no distinction between $e\in\textrm{Pic} (\mathbb{F}_1)$ and $\w{e}\in\textrm{Pic} (\PP^2_1)$ and  use the same symbol $e$. 

\vni
Let $C_\h{E}\in |aH+bL|$ and we solve equations (\ref{scrolldegree}) and  (\ref{scrollgenus}); $n=3, d=13, ~p_a(C_\h{E})=r+10$. The solutions of integer pairs $(a,b)\in\mathbb{N}\times\mathbb{Z}$ exist only when $r=8$; $(a,b)\in\{(5,-2), (4,1)\}$. 
Hence $C_\h{E}$ is smooth only if $r=8$ and is singular if $r=6,7$. Set $\Nn_{1}:=\h{O}_S(5H-2L)$
and  $\Nn_{2}:=\h{O}_S(4H+L)$.

\vni
{\bf (A-1):}  We first deal with the case $r=8$.
 By the  correspondence \eqref{picid}, we have
\begin{itemize}
\item[(i)] 
$C_\h{E}\in |\Nn_{1}|=|5H-2L|=|5(e+2f)-2f|=|8f+5e|=|8l-3e|=(8;3)$

\noindent
$|D|:=|K_S+C_\h{E}-H|=|2H-L|=|2(e+2f)-f|=|3l-e|=(3;1)$

\noindent
$\h{E}^\vee=|K_{C_\h{E}}-\h{E}|=(3;1)_{|{C_{\h{E}}}}$
\vni

\item[(ii)]
$C_\h{E}\in|\Nn_{2}|=|4H+L|=|4(e+2f)+f|=|9f+4e|=|9l-5e|=(9;5)$

\noindent
$|D|:=|K_S+C_\h{E}-H|=|H+2L|=|e+4f|=|4l-3e|=(4;3)$

\noindent
$\h{E}^\vee=|K_{C_\h{E}}-\h{E}|=(4;3)_{|{C_{\h{E}}}}$
\end{itemize}
\vni
Let $\rho_{\Nn_{i}}: H^0(S,\h{O}_S(D))\rightarrow H^0(\ce, \h{O}_S(D)\otimes\h{O}_{\ce})$ be the restriction map.
We have $$\ker\rho_{\Nn_i}\cong H^0(S,\Nn_i^{-1}\otimes\h{O}_S(D))\cong H^0(S,\h{O}_S(-3H+L))=0,$$  $H^1(S, \h{O}_S(-3H+L))=0$ by 
Kodaira vanishing theorem since  $|3H-L|$ is (very) ample. Hence $\rho_{\Nn_i}$ ($i=1,2$) is an isomorphism and therefore in both cases (i) and (ii), the linear system $\h{E}^\vee$ on $C_\h{E}$ is very ample since the corresponding linear system $(3;1)$ or $(4;3)$ on $S\cong\PP^2_1$ is very ample. One may interpret our current situation as follows. 

\vni 
{\bf Observation 1:}
Given
a very ample $(\h{D}=g^r_d,\w{C})\in\h{G}\subset \w{\h{G}}_{\h{L}}\subset\h{G}^r_d$, we take the residual series $\h{D}^\vee=\h{E}$ on $\w{C}$. If $r=8$, $\h{E}$ induces a morphism onto a smooth curve $\ce\subset\PP^4$ sitting on the surface $S$ of minimal degree belonging to one of the  linear systems $$\h{N}_1=|5H-2L| \textrm{ or }\h{N}_2=|4H+L|.$$ Therefore $\ce$ may be regarded as an element of the (set theoretic) union of Severi varieties  $\Sigma_{\h{N}_1,\delta}\cup\Sigma_{\h{N}_2,\delta}\subset\textrm{Div}S$ where $\delta =0$. 

\vni
One may  do the same thing for the cases $r=6,7$. Upon fixing a rational normal surface scroll $S\subset\PP^4$ once and for all (up to {$\textrm{Aut}(S)$}) and putting $\delta =8-r$, we have a natural generically injective correspondence 
\begin{align*}&\w{\h{G}}_{\h{L}}\hskip 24pt\stackrel{\tau}{\dashrightarrow}& &\w{\h{G}}^\vee_{\h{L}}\hskip 24pt\stackrel{\iota}{\dashrightarrow}& &{\tilde \Sigma}_{\h{N}_1,\delta}\cup{\tilde\Sigma}_{\h{N}_2,\delta}\subset\textrm{Div}{S}/\textrm{Aut}(S)&\\
&\uin&&\uin&&\hskip 24pt\uin&\\
&\h{D}\hskip 24pt\longmapsto& &\h{D}^\vee\hskip 24pt\longmapsto& &\hskip 20pt C_{\h{D}^\vee}=\ce&
\end{align*}
where ${\tilde\Sigma}_{\h{N}_i,\delta}$ is the image of $\Sigma_{\h{N}_i,\delta}$ under the 
projection $\textrm{Div}S\rightarrow \textrm{Div}S/\textrm{Aut}(S)$.
Let  $\h{U}\subset{\tilde\Sigma}_{\h{N}_1,\delta}\cup{\tilde\Sigma}_{{\h{N}_2},\delta}$ be an open set on which the inverse of the rational map $\iota\circ\tau$ is defined.  

\vni
In the case $r=8$, as we saw in (A-1),  for a general  $\ce\in\h{N}_i=\Sigma_{\h{N}_i,0}$, $|K_S+\ce-\h{O}_S(1)|_{\ce}$ is very ample, thus we may conclude
$$\h{U}\cap {\tilde\Sigma}_{\h{N}_i,\delta}\neq\emptyset \textrm{  for }i=1,2. $$
\vni
Hence there are two irreducible families $\h{G}_i \subset \w{\h{G}}_\h{L}$ corresponding to ${\tilde\Sigma}_{\h{N}_i,\delta}$ of dimensions
\footnote {In general $\textrm{Aut}(\mathbb{F}_n)$ is isomorphic to the semi-direct product 
$\textrm{Sym}^n\mathbb{C}^2\ltimes GL(2)/\mu_n$ where $\mu_n$ is the group of $n$-th roots of unity in 
the group of scalar matrices; cf. \cite[page 12]{Blanc}.
Hence $\textrm{Aut}(\mathbb{F}_1)\cong \mathbb{C}^2\ltimes GL(2)$.  Alternatively, by recalling that  rational normal surface scrolls in $\PP^s$ move in an irreducible family $\h{S}_s$ of dimension   $(s+3)(s-1)-3$,
$$\dim\h{G}_i=\dim\h{N}_i+\dim\h{S}_4-\dim\PP(GL(5))=\dim\h{N}_i-6.$$}
$$\dim\h{G}_i=\dim\h{N}_i/\textrm{Aut}(S)=\dim\h{N}_i-\dim(\mathbb{C}^2\ltimes GL(2))=\dim\h{N}_i-6.$$
Therefore it follows that $\w{\h{G}}_{\h{L}}$ which $\HL{21,18.8}$ sits  over as a $\Aut(\PP^8)$-bundle is reducible with two components $\h{G}_i$
corresponding to two irreducible Severi varieitis $\Sigma_{\h{N}_i,0}$, $i=1,2$. We remark that 
one of the two irreducible families $\h{G}_i$ is not in the boundary of the other by semi-continuity of gonality;$$\dim\h{G}_2=33>\dim\h{G}_1=32 \mbox{ whereas } \textrm{gon}(C_1)>\textrm{gon}(C_2) \mbox{ for } C_i\in\mathcal{N}_i$$
 
\vni
{\bf Observation 2:}
We remark that the original smooth curve $C\subset\PP^8$ with $(d,g)=(21,18)$ is neither extremal nor nearly extremal; $\pi_1(d,r)=g$. Therefore it might be quite complicated to carry out comprehensive analysis of curves in $\PP^8$ lying on a surface of degree eight or seven.
However the residual curve $\ce\subset\PP^4$ is an extremal curve lying on a surface of minimal degree, 
from which we were able to deduce the reducibility  $\HL{21,18.8}$ from the reducibility of the family of residual curves which are members of two typical Severi varieties $\Sigma_{\h{N}_i,0}  (i=1,2)$ on the rational surface $S$. 



\vni 
{\bf (A-2):}
In the same vein, we proceed  to handle the cases $6\le r\le 7$. In this case, the residual (singular) curve $\ce\subset\PP^4$ is not an extremal curve but  a nearly extremal curve; $$\pi_1(13,4)=15\lneq g=r+10\lneq\pi(13,4)=18.$$
By Remark \ref{Severi} (i), we may assume  that $\ce$ is a nodal curve lying on a surface of minimal degree $S\cong\mathbb{F}_1$ with $\delta =8-r$ nodes as its only singularities. 
Again,  the residual curve $\ce$ -- as a divisor of $S\cong\mathbb{F}_1$ -- may be regarded as an element in the union of Severi varieties  
$$\Sigma_{\h{N}_1,8-r}\cup\Sigma_{\h{N}_2,8-r}.$$

\vni
We now take a curve $C\subset S\subset\PP^4$ corresponding to a general element  of $\Sigma_{\h{N}_i,8-r}$.
We blow up $\mathbb{F}_1\cong\PP^2_1$ at the $(8-r)$ nodal points to get the surface $\PP^2_{9-r}$.  We may assume that the singularities of $C$ are not in the exceptional divisor $e$ of $\PP^2_1$. Let $\w{C}$ be the proper (strict) transformation of $C$ under the blow up $\mathbb{P}^2_{9-r}\stackrel{\pi}{\rightarrow}\PP^1_1\cong\mathbb{F}_1\cong S$. Let $e$, $e_1, \cdots e_{8-r}$ be the exceptional divisors of the blown up surface $\PP^2_{9-r}$ and set $\w{H}=\pi^*\h{O}_S(1)$.
We will determine the class of $\w{C}$ and the linear system $|K_{\PP^2_{9-r}}+\w{C}-\w{H}|$ which may cut $\Ee^\vee$ on $\w{C}$.  Our goal is to check if $$\h{D}=|K_{\PP^2_{9-r}}+\w{C}-\w{H}|_{\w{C}}$$ is very ample.
In other words, 
we would like to examine if 
the series $\h{D}$ on smooth $\w{C}$ -- which is the residual series of the series cut out by hyperplanes in $\PP^4$ --  is very ample.
If this is the case, we may conclude that  either  one of the Severi varieties $\Sigma_{\h{N}_i,{8-r}}$ --  which is irreducible by 
a work of Tyomkin \cite{tyomkin} -- corresponds to an irreducible component of $\w{\h{G}}_\h{L}\subset\h{G}^r_{2r+5}$. 
We denote by 
\begin{equation}\label{severi}
\Sigma_{\h{N}_i,\delta}^\vee:=\{|K_{\PP^2_{9-r}}+\w{C}-\w{H}|_{|{\w{C}}}:C\in\Sigma_{\h{N}_i,\delta}, \,\w{H}=\pi^*\h{O}_S(1),\, \PP^2_{9-r}\stackrel{\pi}\rightarrow S\}\end{equation} the family consisting of such $\h{D}$'s.

\begin{itemize}
\item[(iii)]  $C\in\Sigma_{\h{N}_1,{8-r}}\subset\h{N}_1=|5H-2L|$: With a minor modification of the previous computation in (A-1)-(i), we have $$\w{C}\in (8;3,2^{8-r})$$ and may deduce that 
\begin{align*}
|D|:&=|K_{\PP^2_{9-r}}+\w{C}-\w{H}|\\&=|(-3;-1^{9-r})+(8;3,2^{8-r})-(2;1, 0^{8-r})|=(3;1^{9-r}),
\end{align*}
which is very ample inducing an embedding $\phi: \PP^2_{9-r}\stackrel{(3;1^{9-r})}{\longrightarrow}\PP^r$ as well as an isomorphism $\w{C}\cong \phi({\w{C}})\subset\PP^r$;
$$\deg \phi({\w{C}})= \w{C}\cdot (K_{\PP^2_{9-r}}+\w{C}-\w{H})=2r+5, ~g(\phi(\w{C}))=g(\w{C})=r+10.$$

\vni
 Borrowing the notation used in the proof of Theorem \ref{sub6x} (\ref{octic}), we note that the linear system $$(8;3,2^{8-r})=\h{L}(0,r)$$
 to which $\w{C}$ belongs, is a special case of the one we encountered there.
Furthermore, one may check  that the restriction map
$$\rho_{\h{L}(0,r)}: H^0(\PP^2_{9-r},\h{O}_{\PP^2_{9-r}}(D))\rightarrow H^0(\w{C}, \h{O}_{\PP^2_{9-r}}(D)\otimes\h{O}_{\w{C}})$$
is an isomorphism;
\begin{align*}
\ker\rho_{\h{L}(0,r)}&=H^0(\PP^2_{9-r},\h{O}_{\PP^2_{9-r}}(D)\otimes{\h{L}(0,r)}^{-1})\\&=
H^0(\PP^2_{9-r},\h{O}_{\PP^2_{9-r}}(K_{\PP^2_{9-r}}-\w{H}))\\&=H^0(\PP^2_{9-r},\h{O}_{\PP^2_{9-r}}(-(5;2,1^{8-r})))=0
\end{align*}
and $H^1(\PP^2_{9-r},\h{O}_{\PP^2_{9-r}}(D)\otimes{\h{L}(0,r)}^{-1})=0$ by Kodaira vanishing theorem.
Therefore a general element of  $\Sigma_{\h{N}_1,\delta}^\vee$ is very ample and  $\Sigma_{\h{N}_1,\delta}^\vee$ corresponds to 
a component of $\w{\h{G}}_\h{L}\subset\h{G}^r_{2r+5}$.

\item[(iv)] $C\in\Sigma_{\h{N}_2,{8-r}}\subset\h{N}_2=|4H-L|$:  In the same way, we get $$\w{C}\in (9,5,2^{8-r})$$ and 
\begin{align*}
|D|:=|K_{\PP^2_{9-r}}+\w{C}-\w{H}|&=(-3;-1^{9-r})+(9;5,2^{8-r})-(2,1, 0^{8-r})\\&=(4;3,1^{8-r}),
\end{align*}
which is {\bf not} very ample unless $r=8$ by \cite{sandra}. On the other hand, the corresponding restriction 
map is still an isomorphism. 
Finally we remark that the $(-1)$ curve $l-e-e_1=(1;1,1,0^{7-r})$ is contracted by the morphism induced by 
$(4;3,1^{8-r})$, whereas $\w{C}\cdot (l-e-e_1)=2$ hence producing a singularity on the image curve. Therefore it follows that a general element of  $\Sigma_{\h{N}_2,\delta}^\vee$ is {\bf not }very ample and  $\Sigma_{\h{N}_2,\delta}^\vee$ {\bf does not} correspond to 
a component of $\w{\h{G}}_\h{L}\subset\h{G}^r_{2r+5}$.
\end{itemize}

\vni
{\bf (B)} We assume that $\ce$ lies on a rational normal cone $S\subset\PP^4$.  We borrow  the notation from Remark \ref{minimal} (iii) and take  $n=3$, $d=13$ there.
Let $\widetilde{C}_{\Ee}$ be the proper transformation
of $\ce$ under the minimal desingularization $\FF_3\to S\subset\PP^4$ given by $|C_0+3f|$.  By \eqref{conevertex} and \eqref{cone}, we have 

\[(k,m,p_a(\w{C}_\Ee))=
\begin{cases} (4,1,18)\\
(3,4,15)\\
(2,7,9)
\end{cases}
\] 
Note that the second  and third triples are not possible since $p_a(\w{C}{_\Ee})<g=r+10=g(\w{C}_{\Ee})$, a contradiction. Hence we only have 
$(k,m,p_a(\w{C}_\Ee))=(4,1,18)$.
\vni
{\bf (B-1):} If $r=8$, $\ce\subset S$ is smooth: $p_a(\w{C}{_\Ee})=18=g$.
By Propostion \ref{specialization}, $\ce$ is a specialization of curves lying on a smooth rational normal scroll.
\vni {\bf (B-2):} If $r=7$, $\w{C}{_\Ee}$ is singular with a double point $p_0\in\w{C}_\Ee$. Setting 
$\widetilde{C}_{\Ee}\in|aC_0+bf|$, we have $\widetilde{C}_{\Ee}\cdot (C_0+3f)=b=13, a=k=4$ and $\widetilde{C}_{\Ee}\in
|4C_0+13f|$.
\vni
Choose $f_0$, the fibre containing the unique singular point $p_0\in\w{C}_\Ee$.
Note that $(k,m,p_a(\w{C}_\Ee))=(4,1,18)$, thus $\ce$ is smooth at the vertex of the cone $S\subset\PP^4$ and we may assume $p_0\notin C_0$. We blow up $\mathbb{F}_3$ at $p_0\in\widetilde{C}_{\Ee}$ and let $e$ be the exceptional divisor of the blow up $\mathbb{F}_{3,1}\stackrel{\pi}{\rightarrow}\mathbb{F}_3$. Let $\tilde{f}_0$ ($\widehat{C}$ resp.) be the proper transformation of $f_0$ ($\widetilde{C}_{\Ee}$ resp.). By abusing notation, we denote 
$\pi^*(C_0) ~~\&~~ \pi^*(f)$ by $C_0 ~~\&~~f$. We have
$$\widehat{C}\in|4C_0+13f-2e|$$ and consider
\begin{align*}
\h{M}&:=|\widehat{C}+K_{\mathbb{F}_{3,1}}-(C_0+3f)|\\&=|(4C_0+13f-2e)+(-2C_0-5f+e)-(C_0+3f)|\\&=|C_0+5f-e|.
\end{align*}
 Since 
$\h{O}_{\mathbb{F}_{3,1}}(\h{M}-\widehat{C})=\h{O}_{\mathbb{F}_{3,1}}(-(3C_0+8f-e))$
and $f\cdot (\h{M}-\widehat{C})<0$, we see that $h^0({\mathbb{F}_{3,1}},\h{O}(\h{M}-\widehat{C}))=0$ implying the injectivity of  the restriction map 
$$\rho: H^0({\mathbb{F}_{3,1}},\h{O}(\h{M}))\longrightarrow H^0(\widehat{C},\h{O}(\h{M})\otimes\h{O}_{\widehat{C}}).$$
Suppose $\rho$ is not surjective, i.e. assume
$$\mathrm{Im}(\rho)\subsetneq H^0(\widehat{C},\h{M}\otimes\h{O}_{\widehat{C}}).$$
By the projection formula, $\pi_*\pi^*\h{O}_{\mathbb{F}_3}(h+5f)=\h{O}_{\mathbb{F}_3}(h+5f)$
 \begin{align*}
h^0(\mathbb{F}_{3,1},\pi^*\h{O}_{\mathbb{F}_3}(h+5f))&=h^0(\mathbb{F}_3, \pi_*\pi^*\h{O}_{\mathbb{F}_3}(h+5f)))\\&=h^0(\mathbb{F}_{3},\h{O}(h+5f))=9
\end{align*}  and hence
\begin{align*}h^0(\mathbb{F}_{3,1},\h{O}(\h{M}))&=h^0(\mathbb{F}_{3,1},\h{O}(h+5f-e))\\&=h^0(\mathbb{F}_{3},\h{O}(h+5f))-1=8.
\end{align*}
Then by the injectivity and the non-surjectivity of $\rho$, 
\begin{align*}s:=\dim\PP(H^0(\widehat{C},\h{O}(\h{M})\otimes\h{O}_{\widehat{C}}))&>\dim\PP(H^0(\mathbb{F}_{3,1},\h{O}(\h{M}))\\&=\dim\PP(\mathrm{Im}(\rho))=7.
\end{align*}
Since $|\mathrm{Im}(\rho)|\subsetneq \PP(H^0(\widehat{C},\h{O}(\h{M})\otimes\h{O}_{\widehat{C}}))$ induces a morphism birational onto its image, the complete linear system $\PP(H^0(\widehat{C},\h{O}(\h{M})\otimes\h{O}_{\widehat{C}}))$ is still birationally very ample, which contradicts the Castelnuovo genus bound for a curve of degree $\widehat{C}\cdot\h{M}=(4h+13f-2e)\cdot (h+5f-e)=19$ in $\PP^{s\ge 8}$; $\pi(19,8)=15<g=16$. Therefore we have 
$$\mathrm{Im}(\rho)= H^0(\widehat{C},\h{O}(\h{M})\otimes\h{O}_{\widehat{C}})$$
and the restriction map $\rho$ is an isomorphism. 
Note that 
$$\tilde{f}_0^2=-1, \tilde{f}_0\cdot e=1,  C_0\cdot\tilde{f}_0=1,\widehat{C}\cdot \tilde{f}_0=2, 
\h{M}\cdot\tilde{f}_0=(C_0+5f-e)\cdot \tilde{f}_0=0.$$
Hence the morphism $\psi$ induced by  $\h{M}=|\widehat{C}+K_{\mathbb{F}_{3,1}}-(h+3f)|$ contracts the $(-1)$ curve $\tilde{f}_0$ and the image $\psi (\widehat{C})\subset\PP^7$ acquires a singularity since $\widehat{C}\cdot \tilde{f}_0=2$. Recall that $\Ee=\h{O}_{\widehat{C}}(1)$. Hence 
$$\PP(H^0(\widehat{C},\h{O}(\h{M})\otimes\h{O}_{\widehat{C}}))=|\Ee^\vee|=|\Dd|$$ is not 
very ample, a contradiction.

\vni{\bf (B-3):} If $r=6$, $\w{C}{_\Ee}\subset\FF_3$ is singular. We adopt the same  strategy (as for the case $r=7$) to show 
that a general element in the Severi variety of curves of geometric genus $g=16$ on $\mathbb{F}_3$ in  $\h{L}=|4h+13f|$ does not have very ample $|K_{\widehat{C}}(-1)|$, where $\widehat{C}$ is the normalization of $\w{C}_{\Ee}\subset\mathbb{F}_3$. Since being very ample is an open condition, we may assume that $\w{C}_{\Ee}\subset\FF_3$ has two nodes which lie outside the section $C_0\subset\FF_3$ of minimal self intersection.  Therefore all the argument remain parallel to the case $r=7$ -- except blowing up $\FF_3$ at two nodal points -- and we omit details.

\vni
{\bf Conclusion:}
\begin{enumerate}
\item[(1)] From (A-2), a general element of the family $\Sigma_{\h{N}_2,\delta}^\vee$ is not very ample if $\delta =8-r\ge 1$. Thus we have 
$$\h{U}\cap\Sigma_{\h{N}_2,\delta}=\emptyset \textrm{ and  }\h{U}\cap\Sigma_{\h{N}_1,\delta}\neq\emptyset$$
\item[(2)] From (B), the family of nodal curves lying on a rational normal cone $S$ does not have very ample residual hyperplane series and hence does 
contribute to a component of $\HL{2r+5,r+10,r}$ when $r=6,7$.
\item[(3)] The irreducibility of $\HL{2r+5,r+10,r}$ for $r=6,7$ readily follows.
\end{enumerate}
\end{proof}


 \vni
\begin{rmk}\label{easyveryample} (a) We saw in the course of the proof of Theorem \ref{r+10} that  the peculiar and seemingly artificial linear system  (\ref{octic}) on del Pezzo  surfaces which we used in the proof of Theorem \ref{sub6x}  arises naturally by analyzing nearly extremal curves induced by $\h{E}=\h{D}^\vee$ corresponding to a  general element  of a component of 
 $\HL{2r+5,r+10,r}$. Another such linear system (\ref{nonic}) in Theorem \ref{sub6x}) corresponds to a general element of a 
component of $\HL{2r+6,r+11,r}$ -- another  peculiar Hilbert scheme which is non-empty only for low $r$ -- and this $\HL{2r+6,r+11,r}$ will be studied in the next section. 

\noindent
(b) The generic example $\h{L}(0,r)$ (for $r=6,7,8$)  corresponding to $\Sigma_{\h{N}_1, 8-r}$  remain very ample for $3\le r\le 5$, . However, in the cases $3\le r\le 5$, it is possible that $\ce$ may lie on a surface in $\PP^4$ of degree $t\ge 4$.
Therefore one cannot determine the irreducibility of $\HL{2r+5, r+10, r}$ solely based on the irreducibility of the corresponding Severi varieties. 
\vni

\end{rmk}

\begin{rmk} 
\begin{itemize} 
\item[(i)] The irreducibility of $\HL{2r+5,r+5,r}$ for $r=5$ is known by a recent work of the author jointly with E. Ballico \cite[Theorem 4.1]{rmi}. 

\item[(ii)] For $r=4$, the irreducibility of $\HL{2r+5,r+10,r}$ has not been settled yet. However, we are able to show that $\HL{13,14,4}$ has a generically reduced component of minimal possible dimension as follows.
\end{itemize}
\end{rmk}
\begin{prop}\label{d13r4} $\Hh_{13,14,4}$ has a component of the expected dimension 
$$\Xx(13,14,4)=\lambda(13,14,4)+\dim\Aut(\PP^4)=52.$$
\end{prop}
\begin{proof} We borrow notations and computations from the proof of Theorem \ref{sub6x} (c). Take a general $C\in (8;3,2^4)=\Ll(0,4)$; cf. \eqref{octic}. We have $C\subset\PP^2_5\subset\PP^4$ and $(d,g)=(13,14)$. We consider the irreducible family $\Ff$ of smooth curves lying on  smooth del Pezzo surfaces in $\PP^4$;
$$\Ff:= \{C \,|\,(d,g)=(13,14), \,C\subset T\cong\PP^2_5\subset\PP^4, \,C\in (8;3,2^4)\}.$$

\vni
Note that the residual curve (by definition, the curve induced by the residual series of the hyperplane series of $C$) $C^\vee\subset\PP^4$ lies on a smooth cubic scroll $S\subset\PP^4$, 
$C^\vee\in|5H-2L|=\Nn_1$ and $C^\vee\in\Sigma_{\Nn_{1,4}}$; cf. part (A-2) (iii) of the proof of Theorem \ref{r+10}. We have
$\dim\Sigma_{\Nn_{1,4}}=\dim|\Nn_1|-4=34$, hence 
$$\dim\Ff=\dim\Sigma_{\Nn_{1,4}}-\dim\Aut(S)+\dim\Aut(\PP^4)=34-6+24=52.$$
We want to estimate the dimension of a component $\h{H}\subset\h{H}_{13,14,4}$ containing $\Ff$. To do this, we compute $h^0(C,N_{C|\PP^4})=\dim T_C{\h{H}}$.
 
 
 \vni
We consider the standard exact sequence
 $$0\rightarrow N_{C|S}\rightarrow N_{C|\PP^4}\rightarrow N_{{S|\PP^4}_{|C}}\rightarrow 0.$$
 
 \vni 
We note that 
 $$N_{C|S}=\h{O}_S(C)_{|C}=\h{O}_S(C)\otimes \h{O}_C.$$ 
 \vni
 (i) Claim: $h^0(C, N_{C|S})=26$ and  $h^1(C, N_{C|S})=0$.
 \vni This follows from the exact sequence
 $$0\rightarrow\h{O}_S\rightarrow\h{O}_S(C)\rightarrow\h{O}_S(C)\otimes\h{O}_C\rightarrow0~\, ;$$ 
 a routine computation yields $$h^0(S, \h{O}_S(C))=\frac{(8+1)(8+2)}{2}-\frac{3(3+1)}{2}-4\cdot \frac{2(2+1)}{2}=27,$$ and 
 $h^1(S,\h{O}_S)=0$ since $S$ is rational. 
 Thus,  $$h^0(C, N_{C|S})=26 \mathrm{~ and ~}  h^1(C, N_{C|S})=0$$ by Riemann-Roch.
 
 \vni
 (ii) Claim: $$h^0(C, N_{{S|\PP^4}_{|C}})=h^0(C, N_{S|\PP^4}\otimes\h{O}_C)=26.$$
 \vni
 Since $S$ is a complete intersection of two quadrics, we have
 $$N_{S|\PP^4}=\h{O}_{S}(2)\oplus\h{O}_{S}(2), $$
 and hence 
 $$N_{S|\PP^4}\otimes\h{O}_C=(\h{O}_{S}(2)\oplus\h{O}_{S}(2))\otimes\h{O}_C=\h{O}_C(2)\oplus\h{O}_C(2).$$
 By Riemann-Roch,  $$h^0(C, \h{O}_C(2))=26-14+1+h^0(K_C\otimes\h{O}_C(-2))
 .$$ Note that
 $$K_S\otimes\h{O}_S(C)\otimes\h{O}_S(-2)=|(8;3,2^4)-3(3;1^5)|=|(-1;0,-1^4)|$$
 and by adjunction
 $$(K_S\otimes\h{O}_S(C)\otimes\h{O}_S(-2))\otimes\h{O}_C=K_C\otimes\h{O}_C(-2),
 $$ we have $h^0(K_C\otimes\h{O}_C(-2))=0$.
 Hence  $h^0(C,\h{O}_C(2))=13$ and the claim follows. 
 
 \vni
 (iii) By Claims (i) and (ii) we have $$h^0(C, N_{C|\PP^4})=h^0(C, N_{S|\PP^4}\otimes\h{O}_C)+h^0(C, N_{C|S})=52,$$ and hence
$$\dim\Ff\le h^0(C, N_{C|\PP^4})=52.$$
\end{proof}

\begin{rmk}
\begin{itemize}
\item[(i)] We were not able to determine the irreducibility of $\Hh_{13,14,4}$. We only showed that there is a generically reduced component of the expected dimension $\Xx(13,14,4)=52$. If one can find an irreducible 
family whose dimension exceeds $52$, this will lead us to the reducibility of $\Hh_{13,14,4}$.
\item[(ii)] One can check that there is no component of $\Hh_{13,14,4}$ whose smooth element lies on a cubic scroll or a normal cubic cone, utilizing standard computation used in the proof Theorem \ref{r+10}, e.g. \eqref{scrolldegree}, \eqref{scrollgenus},  \eqref{cone}, \eqref{conevertex}.
\item[(iii)] On the other hand, there is  only one component of $\Hh_{13,14,4}$ whose general element lies
on a quartic surface, which is the one we described in Proposition \ref{d13r4}. By  \cite[Theorem 4.1]{rmi}, curves on a singular del Pezzo is a specialization of curves on a smooth del Pezzo surface.  By a formula similar to \eqref{cone}, \eqref{conevertex} for curves on a cone over an elliptic curve, there is no smooth curve on an elliptic cone. 
\item[(iv)] The author tried without success to determine the gonality of the curve $C\in(11;5,3^7)$ on $\PP^2_8\stackrel{(4;2,1^7)}{\longrightarrow} S\subset\PP^4$, a smooth curve in $\PP^4$ with $(d,g)=(13,14)$ whereas $\deg S=(4;2,1^7)^2=5$. If $C$ is $6$-gonal, this would imply that there is another component containing $C$, by semi-continuity of gonality. By \csi, a general curve $C\in\Ff$ in Proposition \ref{d13r4} is $5$-gonal. 
\item[(v)] Through a tedious computation, it can also be shown that 
there is no component of $\Hh_{13,14,4}$ whose smooth element lies on a Bordiga surface in $\PP^4$. 
\end{itemize}
\end{rmk}

\vni
We finish this section with the following theorem concerning the irreducibility of $\Hh_{2r+5,r+10,r}$ 
for $r=3$.
\begin{thm}\label{$g=13$} The Hilbert scheme $\h{H}_{11,13, 3}$ is irreducible of the expected dimension $44$.
\end{thm}
\begin{proof} 
A smooth $C\subset \PP^3$ with $(d,g)=(11,13)$ does not lie on a quadric surface because 
there is no integer solution for the equations;
 $$a+b=11=d, ~~(a-1)(b-1)=13=g.$$
We also see that $C$ lies on at most one cubic surface, by Bezout. 
Note that smooth space curves of degree $d$ and genus $g$ on a smooth cubic surface form a finite union of locally closed irreducible family in $\h{H}_{d,g,3}$ of dimension $d+g+18$ if $d\ge 10$ by \cite[Proposition B.1]{Gruson}. Since $d+g+18< 4d$ for $(d,g)=(11,13)$, the  family of curves lying on 
smooth cubics does not constitute a component, even though $\Hh_{11,13,3}\neq\emptyset$ as we have seen in Theorem \ref{sub6x}.
\vni
On the other hand, it is possible that there might exist  a component of $\mathcal{H}_{11,13,3}$ whose general element lies only on a singular cubic surface. However, one may argue that no such component exists as follows. 
Note that every singular cubic surface $S\subset \PP^3$ is one of the following three types.
\vskip 4pt
(i) $S$ is a normal cubic surface with some double points only.

(ii) $S$ is a normal cubic cone.

(iii) $S$ is not normal, which may possibly be a cone. 

\vni
For the case (i), let $S$ be a normal cubic surface which is not a cone. By a result due to Brevik \cite[Theorem 5.24]{Brevik}, every curve on $S$ is a specialization of curves on a smooth cubic surface. Therefore we are done for the case (i). 

\vni
For the case (ii), we let $C$ be a smooth curve of degree $d$ and genus $g$ on a normal cubic cone $S$. Recall that 

(a) $g=1+d(d-3)/6-2/3$ if $C$ passes through the vertex of $S$ 

(b) $g=1+d(d-3)/6$, otherwise

\noindent
which can be found in \cite[Proposition 2.12]{Gruson} as an application of C. Segre formula.
However $(d,g)=(11,13)$ satisfies neither of the above. Alternatively one may check  directly that
there is no smooth curve of degree $d=11$ and genus $g=13$ on a cone $S$ over a smooth plane cubic
$E\subset\PP^2$, realizing $S$ as the image of the ruled surface $\PP(\Oo_E\oplus\Oo_E(3))$.
\vni
(iii) Let $C$ be a smooth curve of degree $d$ and genus $g$ on a non-normal cubic surface $S$. Recall that if $S$ is a cone, then $S$ is a cone over a singular plane cubic, in which case $S$ is a projection of a cone $S'$ over a twisted cubic in a hyperplane in  $\PP^4$ from a point outside $S'$. Furthermore,  the minimal desingularisation $\tilde S$ of $S'$ is isomorphic to the ruled surface 
$$\mathbb{F}_3=\PP (\mathcal{O}_{\PP^1}\oplus \mathcal{O}_{\PP^1}(3)).$$ 
If $S$ is not a cone,  $S$ is a projection of a rational normal scroll $$S''\cong\tilde S\cong\mathbb{F}_1=\PP (\mathcal{O}_{\PP^1}(1)\oplus \mathcal{O}_{\PP^1}(2))\subset \PP^4$$  from a point outside $S''$. 
In both cases, we have $\text{Pic}~ \tilde S=\mathbb{Z}h\oplus \mathbb{Z}f \cong\mathbb{Z}^{\oplus 2}$, where $f$ is the class of a fiber of $\tilde S \rightarrow \PP^1$ and $h=\pi^*(\mathcal{O}_S(1))$ with $\tilde S\stackrel{\pi}{ \rightarrow} S$. Note that $h^2=3, f^2=0$, $h\cdot f=1$ and $K_{\tilde S}\equiv -2h+f$. Denoting by $\tilde C\subset\tilde S$ the strict transformation of the curve $C\subset S$, we set $k:=(\tilde C \cdot f)_{\tilde S}$. We have $\tilde C \equiv kh+(d-3k)f=kh+(11-3k)f$.
By adjunction formula, it follows that
$$g=13=\frac{(2\cdot 11-3k-2)(k-1)}{2},$$ which does have an integer solution and thus we are done with the case (iii).

\vni
Since $\Hh_{11,13,3}\neq\emptyset$, we take a component $\Hh\subset\Hh_{11,13,3}$.
Let $C$ be a general element of $\Hh$. We recall that there is a smooth curve of degree $d=11$ and genus $g=13$ lying on a smooth 
irreducible quartic surface by a result of Mori \cite[Theorem 1]{Mori}. We may assume that  $C$ does not lie on a quadric or a cubic. 

\vni
Since every smooth curve $C$ of genus $g=13$ of degree $d=11$ in $\PP^3$ lies on at least 
\begin{equation}\label{quartic11}
h^0(\PP^3, \h{O}(4))-h^0(C,\h{O}(4))=35-(4\cdot 11-g+1)=3
\end{equation}
 independent quartics, we see that $C$ is 
residual to a curve $D\subset\PP^3$ of degree $e=5$ and genus $h=1$ in the complete intersection of two (irreducible) quartics
by the basic relation
\begin{equation*}
2(g-h)=(s+t-4)(d-e).
\end{equation*}

\vni
Consider the locus
$$\Sigma\subset\mathbb{G}(1,\PP(H^0(\PP^3,\h{O}(4))))=\mathbb{G}(1,34)$$
of pencils of quartic surfaces whose base locus consists of a curve $C$ of degree $d=11$ and genus $g=13$ and an elliptic quintic $D$ which are directly linked via complete intersection of quartics in the pencil, together with the two obvious maps 
\[
\begin{array}{ccc}
\hskip -48pt\mathbb{G}(1,34)\supset\Sigma&\stackrel{\pi_C}\dashrightarrow~~
\h{I}_4\subset\h{H}_{11,13,3}\\
\\
\dashdownarrow\vcenter{\rlap{$\scriptstyle{{\pi_D}}\,$}}
\\ 
\\
\h{H}_{5,1,3}
\end{array}
\]
where $\h{I}_4$ is the image of $\Sigma$ under $\pi_C$.
An elliptic quintic $D\subset\PP^3$ lies on at least 
\begin{equation}\label{D4}
h^0(\PP^3,\h{O}(4))-h^0(D,\h{O}(4))=35-(4\cdot 5-1+1)=15
\end{equation} independent quartics. 
Note that $D\in\h{H}_{5,1,3}$  is  directly linked 
to $F\in\h{H}_{4,0,3}$ via complete intersection of two cubics. 
Since $F$ is rational,  $h^0(\PP^3,\h{I}_F(2))=1$; otherwise $F$ is elliptic. From $0\rightarrow \h{I}_F(2)\rightarrow \h{O}_{\PP^3}(2)\rightarrow \h{O}_F(2)\rightarrow 0$, we have  $h^1(\PP^3,\h{I}_F(2))=0$. From the  well-known relation
$$h^1(\PP^3,\h{I}_{C_1}(m))=h^1(\PP^3,\h{I}_{C_2}(s+t-4-m))$$
when $C_1$ and $C_2$ are directly linked by complete intersection of two surfaces of degrees $s$ and $t$,
we have \begin{equation}\label{4=1}h^1(\PP^3,\h{I}_C(4))=h^1(\PP^3,\h{I}_D)=h^1(\PP^3,\h{I}_F(2))=0
\end{equation}
\begin{equation}\label{D1}
h^1(\PP^3,\h{I}_D(4))=h^1(\PP^3,\h{I}_F(-2))=0.
\end{equation}
 By \eqref{D1} and \eqref{D4},  we have $h^0(\PP^3,\Ii_D(4))=15$ and therefore $\pi_D$ is generically surjective with fibers open subsets of $\mathbb{G}(1,14)$.  Since $\dim\h{H}_{5.1.3}$ is known to be irreducible (cf.\cite{E1} or \cite[page 54]{H1}),  $\Sigma$ is irreducible and $$\dim\Sigma=\dim\mathbb{G}(1,14)+\dim\h{H}_{5,1,3}=26+4\cdot 5=46.$$ On the other hand, by \eqref{4=1} and \eqref{quartic11} every smooth curve $C\in \h{I}_4$ lies on exactly $3$ independent quartics and hence $\pi_C$ is generically surjective with fibers open subsets of $\mathbb{G}(1,2)$. Finally it follows that $\h{H}_{11.13.3}$ is irreducible of dimension
$$\dim\Sigma -\dim{\mathbb{G}}(1,2) =4\cdot11.$$
\end{proof}

\section{Irreducibility of $\HL{2r+6,r+11, r}$}
In this section we determine the irreducibility of $\HL{2r+6,r+11, r}$ for $7\le r\le11$. Recall that 
$\HL{2r+6,r+11, r}\neq\emptyset$ only if $3\le r\le11$ by Theorem \ref{sub6x} (c). The following result employs a similar method we used in Theorem \ref{r+10}. However, there appear some subtleties which need some caution. 
\begin{thm}\label{2r+6} $\HL{2r+6,r+11,r}$ is irreducible if $r=8,9,11$ and reducible if $r=10$.
\end{thm}
\begin{proof} 
Given a component $\h{G}\subset{\w{\h{G}}}_\h{L}\subset\h{G}^r_{2r+6}$, choose a general $\h{D}\in\h{G}$ and set $\h{E}=\h{D}^\vee=g^4_{14}$. We claim that $\h{E}$ is base-point-free and birationally very ample.
If $\h{E}$ is compounded, we apply Lemma \ref{easylemma1} to deduce that  $\h{E}^\vee=\h{D}$ is not very ample. $\h{E}$ is base-point-free since $\pi(13,4)=18<g=r+11$.

\vni
In case $8\le r\le11$, $$\pi_1(14,4)=18<g=r+11\le p_a(\ce)\le\pi(14,4)=22 ~~ \hskip 12pt$$ hence $\ce\subset\PP^4$ is an extremal or a nearly extremal curve lying on a cubic surface  $S\subset\PP^4$. 

\vni
{\bf (A)}
We first assume $S\cong\mathbb{F}_1\cong\PP^2_1$.
Set $$\delta=\delta(p_a(\ce)):=p_a(\ce)-g=p_a(\ce)-(r+11)\ge 0$$
and we list up all the pairs $(a,b)\in\mathbb{N}\times\mathbb{Z}$ satisfying the equations
(\ref{scrolldegree}), (\ref{scrollgenus}) in Remark \ref{minimal}; $n=3, d=14$ with  $19\le p_a(\ce)\le22$; 
\begin{equation}\label{g=14}
\begin{cases}
p_a(\ce)=22; ~(a,b)=(5,-1), &8\le r\le 11\\

p_a(\ce)=21; ~(a,b)=(4,2), &8\le r\le 10\\

p_a(\ce)=20; ~(a,b)=(6,-4), &8\le r\le 9\\

p_a(\ce)=19; ~\mathrm{no ~solution}\\

p_a(\ce)=18; ~\mathrm{no ~solution}
\end{cases}
\end{equation}

\vni
Denoting the three linear systems on $S$ by 
$$\h{M}_1=|5H-L|, \h{M}_2=|4H+2L|, \h{M}_3=|6H-4L|, $$
\vni
we have the following possibilities;
\[
\begin{cases}
r=11; &\ce\in\h{M}_1, ~\delta=0 \\

r=10; &\ce\in \h{M}_1, ~\delta=1 ~\textrm{ or }~ \ce\in\h{M}_2, ~\delta=0\\

r=~9; &\ce\in \h{M}_1, ~\delta=2 ~\textrm{ or }~  \ce\in\h{M}_2, ~\delta=1 ~\textrm{ or }~ \ce\in\h{M}_3, ~\delta =0\\

r=~8; &\ce\in \h{M}_1, ~\delta=3 ~\textrm{ or }~ \ce\in\h{M}_2, ~\delta=2 ~\textrm{~ or ~}~ \ce\in\h{M}_3, ~\delta =1.
\end{cases}
\]

\vni
We need to check, given a fixed $r$ and for each possible value $p_a(\ce)$,  what linear systems $\h{M}_i$ such that $\ce\in\h{M}_i$ contribute to a component $\h{G}\subset\w{\h{G}}_\h{L}\subset\h{G}^r_{2r+6}$ whose general element is complete and very ample. {\bf In other words, we are asking;  for what $\h{M}_i$ with $\ce\in\h{M}_i$, is  the linear system 
$$|D|:=|K_{\PP^2_{\delta+1}}+\widetilde{C}_{\Ee}-\w{H}|_{{|\widetilde{C}_{\Ee}}}$$ very ample ?} 

\vni

\vni Here $\widetilde{C}_{\Ee}$ is the proper transformation of $\ce\in\Sigma_{{\h{{M}}_i},\delta}$ under the  blow up of $\mathbb{F}_1\cong\PP^2_1\cong S$ at $\delta=p_a(\ce)-g=p_a(\ce)-(r+11)$ nodal points of $\ce$,  $\mathbb{P}^2_{\delta+1}\stackrel{\pi}{\rightarrow}\mathbb{F}_1\cong S$ and $\w{H}=\pi^*\h{O}_S(1)$.

\vni
\begin{itemize}
\item[(i)] $p_a({\ce})=22\Longleftrightarrow \ce\in\h{M}_1$; every $8\le r\le 11$ is possible and $\delta=11-r$.
Since 
$$\ce\in\h{M}_1=|5H-L|=|5(e+2f)-f|=|9l-4e|=(9;4),$$ 
we have
  $$\widetilde{C}_{\Ee}\in (9,4,2^{11-r})$$ on  $\mathbb{P}^2_{\delta+1}=\PP^2_{12-r}$ -- which is a special case of the linear system  (\ref{nonic}) in Theorem \ref{sub6x} -- and 
  \begin{align*}
|K_{\PP^2_{12-r}}+\widetilde{C}_{\Ee}-\w{H}|&=|(-3;-1^{12-r})+(9;4,2^{11-r})-(2,1, 0^{12-r})|\\&=|(4;2,1^{11-r})|,
\end{align*}
which is very ample\footnote{$|(4;2,1^{11-r})|$ is very ample  not only for $8\le r\le 11$ but also for $4\le r\le 11$.}. Hence we have 
an irreducible family of nodal curves lying on a Hirzebruch surface $\mathbb{F}_1$ -- parametrized by the irreducible Severi variety $\Sigma_{\h{M}_1,11-r}$ -- whose
non-singular model
$$\widetilde{C}_{\Ee}\subset\PP^2_{12-r}\lhook\joinrel\xrightarrow{(4;2,1^{11-r})}\PP^r$$ 
 is embedded into $\PP^r$ by the linear system 
$|K_{\PP^2_{12-r}}+\widetilde{C}_{\Ee}-\w{H}|_{|\widetilde{C}_{\Ee}}$
as a  linearly normal curve of degree  
$$d=(K_{\PP^2_{12-r}}+\widetilde{C}_{\Ee}-\w{H})\cdot\widetilde{C}_{\Ee}=2r+6.$$
Therefore we may conclude that for every $8\le r\le 11$ and $\delta=p_a(\ce)-g=11-r$, 
there is a component $\h{G}_{\Mm_1,11-r}\subset\h{G}^r_{2r+6}$ associated with the Severi variety $\Sigma_{\h{M}_1,11-r}$, i.e $$\h{G}_{\Mm_1,11-r}:={\widetilde\Sigma}_{\h{M}_1,11-r}^\vee={\Sigma}_{\h{M}_1,11-r}^\vee/\Aut(S),  ~\mathrm{where}$$
$$\Sigma_{\h{M}_1,11-r}^\vee:=\{~|K_{\PP^2_{12-r}}+\widetilde{C}_{\Ee}-\w{H}|_{|{\widetilde{C}_{\Ee}}}~|~\ce\in\Sigma_{\h{M}_1,\delta}\}.$$
Note that a general $\ce\in\Sigma_{\h{M}_1,11-r}$ is $5$-gonal by \csi.
\vni
\item[(ii)] $p_a({\ce})=21\Longleftrightarrow \ce\in\h{M}_2$\,; only $8\le r\le 10$ is possible and $\delta=10-r$.
\vni On $S\cong\PP^2_1$, we have
$$\ce\in\h{M}_2=|4H+2L|=|4(e+2f)+2f|=|10l-6e|=(10;6).$$ 
Thus we have $$\widetilde{C}_{\Ee}\in (10;6,2^{10-r})$$
after  blowing up $\PP^2_1$ at $\delta=10-r$ nodal points of $\ce$ and 
\begin{align*}|K_{\PP^2_{11-r}}+\widetilde{C}_{\Ee}-\w{H}|&=|(-3;-1^{11-r})+(10;6,2^{10-r})-(2,1, 0^{10-r})|\\&=|(5;4,1^{10-r})|.
\end{align*}
Note that both $|\widetilde{C}_{\Ee}|=(10;6,2^{10-r})$ and $|K_{\PP^2_{11-r}}+\widetilde{C}_{\Ee}-\w{H}|=(5;4,1^{10-r})$ are  very ample only if $r=10$ by \cite{sandra}. Hence only for the case $r=10$, there is an irreducible family
$\h{G}_{\Mm_2,0}\subset\w{\h{G}}_\h{L}\subset\h{G}^r_{2r+6}$ 
associated with the Severi variety $\Sigma_{{\h{M}_2},10-r}=\Sigma_{\h{M}_2,0}$, i.e.
$$\h{G}_{\Mm_2,0}={\w{\Sigma}}_{\h{M}_2,0}^\vee.$$
Therefore 
we obtain another irreducible family of (smooth) curves on $\mathbb{F}_1$  which is embedded into $\PP^{10}$ by  
$|K_{\PP^2_1}+{\ce}-\h{O}_S(1)|=(5;4)$ as a linearly normal curve with $(d,g)=(2r+6, r+11)=(26,21)$. By \csi, a general $\ce\in\Sigma_{{\h{M}_2},10-r}=\Sigma_{\h{M}_2,0}$ is $4$-gonal. 
For $r=10$,  we have $$\dim\h{G}_{\Mm_1,11-r}=\dim\h{M}_1-1-\dim\textrm{Aut}(S)=37<\dim\h{G}_{\Mm_2,0}=38,$$
thus $\h{G}_{\Mm_1,1}$ and $\h{G}_{\Mm_2,0}$ are two distinct components of $\w{\h{G}}_\h{L}$ by semi-continuity of gonality.
\vni
\item[(iii)] $p_a({\ce})=20\Longleftrightarrow \ce\in\h{M}_3=|6H-4L|$\,; only $8\le r\le 9$ is possible.
By the same computation as before, 
 after blowing up $\PP^2_1$ at $\delta(p_a(\ce))=p_a(\ce)-g=9-r$ nodal points of $\ce$, we get  on $\PP^2_{10-r}$
 $$\widetilde{C}_{\Ee}\in (8;2,2^{9-r})=(8;2^{10-r})$$  and
\begin{align*}
|K_{\PP^2_{10-r}}+\widetilde{C}_{\Ee}-\w{H}|&=(-3;-1^{10-r})+(8;2^{10-r})-(2,1, 0^{9-r})\\&=(3;0,1^{9-r}),
\end{align*}
which is not very ample. Hence $|D|=|K_{\PP^2_{\delta+1}}+\widetilde{C}_{\Ee}-\w{H}|_{\widetilde{C}_{\Ee}}$
is not very ample.\footnote{For every $i=1,2,3$ with ${\ce}\in\h{M}_i$, the restriction map 
$$\rho: H^0(\PP^2_{\delta +1}, \h{O}(K_{\PP^2_{\delta+1}}+\widetilde{C}_{\Ee}-\w{H}))\rightarrow H^0(\widetilde{C}_{\Ee}, \h{O}(K_{\PP^2_{\delta+1}}+\widetilde{C}_{\Ee}-\w{H})\otimes\Oo_{\widetilde{C}_{\Ee}})$$ 
is surjective  since $h^1(\PP^2_{\delta +1}, \h{O}(K_{\PP^2_{\delta+1}}-\w{H}))=0$. }
\end{itemize}

\vni
Summing up, all the possible components of $\w{\h{G}}_\h{L}\subset\h{G}^r_{2r+6}$  such that $\ce$ lies on a smooth cubic scroll in $\PP^4$ are the following.
\[
\begin{cases}
r=11; &\h{G}_{\Mm_1,0}\\
r=10;  &\h{G}_{\Mm_1,1}, ~\h{G}_{\Mm_2,0} \\
r=~9; &\h{G}_{\Mm_1,2} \\
r=~8; &\h{G}_{\Mm_1,3}
\end{cases}
\]

\vni
{\bf (B)} We assume that $\ce$ lies on a rational normal cone $S\subset\PP^4$.  We borrow  the notation from Remark \ref{minimal} (iii) and take  $n=3$, $d=14$ there. Let $\widetilde{C}_{\Ee}$ be the proper transformation
of $\ce$ under the minimal desingularization $\FF_3\to S\subset\PP^4$ given by $|C_0+3f|$.
Again by \eqref{conevertex} and \eqref{cone} in Remark \ref{minimal}, we have 

\[(k,m,p_a(\w{C}_{\Ee}))=
\begin{cases} (4,2,21)\\
(3,5,17)\\
(2,8,10).
\end{cases}
\] 
For the second and the third triples, $p_a(\w{C}_{\Ee})<g=r+11$ for any $8\le r\le 11$, hence only the first triple $(k,m,p_a(\w{C}_{\Ee}))=(4,2,21)$ is possible.
\vni
{\bf (B-1):} If $r=11$, $g(\w{C}_{\Ee})=g=r+11=22>p_a(\w{C}_{\Ee})$ a contradiction. Thus there is no curve  (singular or smooth) $\ce$ with geometric genus $g=r+11=22$ and degree $d=14$ on the rational normal cone $S\subset\PP^4$.

\vni {\bf (B-2):} If $r=10$, $$p_a(\w{C}_{\Ee})=21=g=r+11=g(\ce)=g(\w{C}_{\Ee}),$$ hence $\w{C}_{\Ee}$ is smooth. On the other hand, the image $\ce$ of $\w{C}_{\Ee}$ under the 
morphism $\FF_3\to S\subset\PP^4$ is singular since $m=2$ and hence $g=g(\ce)<g(\w{C}_{\Ee})=21$, which is a contradiction. Again, there is no curve  (singular or smooth) $\ce$ with geometric genus $g=r+11=22$ and degree $d=14$ on the rational normal cone $S\subset\PP^4$.

\vni {\bf (B-3):} If $r=9$, $\w{C}_{\Ee}$ is singular with one double point; $\w{C}_{\Ee}\in|4C_0+14f|$.

\vni {\bf (B-4):} If $r=8$, $\w{C}_{\Ee}$ is singular with a double point and an additional singular (possibly infinitely near) point; $\w{C}_{\Ee}\in|4C_0+14f|$. 

\vni Note that we are under the condition $(k,m,p_a(\w{C}_{\Ee}))=(4,2,21)$, thus ${C}_\Ee$ is singular at the vertex of the cone $S\subset\PP^4$. Hence at least one singular point of $\w{C}_\Ee$ lies on the section $C_0$. However, in order to avoid complicated computation and unnecessary extra notation, we will show that a general element in the Severi variety of curves of geometric genus $g=r+11 \,(8\le r\le 9)$ in  the linear system $\h{N}=|4h+14f|$ on $\mathbb{F}_3$  does not have very ample $|K_{\widehat{C}}(-1)|$ where $\widehat{C}$ is the normalization of $\w{C}_{\Ee}\subset\mathbb{F}_3$. This simplification can be justified because being very ample is an open condition.

\vni
We did a similar job in the proof (part (B-2)) of Theorem \ref{r+10} for one node case. Even though all the arguments are parallel to part (B-2) in Theorem \ref{r+10}, we provide an outline of a proof for the two nodes case for  convenience of readers. 
\vni
We assume that $\w{C}_{\Ee}\subset\mathbb{F}_3$ has two nodes outside $C_0$ as its only 
singularities. Let $e_i\, (\,i=1, 2)$ be exceptional divisors and let
$f_i \, (i=1, 2)$ be the fibers containing the two nodal points of $\w{C}_{\Ee}$. After resolving
the two nodes we get a smooth curve ${\widehat{C}}\subset\mathbb{F}_{3,2}$ on the surface $\mathbb{F}_3$ blown up at two points.
We have
$${\widehat{C}}\in|4C_0+14f-\sum2e_i|.$$
Set
\begin{align*}
\h{M}:&=|{\widehat{C}}+K_{{\mathbb{F}_{3,2}}}-(C_0+3f)|\\&=|(4C_0+14f-\sum2e_i)+(-2C_0-5f+\sum e_i)-(C_0+3f)|\\&=|C_0+6f-\sum e_i|.
\end{align*}
The injectivity of  the restriction map 
$$\rho: H^0({{\mathbb{F}_{3,2}}},\h{O}(\h{M}))\longrightarrow H^0({\widehat{C}},\h{O}(\h{M})\otimes\h{O}_{{\widehat{C}}})$$ follows from 
$$h^0(\mathbb{F}_{3,2},\h{O}_{{\mathbb{F}_{3,2}}}(\h{M}-{\widehat{C}}))=h^0(\mathbb{F}_{3,2},\h{O}_{{\mathbb{F}_{3,2}}}(-(3C_0+8f-\sum e_i)))=0.$$
\vni
The surjectivity of the restriction map $\rho$ follows from Castelnuovo genus bound for a curve of degree $\widehat{C}\cdot\h{M}=(4C_0+14f-\sum 2e_i)\cdot (C_0+6f-\sum e_i)=22$ in $\PP^{s\ge 9}$ unless $\rho$ is surjective; $\pi(22,9)=18<g=19$.
Therefore we have 
$$\mathrm{Im}(\rho)= H^0(\widehat{C},\h{O}(\h{M})\otimes\h{O}_{\widehat{C}})$$
and the restriction map $\rho$ is an isomorphism. 
Denoting by $\w{f}_i$  the proper transformation of $f_i$ under the blow up, 
$$\tilde{f}_i^2=-1, \tilde{f}_i\cdot e_i=1,  C_0\cdot\tilde{f}_i=1, {\widehat{C}}\cdot \tilde{f}_i=2, 
\Mm\cdot\tilde{f}_i=(C_0+6f-\Sigma e_i)\cdot \tilde{f}_i=0.$$
Hence the morphism $\psi$ given by $\h{M}=|{\widehat{C}}+K_{{\mathbb{F}_{3,2}}}-(C_0+3f)|$ contracts  $(-1)$ curves $\tilde{f}_i$ and the image curve $\psi ({\widehat{C}})\subset\PP^5$ 
acquires singularities.
Since the complete linear system $\mathrm{Im}(\rho)$ maps ${\widehat{C}}$ onto $\psi ({\widehat{C}})\subset\PP^5$ with at least $2$ singular points, 
$\h{M}\otimes\h{O}_{{\widehat{C}}}=|K_{{\widehat{C}}}(-1)|$ is not very ample.
\vni
{\bf Conclution:} Part (B) shows that there is no component of $\HL{2r+6,r+11,r}$ whose general element lies on a rational normal cone and this finishes the proof.
\end{proof}

\vni
We would like to push forward one step further to determine the reducibility of $\HL{2r+6,r+11,r}$ for $r=7$. In this case, the curve $\ce\subset\PP^4$ may not be nearly extremal; $\pi_1(14,4)=18=r+11=g$. Therefore we cannot use a similar method which we used in Theorem \ref{2r+6}.  We instead estimate the dimension
of a component containing a certain (natural) irreducible family of curves by computing the normal bundle corresponding to a general element of the family. This is the point where a certain subtlety arises.
One drawback of this method is that we cannot identify all the irreducible components as completely as we did before, unless we carry out lengthy and cumbersome analysis of all the possibilities with our rather imperfect techniques. 
\vni

\begin{prop}\label{r+117}$\HL{2r+6,r+11,r}, r=7$ is reducible.\end{prop}
\begin{proof} We retain almost all the notations which were used throughout this section.  Choose a 
general $\h{D}\in\h{G}\subset\w{\h{G}}\subset\h{G}^r_{2r+6}$ and set $\h{E}=\h{D}^\vee=g^4_{14}$. 
\vni
{\bf Claim:} $\h{E}$ is birationally very ample with possibly non-empty base locus $\Delta$ such that $0\le\deg\Delta\le 1$. If $\h{E}$ is compounded inducing $C\stackrel{\eta}{\rightarrow} E$, we  apply Clifford's theorem to the complete linear system $\w{\h{E}}$ on the base curve $E$ such that   $\eta^*{\w{\h{E}}}=|\h{E}-\Delta|$. By Lemma \ref{easylemma1}, $\Dd=\h{E}^\vee$ is not very ample, a contradiction. 
If $\deg\Delta\ge 2$,  we have $\pi(12,4)=15<g=r+11=18$. Hence $\deg\Delta\le 1$ and $\h{E}$ is birationally very ample. 

  \vni
{\bf (A) $\deg\Delta=1$:} Set $\h{E}':=|\h{E}-\Delta|=g^4_{13}$. Since $\pi(13,4)=18=g=r+11$,  the curve $C_{\h{E}'}\subset\PP^4$ induced by $\h{E}'$ is an extremal (smooth) curve lying on
 a cubic surface $S\subset\PP^4$.  We may assume that $S$ is smooth by Remark \ref{specialization}.
  By solving \eqref{scrolldegree}, \eqref{scrollgenus}, we have either $C_{\h{E}'}\in|5H-2L|$ or $C_{\h{E}'}\in|4H+L|$.   
 Noting that $C_{\Ee'}\cong C$, we have $$|K_{C_{\h{E}'}}-\h{E}'-\Delta|=|K_{C_{\h{E}'}}-\h{E}|=|K_C-\h{E}|=\h{D},~~ |K_{C_{\h{E}'}}-\h{E}'|=g^8_{21}.$$ Thus our $C\subset\PP^7$ with $(d,g)=(20,18)$ is the image of the projection with center at $\Delta$
  from the (smooth) image curve of the morphism induced by $|K_{C_{\h{E}'}}-\h{E}'|$. 
  
  \vni 
  {\bf (A-1)} $C_{\h{E}'}\in|5H-2L|$: We have $|K_S+C_{\h{E}'}-H|=|2H-L|=|3l-e|$; cf.   \eqref{picid}.
  Let $T\subset\PP^8$ be the image of the embedding $S\stackrel{\phi}\longrightarrow T\subset\PP^8$ induced by $|2H-L|$ and set 
 $C'=\phi(C_{\Ee'})\subset T$.  Note that $T$ is a smooth del-Pezzo surface of degree $(2H-L)^2=8$. Since there is only finitely many lines on a smooth del-Pezzo of degree $8$, projecting $C'$ into $\PP^7$  from a general point $p\in C'\subset\PP^8$ induces an isomorphism onto a curve $C\subset\PP^7$, $\deg C=20$.
  Set
 $$\h{F}:=\{ ~g^8_{21}~ 
 | ~g^8_{21}=|K_{C_{\h{E}'}}-\h{E}'|,\, \h{E}'=\h{E}\setminus\Delta,\, \h{E}=\h{D}^\vee,\,\h{D}\in\h{G}\}\subset\h{G}^8_{21}.$$
 By construction we have $$\dim\h{F}=\dim|5H-2L|-\dim\Aut(S)=32.$$ Setting
 $\h{F}_0:=\{~|g^8_{21}-p| ~| ~g^8_{21}\in\h{F},~ p\in C'=\phi(C_{\Ee'})\}\subset\Gg^7_{20},$ we have
 $$\dim\h{F}_0=\dim\h{F}+1=33>\lambda(14,18,4)=\lambda(20,18,7)=29.$$
 \vni
 {\bf (A-2)} $C_{\h{E}'}\in|4H+L|$: Both $|4H+L|$ and $|K_S+C_{\h{E}'}-H|=|H+2L|$ are very ample. 
 Under the  isomorphism $S\stackrel{|H+2L|}\longrightarrow T\subset\PP^8$, we have $$\deg T=(H+2L)^2=7$$
 and hence $T$ is a rational normal surface scroll. Since $(H+2L)\cdot L=1$ and $(4H+L)\cdot L=4$, the image in $T$ of the ruling $|L|$ on $S$ cut out a $4$-secant line on the image curve of $C_{\Ee'}$.  Hence the projection with center at $\Delta$, i.e. the morphism induced by $\Dd=\h{E}^\vee=\h{E}'^\vee-\Delta$ is not very ample.
 \vni
 {\bf (B) $\Delta =\emptyset$ :} Since 
 $\pi_1(14,4)=18=r+11=g$, $\ce\subset\PP^4$ may lie either on a cubic surface or on a quartic surface. 
 \vni
 {\bf (B-1)} Assume $\ce$ lies on a smooth cubic scroll $S\subset\PP^4$. We have
 $$g=\pi_1(14,4)=18\le p_a(\ce)\le 22=\pi(14,4).$$  The pairs $(a,b)\in\NN\times\ZZ$
 satisfying \eqref{scrolldegree},\eqref{scrollgenus} for $18\le p_a(\ce)\le 22$ are already listed in \eqref{g=14} (in the proof of Theorem \ref{2r+6}). 
 
\vni 
We take $\ce\in \h{M}_1=|5H-L|$, where $p_a(\ce)=22$, $\delta=\delta(p_a(\ce))=11-r=4$.
The proper transform $\w{C}_{\Ee}\in |(9;4,2^4)|\in \textrm{Pic}(\PP^2_5)$ under the blow up 
$\mathbb{P}^2_{\delta(p_a(\ce))+1}\stackrel{\pi}{\longrightarrow}\mathbb{F}_1\cong S$
is embedded into $\PP^7$ as a curve of degree $d=20$ by the very ample $(4;2,1^4)$; cf. part (A)--(i) in the proof of Theorem \ref{2r+6}. 

\vni
 Therefore the irreducible family $\h{F}_1\subset\h{G}^7_{20}$ arising this way, i.e.,
 $$\Ff_1={\widetilde\Sigma}_{\h{M}_1,11-r}^\vee=\{~|K_{\PP^2_{12-r}}+\widetilde{C}_{\Ee}-\w{H}|_{|{\widetilde{C}_{\Ee}}}~|~\ce\in\Sigma_{\h{M}_1,\delta}\}/\Aut(S)$$
  has dimension;
  \begin{align*}
  \dim\h{F}_1&=\dim{\widetilde\Sigma}_{\h{M}_1,11-r}^\vee=\dim{\Sigma}_{\h{M}_1,11-r}^\vee/\Aut(S)\\&=\dim\Sigma_{\h{M}_1,11-r}-\dim \textrm{Aut}(S)=\dim|5H-L|-\delta-6=34.
  \end{align*}
  \vni In the proof of Theorem \ref{2r+6}, there are several other linear systems $\Mm_i$ ($i=2,3$) such that $\ce\in\Mm_i$. However we already verified that $\ce^\vee\subset\PP^7$ (the residual curve of $\ce$) is not smooth.

 \vni
 {\bf (B-2)} Assume $\ce$ lies on a smooth del Pezzo  surface $S\subset\PP^4$ and let $\ce\in (a;b_1,\cdots,b_5)$. Solving
 $$\deg\ce=3a-\sum^5_{i=1} b_i=14, ~\ce^2=a^2-\sum^5_{i=1} b_i^2=2g-2-K_S\cdot\ce=48$$ we get
 \begin{equation}\label{del}
 (a;b_1,\cdots,b_5)=(10;4,3^4) \textrm{ or } (11,4^4,3).
 \end{equation}
 Note that $(10;4,3^4) \textrm{ and } (11,4^4,3)$ are very ample. Also note that a general member 
 in $(10;4,3^4)$ is isomorphic to a general member in $(11,4^4,3)$ via obvious quadratic transformation. 
 
 \vni Let $\ce\in(10;4,3^4)$. Since 
 $$|{\ce}+K_S-H|=(10;4,3^4)-2(3;1^5)=(4;2,1^4)$$ is very ample, there is another irreducible family 
 $\h{F}_2\subset\h{G}^7_{20}$ whose  residual series $\Dd^\vee=\Ee$ of a general member $\h{D}\in\h{F}_2$ gives rise to a smooth curve $\ce$ on a del Pezzo  surface in $\PP^4$, where $\h{E}^\vee=\h{D}$ induces an embedding into $\PP^r$ with
 $r=\dim|(4;2,1^4)|=7$, i.e.
 $$\Ff_2=\{\,g^7_{20}=\Dd \,|\, \Ee=\Dd^\vee,\,\ce\subset S\subset\PP^4,  S\mathrm{ ~is ~ a ~smooth ~del ~Pezzo} \}.$$

\vni
Since a smooth del Pezzo surface in $\PP^4$ is a complete intersection of two quadric hypersurfaces, which in turn corresponds to an element of the Grassmannian $\GG(1,\PP(H^0(\PP^4,\Oo(2)))$, an easy dimension count yields
$$\dim\h{F}_2=\dim|(10;4,3^4)|+\dim\mathbb{G}(1,14)-\dim\Aut(\PP^4)=33.$$
 
 \vni
 So far we have the following three irreducible families consisting of very ample linear series' inside $\h{G}^7_{20}$ having dimensions greater than $\lambda(20,18,7)=\lambda(14,18,4)=29$.
 
 \vni
 \begin{enumerate}
 \item[(0)] $\h{F}_0$: $\dim\h{F}_0=33$; for $\Dd\in\Ff_0$, $\h{E}=\Dd^\vee$ has non-empty base locus, $C_{\h{E}'}\subset\PP^4$ induced by the moving part $\h{E}'$ of $\h{E}$ is  a smooth curve of degree $13$ lying on a rational normal scroll $S$. $C_{\h{E}'}$ is $5$-gonal. $\h{E}'^\vee$ is very ample; Part (A).

 \item[(1)]  $\h{F}_1$: $\dim\h{F}_1=34$; $\ce\subset\PP^4$ induced by $\h{E}=\Dd^\vee$ is singular lying on a rational normal scroll $S$ and  $\ce$ is $5$-gonal. 
 
 \item[(2)] $\h{F}_2$: $\dim\h{F}_2=33$; $\ce\subset\PP^4$ induced by $\h{E}=\Dd^\vee$ is a smooth curve lying on a smooth del Pezzo  surface in $\PP^4$ and $\ce$ is $6$-gonal.\footnote{$\ce\in (10;4,3^4)$ being $6$-gonal is not totally clear. One may consult \cite[Theorem 2]{Sakai} for a quick check.} 
 
 \end{enumerate}

 
 

 
 \vni
 {\bf Observation:}
 \begin{itemize}
 \item[(i)] Note that $\h{F}_2$ is not in the boundary of $\h{F}_1$; by lower semi-continuity of gonality or by the fact that 
singular curves cannot be specialized to a smooth curve.


 \item[(ii)] A general element of $\Ff_2$ is $6$-gonal. Therefore $\Ff_2$ is not in the
boundary of $\Ff_0$.

 \item[(iii)] Since $\dim\Ff_0=\dim\Ff_2$, $\Ff_0$ is not in the boundary of $\Ff_2$.
\end{itemize}
  \vni 
Suppose that Hilbert scheme $\HL{2r+6,r+11,r}=\HL{20,18,7}$ is irreducible. By the observation,  the only possibility is that all three families $\Ff_0, \h{F}_1, \h{F}_2$ are contained in another irreducible family of dimension strictly bigger than $\dim\h{F}_1$.
 \vni 
 {\bf (C):}
 In what  follows
 we argue that this is impossible by showing that $\h{F}_2$ is 
dense in a  component $\w{\h{F}}_2\subset\h{G}^7_{20}$ and hence $\dim\w{\h{F}}_2=33$.

\vni
Take a general $\Dd\in\Ff_2$ and  set $\Ee=\Dd^\vee$. Let $\ce\in (10;4,3^4)$; cf. \eqref{del}. $\ce\subset S\subset\PP^4$ is smooth where $S$ is a smooth del Pezzo  surface. Let $c\in\HL{14,18,4}$ be the point corresponding to $\ce$ and let $\h{H}$ be a component cotaining $c$.
 We compute 
 $$h^0({\ce},N_{{\ce}|\PP^4})=\dim T_c{\h{H}}$$
 as follows.
 
 
 \vni
We consider the standard exact sequence
 $$0\rightarrow N_{{\ce}|S}\rightarrow N_{{\ce}|\PP^4}\rightarrow N_{{S|\PP^4}_{|{\ce}}}\rightarrow 0.$$
 
 \vni 
Note that 
 $$N_{\ce|S}=\h{O}_S({\ce})_{|{\ce}}=\h{O}_S({\ce})\otimes \h{O}_{\ce}.$$ 
 \vni
 (i) Claim: $h^0({\ce}, N_{{\ce}|S})=31$ and  $h^1({\ce}, N_{{\ce}|S})=0.$ 
 \vni The first part follows from the exact sequence
 $$0\rightarrow\h{O}_S\rightarrow\h{O}_S({\ce})\rightarrow\h{O}_S({\ce})\otimes\h{O}_{\ce}\rightarrow0\,;$$ 
 a routine computation yields $h^0(S, \h{O}_S({\ce}))=32$ and since $S$ is rational, we have
 $h^1(S,\h{O}_S)=0$. 
 Since
 $\deg\h{O}_S({\ce})\otimes\h{O}_{\ce}={\ce}^2=(10;4,3^4)^2=48,$ $\h{O}_S({\ce})\otimes\h{O}_{\ce}$ is non-special and hence $h^1({\ce}, N_{{\ce}|S})=0$.
 
 \vni
 (ii) Claim: $$h^0({\ce}, N_{{S|\PP^4}_{|{\ce}}})=h^0({\ce}, N_{S|\PP^4}\otimes\h{O}_{\ce})=26.$$
 \vni
 Since $S$ is a complete intersection of two quadrics, we have
 $$N_{S|\PP^4}=\h{O}_{S}(2)\oplus\h{O}_{S}(2), $$
 and hence 
 $$N_{S|\PP^4}\otimes\h{O}_{\ce}=(\h{O}_{S}(2)\oplus\h{O}_{S}(2))\otimes\h{O}_{\ce}=\h{O}_{\ce}(2)\oplus\h{O}_{\ce}(2).$$
 By Riemann-Roch,  $$h^0({\ce}, \h{O}(2))=28-18+1+h^0(K_{\ce}\otimes\h{O}_{\ce}(-2))
 .$$ Note that
 \begin{equation}\label{4fold}
 K_S\otimes\h{O}_S({\ce})\otimes\h{O}_S(-2)=|(10;4,3^4)-3(3;1^5)|=|(1;1,0^4)|=|l-e_1|
 \end{equation}
 and by adjunction we have
 $$(K_S\otimes\h{O}_S({\ce})\otimes\h{O}_S(-2))\otimes\h{O}_{\ce}=K_{\ce}\otimes\h{O}_{\ce}(-2).
 $$ The restriction map $$\rho: H^0(S,K_S\otimes\h{O}_S({\ce})\otimes\h{O}_S(-2))\rightarrow H^0({\ce}, K_{\ce}\otimes\h{O}_{\ce}(-2))$$ is surjective since $h^1(S,K_S\otimes\Oo_S(-2))=h^1(S, \Oo(-3(3;1^5)))=0$. By \eqref{4fold}, we see that $|K_{\ce}\otimes\h{O}_{\ce}(-2)|$ is cut out on ${\ce}$ by lines through the ordinary four-fold point on the
 plane model of ${\ce}$, thus we have $$h^0(K_{\ce}\otimes\h{O}_{\ce}(-2))=2.$$
 Hence  $h^0({\ce},\h{O}_{\ce}(2))=13$ and the claim follows. 
 
 \vni
 (iii) By Claims (i) and (ii) we have $$h^0({\ce}, N_{{\ce}|\PP^4})=h^0({\ce}, N_{S|\PP^4}\otimes\h{O}_{\ce})+h^0({\ce}, N_{{\ce}|S})=57,$$ and hence
$$\dim\h{H}\le h^0({\ce}, N_{{\ce}|\PP^4})=57.$$
By duality it follows that 
\begin{align*}
\dim{\w{\h{F}}_2}&=\dim{\w{\h{F}}_2}^\vee= \dim\Hh-\dim\Aut(\PP^4)\\&\le h^0({\ce}, N_{{\ce}|\PP^4})-\dim\Aut(\PP^4)=33=\dim\h{F}_2,
\end{align*}
where $\w{\h{F}}_2\subset\h{G}_{\h{L}}\subset\h{G}^7_{20}$ is the component containing the
 irreducible family  $\h{F}_2$.

\vni
{\bf Conclusion:} ${\w{\Ff}_2}$ is an irreducible component of $\w{\Gg}_\Ll\subset\Gg^7_{20}$ different from the component containing the irreducible family $\Ff_1$, from which the reducibility of $\HL{2r+6,r+11,r}$ for $r=7$ follows. 
\end{proof}
\vni
Leaving the unsettled cases $3\le r\le 6$ behind concerning the irreducibility of $\Hh_{2r+6,r+11,r}$, 
we close this section with a further reducibility result for the case $g=r+12$. Note that for $g=r+12$, $\HL{g+r-5,g,r}\neq\emptyset$ for every $r\ge3$. 
Compared with our previous irreducibility results, the following result is more descriptive and extensive in a certain sense, if not better.
\begin{thm}\label{g=r+12} $\HL{g+r-5,g, r}$, $g=r+12$ is reducible for $9\le r\le 14$. For $r\ge 15$, 
there is a one-to-one correspondence between components of $\HL{2r+7,r+12, r}$ and components of $\h{X}_{3,1}$, the Hurwitz space of triple coverings of elliptic curves. 
\end{thm}
\begin{proof} We present only essential ingredients of proofs, avoiding repetitions of arguments similar to those we used so far.
\vni
For a general $\h{D}\in\h{G}\subset\widetilde{\h{G}}_\h{L}\subset\h{G}^r_{2r+7}$, we assume 
$\h{E}=\h{D}^\vee=g^4_{15}=g^4_e$ is base-point-free.\footnote{Here we do not consider the possibility for $\Ee$ having non-empty base locus, mainly because determining the irreducibility is our main concern.} We distinguish two cases.
\begin{enumerate}
\item[(i)] $\h{E}$ is birationally very ample. Note that if $$\pi_1(e,4)=21\lneq g=r+12\le\pi (e,4)=26 \Longleftrightarrow 10\le r\le 14$$
the image curve $C_\h{E}\subset\PP^4$ induced by
$\h{E}$ lies on a cubic surface $S\subset\PP^4$. Therefore, as we did in the case $r+10\le g\le r+11$, we expect that the the irreducibility (or reducibility in some cases) would 
follow by focusing on (possibly singular) curves lying on a scroll and studying the Severi variety of nodal curves on $S$. 
\item[(ii)] $\h{E}$ is compounded. Since $\h{E}$ have empty base locus, $\Ee$  induces a triple covering $C\stackrel{\phi}{\rightarrow} E\subset\PP^4$ onto a smooth elliptic curve $E$, i,e, $\h{E}=\phi^*(g^4_5)$.   Recall that $\h{E}^\vee$ is very ample if $r\ge 10$ by Lemma \ref{triple0}.
\end{enumerate}
In the range $10\le r\le 14$, we will show that  there exists at least one component for which $\h{E}$ birationally very ample.  On the other hand, on a triple covering $C\stackrel{\phi}{\rightarrow} E$ of an elliptic curve $E$,
the residual series of $\phi^*(g^4_5)$ is very ample. Hence there are at least two candidates of irreducible families which may form components in this range $10\le r\le 14$.

\vni
 For $r\ge 15$, $\Ee$  is compounded by Castelnuovo genus bound and hence $\Ee$ only induces a triple covering of an elliptic curve. Consequently, there corresponds an irreducible family of smooth curves with $(d,g,r)=(2r+7,r+12,r)$ over each irreducible component of the Hurwitz space $\h{X}_{3,1}\subset\h{M}_g$, whose irreducibility is not known\footnote{ to the author}. Therefore if $r\ge 15$, the Hilbert scheme $\HL{2r+7,r+12,,r}$ has the same number of components  as $\h{X}_{3,1}$.

\vni
Suppose $\h{E}$ is birationally very ample.  Assume that $p_a(\ce)=\pi(15,4)=26$. Solving
(\ref{scrolldegree}), (\ref{scrollgenus}) we have
$$\ce\in\Nn_1:=|5H|=|5(e+2f)|=|10f+5e|=|10l-5e|=(10;5).$$ 
\vni
Let ${\w{C}_{\Ee}}$ be the proper transformation of $\ce\subset S$ under the  blowing up $S\cong\PP^2_1$ at 
$$\delta(p_a(\ce)):=p_a(\ce)-g=p_a(\ce)-(r+12)=14-r$$
nodal points. On $\PP^2_{15-r}$ we have 
$${\w{C}_{\Ee}}\in (10,5,2^{14-r})$$ 
and 
\begin{align}\label{edual}
\Ee^\vee&=|K_{\PP^2_{15-r}}+{\w{C}_{\Ee}}-\w{H}|_{|{\w{C}_{\Ee}}}\\&=|(-3;-1^{15-r})+(10;5,2^{14-r})-(2,1, 0^{14-r})|_{|{\w{C}_{\Ee}}}=(5;3,1^{14-r})_{|{\w{C}_{\Ee}}}\nonumber
\end{align}
which is very ample by \cite{sandra}.
We also have
$${\w{C}_{\Ee}}\cdot \Ee^\vee=(5;3,1^{14-r})\cdot (10,5,2^{14-r})=2r+7.$$ 
Therefore we have 
an irreducible family of nodal curves on the Hirzebruch surface $\mathbb{F}_1$ -- parametrized by the irreducible Severi variety $\Sigma_{\h{N}_1,14-r}$ -- whose
residual curve  (i.e. the curve induced by $\h{E}^\vee$) is a smooth, linearly normal curve with $(d,g)=(2r+7, r+12)$ in $\PP^r$.  
Set 
\begin{align*}
&\h{G}_{\Nn_1}:=\\
&\{~\h{E}^\vee~|~ \h{E}=g^4_{15} ~\textrm{is birationally very ample}, \ce\in \Sigma_{\h{N}_1,14-r}, \,p_a(\ce)=\pi(15,4)\}\\
& \hskip 20pt\subset\h{G}\subset\w{\h{G}}_\h{L}\subset \h{G}^r_{2r+7}.
\end{align*}
By our preceding discussion together with a usual dimension count, 
\begin{align}\label{Fdim}
\dim\h{G}_{\Nn_1}&=\dim\Sigma_{\Nn_1,14-r}-\dim\Aut(S)=30+r\nonumber\\&>\lambda(2r+7,r+12,r)=40-r.
\end{align}


\vni  In case $\h{E}$ is compounded, we set
$$\h{G}_{3,1}:=\{ ~\Ee^\vee ~| ~\Ee=\phi^*(g^4_5), ~C\stackrel{\phi}{\rightarrow} E,~ \deg\phi =3, E\mathrm{~is ~elliptic} \}\subset\w{\h{G}}_\h{L}\subset\h{G}^r_{2r+7}.$$
We have\footnote{ The Hurwitz space $\Xx_{n,\gamma}\subset\Mm_g$  which is the locus of  smooth curves of genus $g$ which are degree $n$ ramified coverings of smooth genus $\gamma$ curves is of pure dimension $2g+(2n-3)(1-\gamma)-2$; cf. \cite[Theorem 8.23, p. 828]{ACGH2}.}
$$\dim\h{G}_{3,1}=\dim\h{X}_{3,1}+\dim W^4_5(E)=\dim\h{X}_{3,1}+\dim J(E)=2g-1=2r+23. $$

\vni
Recall that in the range $10\le r\le 14$, we have either 
$$\Dd=\Ee^\vee\in \h{G}_{\Nn_1} \textrm{ or }  \Dd=\Ee^\vee\in \Gg_{3,1}.$$ We also have
$\dim\h{G}_{\Nn_1}<\dim\Gg_{3,1}$. By recalling that a birationally very ample linear series is not a specialization of compounded linear series', $\h{G}_{\Nn_1}$ and each components of $\Gg_{3,1}$ are distinct components of $\w{\h{G}}_\h{L}\subset \h{G}^r_{2r+7}$.
\vni If $r=9$, there is no guarantee that there may exist irreducible families of smooth curves with
$(d,g)=(2r+7,r+12)$ sitting over $\Xx_{3,1}\subset\Mm_g$. Furthermore, $\ce\subset\PP^4$ 
may not be lie only in a surface of minimal degree; $\pi_1(e,4)=21=g=r+12$ if $r=9$. Instead, $\ce$ may lie on a quartic surface unlike the cases $10\le r\le 14$. However we may overcome this 
 situation as follows, which we only give outlines. 
\begin{itemize} \item We assume $\Ee=\Dd^\vee$ is base point free and birationally very ample.
\item If $\ce\subset\PP^4$ lies on a smooth cubic scroll $S$, we retain the notation $\Nn_1:=|5H|$. 
\item In this case we may assume $\ce$ is nodal and is general in
 $\Sigma_{\h{N}_{1,14-r}}$. One easily checks $\Ee^\vee$ is very ample as we did in \eqref{edual}.
Set $$\Ff:=\Sigma_{\h{N}_{1,14-r}}/\Aut(S)\subset\Gg^4_{15}.$$
\item By \eqref{Fdim}, $\lambda(d,g,r)=31<30+r=39=\dim\Ff$, hence the residual series' of the corresponding family $\Ff$
may well contribute to a component of $\HL{2r+7,r+12,r}$.

\item To show that there exists another component different from  $\Ff^\vee$, we 
seek for a suitable irreducible family of curves on  smooth del Pezzo in $\PP^4$. 

\item
 Assume $\ce$ lies on a smooth del Pezzo  surface $S\subset\PP^4$ and set $\ce\in (a;b_1,\cdots,b_5)$. Solving the equations
 $$\deg\ce=3a-\sum^5_{i=1} b_i=15, ~\ce^2=a^2-\sum^5_{i=1} b_i^2=2g-2-K_S\cdot\ce=55$$ we get
 $$(a;b_1,\cdots,b_5)=(10;3^5) \textrm{ or } (11,4^3,3^2), (12; 5,4^4).$$
 Note that all the above three linear systems are very ample. Furthermore,  a general member 
 in one of three is isomorphic to a general member of others via obvious quadratic transformations. 
 
 \item Let $\ce\in(10;3^5)$. There is an irreducible family 
 $\h{F}_1\subset\h{G}^4_{15}$ consisting of hyperplane series $\Ee$ of $\ce\subset\PP^2_5\cong S\subset\PP^4$.
 Since $$|{\ce}+K_S-H|=(10;3^5)-2(3;1^5)=(4;1^5)$$ is very ample, 
 $\h{E}^\vee$ induces an embedding into $\PP^r$ where
 $r=\dim(4;1^5)=9$.
\item
A naive dimension count yields
$$\dim\h{F}_1=\dim(10;3^5)+\dim\mathbb{G}(1,14)-\dim\Aut(\PP^4)=37.$$
 
 \vni
 \item So far we have the following two irreducible families inside $\h{G}^4_{15}$ which have dimension greater than $\lambda(15,21,4)=\lambda(25,21,9)=31$.
 
 \vni
 \begin{enumerate}
 \item[(1)]  $\h{F}$: $\dim\h{F}=39$, $\ce\subset\PP^4$ induced by $\h{E}\in\h{F}$ is singular lying on a rational normal scroll $S$ and  $\ce$ is $5$-gonal. $\h{E}^\vee$ is very ample for a general $\h{E}\in\h{F}$.
 \item[(2)] $\h{F}_1$: $\dim\h{F}_2=37$, $\ce\subset\PP^4$ induced by $\h{E}\in\h{F}_1$ is smooth lying on a smooth del Pezzo  surface and $\ce$ is $7$-gonal.\footnote{ $\ce\in (10;3^5)$ being $7$-gonal is not totally clear. One may consult \cite[Theorem 2, 3]{Sakai}.} $\h{E}^\vee$ is very ample for a general $\h{E}\in\h{F}_1$.
 \end{enumerate}

 
 

 
  \vni 
  \item
If the Hilbert scheme $\HL{2r+7,r+12,r}=\HL{25,21,9}$ is irreducible, then both $\h{F}, \h{F}_1$ would be contained in another irreducible family of dimension strictly bigger than $\dim\h{F}=39$.
 \vni 
\item We may argue that this is not the case by showing that $\h{F}_1$ is 
dense in a  component $\tilde{\h{F}_1}\subset\h{G}^4_{15}$ and hence $\dim\tilde{\h{F}_1}=37$.

\vni
\item
Let $C:=\ce\subset S\subset\PP^4$, where $C\in (10;3^5)$ and $S$ is a smooth del Pezzo  surface in $\PP^4$. Let $c\in\HL{15,21,4}$ be a point corresponding to $C$ and let $\h{H}$ be a component containing $c$.
 We compute 
 $$h^0(C,N_{C|\PP^4})=\dim T_c{\h{H}}.$$
 
 
 
 \vni
\item From the the standard exact sequence
 $$0\rightarrow N_{C|S}\rightarrow N_{C|\PP^4}\rightarrow N_{{S|\PP^4}_{|C}}\rightarrow 0, $$
 it is easy to check
\item (i) Claim: $h^0(C, N_{C|S})=35$ and  $h^1(C, N_{C|S})=0.$ 
 \vni
 (ii) Claim: 
 \begin{align*}
 h^0(C, N_{{S|\PP^4}_{|C}})&=h^0(C, N_{S|\PP^4}\otimes\h{O}_C)\\
 &=h^0(C,(\h{O}_{S}(2)\oplus\h{O}_{S}(2))\otimes\h{O}_C)\\&=2h^0(C,\h{O}_C(2))=26;
 \end{align*}
 By Riemann-Roch,  we have $$h^0(C, \h{O}(2))=30-21+1+h^0(K_C\otimes\h{O}_C(-2))
 .$$ Note that
 $$K_S\otimes\h{O}_S(C)\otimes\h{O}_S(-2)=|(10;3^5)-3(3;1^5)|=|(1;0^5)|=|l|$$
 and by adjunction
 $$(K_S\otimes\h{O}_S(C)\otimes\h{O}_S(-2))\otimes\h{O}_C=K_C\otimes\h{O}_C(-2),
 $$ we see that $|K_C\otimes\h{O}_C(-2)|$ is cut out on $C$ by lines
 in the plane model of $C$ and therefore $$h^0(K_C\otimes\h{O}_C(-2))=3.$$
 Hence  $h^0(C,\h{O}_C(2))=13$ and Claim (ii) follows. 
 
 \vni
 (iii) By Claims (i) and (ii) we have $$h^0(C, N_{C|\PP^4})=h^0(C, N_{S|\PP^4}\otimes\h{O}_C)+h^0(C, N_{C|S})=61,$$ and therefore
$$\dim\h{H}\le h^0(C, N_{C|\PP^4})=61.$$
It finally follows that 
$$\dim\tilde{\h{F}}_1\le h^0(C, N_{C|\PP^4})-\dim\Aut(\PP^4)=37,$$
thus $\h{F}_1$ is dense in $\tilde{\h{F}}_1$. Passing to the residual family $\Ff_1^\vee\subset\Gg^9_{25}$ we see that $\HL{25,21,9}$ has a component other than the one arising from the family $\Ff^\vee$. 
 
 

\end{itemize}



\end{proof}
\begin{rmk}
The family of {\bf simple} triple coverings of a fixed elliptic curve is irreducible by a work of Kanev; cf. \cite{Kanev}.
Therefore it is likely that if $r\ge 15$, $\HL{2r+7,r+12,r}$ is irreducible, for which the author does not have any proper justification at this moment.
\end{rmk}


The following figure is a summary of those irreducibility results of $\HL{d,g,r}$ we discussed in \S4,\S5 and \S6.
\newpage
\begin{figure}
\pgfplotsset{
every axis/.append style={
before end axis/.code={
\fill [green] (5,280) circle(5pt);
\fill [green] (120,280) circle(5pt);
\node at (60,280)
{\footnotesize $\HL{d,g,r}\neq\emptyset$;
\texttt{irreducibility is not known in general}.};
},
},
}
\begin{tikzpicture}
\begin{axis}[grid=both, transpose legend,
legend columns=9,
legend style={at={(0.5,-0.2)},anchor=north, font=\footnotesize},
xlabel= {Dimension of the projective space $r$},ylabel=Genus $g$,
ymin=3,ymax=35,enlargelimits=false,
scatter/classes={
a={mark=square*,blue},
b={mark=triangle*,red},
c={mark=o,draw=black},
d={mark=x},
e={mark=triangle*, blue},
f={mark=triangle*,green}
},
]

\addplot [
scatter,only marks,
scatter src=explicit symbolic,
] coordinates {
(4,12) [a]
(4,13)[b]
(7,16) [a]
(6,15)[b]
(8,17) [a]
(9,18) [a]
(10,19) [a]
(7,16) [a]
(5,14) [a]
(3,12) [b]
(3,14) [c]
(4,15) [c]
(5,16) [c]
(6,17) [c]
(3,14) [c]
(8,17) [c]
(3,13) [a]
(4,14) [f]
(5,15) [a]
(6,16) [a]
(7,17) [a]
(8,18) [b]
(7,18)[b]
(8,19)[a]
(9,20)[a]
(10,21)[b]
(11,22)[a]
(3,15)[c]
(4,16)[c]
(5,17)[c]
(6,18)[c]
(7,19)[c]
(8,20)[c]
(9,21)[b]
(10,22)[b]
(11,23)[b]
(12,24)[b]
(13,25)[b]
(14,26)[b]
(11,20)[a]
(12,21)[a]
(13,22)[a]
(14,23)[a]
(3,11)[d]
(5,13)[d]
(6,14)[d]
(7,15)[d]
(8,16)[d]
(9,17)[d]
(10,18)[d]
(11,19)[d]
(12,20)[d]
(13,21)[d]
(14,22)[d]
(15,24)[a]
(15,25)[d]
(15,23)[d]
(15,26)[d]
(16,24)[d]
(16,25)[a]
(16,26)[d]
(16,27)[d]
(16,28)[e]
(12,23)[d]
(13,24)[d]
(14,25)[d]
(9,19)[d]
(10,20)[d]
(11,21)[d]
(12,22)[d]
(13,23)[d]
(14,24)[d]
(15,27)[e]
};
\addlegendentry{Irreducible}
\addlegendentry{Reducible}
\addlegendentry{ ~$\mathcal{H}^\mathcal{L}_{d,g,r}\neq\emptyset$, but irreducibility is not known yet}
\addlegendentry{ ~ $\mathcal{H}^\mathcal{L}_{d,g,r}=\emptyset$}
\addlegendentry{ ~ $\{(r,g) | g=r+12, r\ge15\}$,
$\#$ of components of $\mathcal{H}^\mathcal{L}_{d,g,r}=$$ ~\#$ of components of $\mathcal{X}_{3,1}$}
\addlegendentry{ ~ Only knows $\exists$  a component of the expected dimension; Proposition \ref{d13r4}.}
\addplot [red!80!black,fill=red,fill
opacity=0.5,
] coordinates {
(3,11) (4,12) (5,13) (6,14)
(7,15) (8,16) (9,17) (10,18)
(11,19) (12,20) (13,21) (14,22) (15,23) (16,24)
}
|- (0,0) -- cycle; \addlegendentry{Pink Zone, $\mathcal{H}^\mathcal{L}_{d,g,r}=\emptyset$}

\addplot [green!20!black] coordinates {
(3,20) (5,30) (13,70)
};
\addlegendentry{ Brill-Noether line $g=5(r+1)$ for $g-d+r=5$}

\addplot [blue!20!black] coordinates {
(3,16) (16,29)
};
\addlegendentry{ Above  line $g=r+13$,  $\mathcal{H}^\mathcal{L}_{d,g,r}\neq\emptyset$;  irreducibility is not known in general.}

\end{axis}
\end{tikzpicture}
\caption{Irreducibility map}
\label{fig:irreducibility}
\end{figure}

\section{An epilogue, beyond index of speciality five }

In this final section, we would like to discuss and make a couple of statements on families of projective curves with arbitrarily given index of speciality $\al$, as an attempt to seek for a reasonable generalization of the results we obtained for $\al\le 5$. 
Even though we need to impose certain restrictions on the range of the triples $(d,g,r)$ in which the statements work, it may be worthy of coming up with some generalizations  for further study along the line of ideas we followed in this paper. For almost all the statements which will be made in this section, the author is reluctant to provide proofs, mainly because this article would get too lengthy otherwise.  Verifications for several statements in this section and more extensive relevant  results can be found in our forthcoming paper \cite{a6}.

\vni
We are interested in studying Hilbert schemes of smooth curves of degree $d$ and genus $g$ in $\PP^r$ with a given fixed index of speciality. There is a sharp upper bound of the index of speciality of the hyperplane series of a linearly normal projective curve  in terms of  $(d,g)$. To be precise, we recall the following. 

\begin{rmk}\label{rbound}
\begin{itemize}
\item[(i)]
By Castelnuovo theory, it is known that the dimension $r$ of a birationally very ample 
linear series $g^r_d$  on a curve of genus $g$ is bounded by (cf. \cite[Lemma1.5]{KKL})
\begin{equation}\label{rdim}
r\le
\begin{cases}
\frac{d+1}{3} &~\mathrm{~if~} d\le g\\
\frac{1}{3}(2d-g+1) &~\mathrm{~if~} d\ge g.
\end{cases}
\end{equation}
\item[(ii)] Setting $g-d+r=\al$, the  bound \eqref{rdim} can be read as
\begin{equation}\label{dgg}
\al\le
\begin{cases}
\frac{3g-2d+1}{3} &~\mathrm{~if~} d\le g\\
\frac{1}{3}(2g-d+1) &~\mathrm{~if~} d\ge g.
\end{cases}
\end{equation}
\end{itemize}
\end{rmk}

\vni
We recall  Claim (3.2.A) in the proof of Theorem \ref{sub6x}, which we quote as the following proposition. 
\begin{prop}\label{nonempty}
Given  $\alpha \ge 3$, we have
$$\HL{g+r-\al,g,r}\neq\emptyset$$ 
if  
\begin{equation*}
r\ge\al +1 \mathrm{~~ and ~~} g\ge r+3\al-2.
\end{equation*}
\end{prop}

\medskip
\vni
\begin{rmk}
\begin{itemize}
\item[(i)]
One of numerical condition $r\ge\al+1$ in Proposition \ref{nonempty} is equivalent to $d\ge g+1$. One may think that imposing the restriction $r\ge\al+1$ is a too much sacrifice in view of the fact that  many hard but interesting cases occur when $d\le g$.
On the contrary, our knowledge on 
the Hilbert scheme $\HL{d,g,r}$ is still on the poor side even  when $r\ge \al+1$ as far as the author perceives.  Therefore it is worthwhile to settle down or at least make a comprehensive treatment under the assumption $r\ge \al+1$.  Accordingly, for the rest of this section, {\bf we almost always assume $r\ge \al+1$} unless otherwise specified. We also assume $\al\ge 6$ simply because we treated $\al\le5$ thoroughly in preceding sections and in earlier papers \cite
{JPAA, speciality, lengthy}. 
\item[(ii)]
We note that the bound \eqref{dgg} is sharp. Indeed, when the  index of speciality attains its maximal possible value  $\al= \frac{1}{3}(2g-d+1)$ -- the condition equivalent to the numerical condition $r=\frac{1}{3}(2d-g+1)$ in \eqref{rdim} -- a curve $C\subset\PP^r$ with maximal index of speciality $\al$ is an extremal curve. Indeed  by a straightforward computation, we have in this case $$g=r+2\al-1=\pi(d,r).$$
\item[(iii)]
If  $g\le r+2\al-2$ -- which is equivalent to the condition $r\gneq\frac{1}{3}(2d-g+1)$ -- we get a contradiction by \eqref{rdim}. Therefore we arrive at the following statement
 which can be viewed as a generalization of Theorem \ref{sub6x} (a), (b) and Proposition \ref{g=r+9} (i), (ii).
 \end{itemize}
 \end{rmk}
 \begin{prop}\label{extreme}Suppose $r\ge \al+1$ and $\al\ge6$. We have
 
 \begin{itemize}
  \item[(i)]
  $\Hh_{d,g,r}=\HL{d,g,r}=\emptyset \mathrm{ ~ for ~}  g\le r+2\al-2$
  \item[(ii)]
 $\Hh_{d,g,r}=\HL{d,g,r}\neq\emptyset \mathrm{ ~ for ~} g=r+2\al-1$ and  is reducible only if $r=\al+1$.
 \end{itemize}
 \end{prop}

\vni
In view of Proposition \ref{extreme} and Proposition \ref{nonempty}, it is natural to ask whether $$\HL{d,g,r}\neq\emptyset\mathrm{ ~ ~ \hskip 12pt for ~~ \hskip 12pt }
r+2\al\le g\le r+3\al -3.$$
When $g=r+2\al$ or $g=r+2\al+1$,  one can prove the following statement which is a partial generalization of 
Theorem \ref{sub6x} (c), (d).

\begin{prop}\label{r+2a+1} Fix $\al \ge 5$:
\begin{itemize}
\item[(a)] $\HL{d,g,r}=\emptyset$ for $g=r+2\al$ if $r\ge \al+4$.
\item[(b)] $\HL{d,g,r}=\emptyset$ for $g=r+2\al+1$ if $r\ge\max\{\al+6, 12\}$.
\end{itemize}
\end{prop}

\vni
The converse of Proposition \ref{r+2a+1}, namely

(a$\dagger$) $\HL{d,g,r}\neq\emptyset$ for $g=r+2\al$ and for every 
$3\le r\le \al+3$

(b$\dagger$) $\HL{d,g,r}\neq\emptyset$ for $g=r+2\al+1$ and for every 
$3\le r\le  \al+5$

\vni
 has not been completely settled yet for $\al\ge 6$. Some partial results have been obtained 
 which does not seem to be adequate to be stated precisely here. 
 
 \vni
When the genus $g$ is still small but not too small with respect to $\al$ and $r$, say when $g=r+2\al+2$, one may prove the following  result which  partially recovers Theorem \ref{sub6x}(e) as well as Theorem \ref{g=r+12}.
\vni

\begin{prop}\label{r+2a+2}
Fix $\al\ge 6$. We have  $$\HL{d,g,r}=\HL{g+r-\al,g,r}\neq\emptyset$$ for $g=r+2\al+2$ if $r\ge \al+5$. Furthermore, in case $g=r+2\al+2\ge \max\{3\al+10, 2\al+17\}$, a general element of any component of $$\HL{d,g,r}=\HL{g+r-\al,g,r}=\HL{2r+\al+2,r+2\al+2,r}$$ is a triple cover of an elliptic curve.
\end{prop}

\begin{rmk}
\begin{itemize}
\item[(i)]
 If $g\ge r+2\al+3$, the existence of linearly normal smooth curves in $\PP^r$ with index of speciality $\al$ is only assured for $\al=5$ as we have seen in Theorem \ref{sub6x} (e). Note  the condition $g\ge r+3\al-2$ we have in Proposition \ref{nonempty} is same as $g\ge r+2\al+3$ if $\al=5$. However, for $\al\ge 6$ there is a non-empty range $r+2\al +3\le g\le r+3\al -3$ in which the existence is yet to be determined by some other (hopefully more systematic) treatment.
 \item[(ii)]
In Proposition \ref{nonempty} (i.e. in the proof of Theorem \ref{sub6x} for the Claim (3.2.A)), we impose numerical restrictions such 
as $r\ge \al+1$ and $g\ge r+3\al-2$ so that Lemma \ref{kveryample} works, where
 we used general $k$-gonal curves with $k\ge 4$. However, if we consider the family of trigonal curves with 
Maroni invariant $m$, it is possible to eliminate the numerical restriction $r\ge \al+1$ and the relax the 
genus assumption $g\ge r+3\al-2$ somehow. However, if we do this, we get very ample line bundle
only with particular degrees, say $d\equiv 2g-2 (\mathrm{mod} ~3)$. To be precise, for $\al\ge 3$
$$\HL{2g-2-3(\al-1),g,r}\neq\emptyset, $$ which however can be regarded a consequence of a classically known results on trigonal curves; cf. \cite[Lemma 1]{ms}.
\end{itemize}
\end{rmk}

\medskip
\vni
The existence and the irreducibility of $\HL{g+r-\al,g,r}$ with $r\ge \al+1$ for certain value of the genus $g$ which is 
small with respect to $\al$ and $r$ -- say $r+2\al -1 \le g\le r+3\al+1$ -- can also be carried out in a similar manner as were done in Proposition \ref{g=r+9}, Theorem \ref{r+10}, Theorem \ref{2r+6} and Propositon \ref{r+117}. By analyzing extremal curves, nearly extremal curves or curves lying on a surface of degree $r$ in $\PP^r$ it is possible to determine the irreducibility of $\HL{g+r-\al,g,r}$ for some triples $(d,g,r)$. The following example -- which is partially complementary to the statement (a$\dagger$) after Proposition \ref{r+2a+1} -- is among one of those. 

\begin{exam} 
$\HL{d,g,r}=\HL{3\al+6,3\al+3,\al+3}\neq\emptyset$ and  is irreducible if $\al\ge 7$ and is reducible if $\al=6$.

\end{exam}

\appendix
\section{Specialization} We provide a proof of Proposition \ref{specialization} for which the author could not find an explicit source of a proof in the literature.

\begin{prop}\rm{(Propostion \ref{specialization})} Smooth curves in $\PP^{n+1}$, $n\ge 2$  lying on a cone over a rational normal curve in  a hyperplane $H\cong\PP^{n}$ is  a specialization of curves lying on a rational normal surface scroll.
\end{prop}

\begin{proof} The proof is based on the proofs of similar statements in \cite{Brevik} and \cite{Nasu} for the case $n=2$ and curves on cubic surfaces in $\PP^3$. 
One first shows that there is a flat family whose special fibre is a cone with a general fibre a rational normal surface scroll as follows. 

\vni
Let $T$ be the affine line $\mathbb{A}^1=\mathrm{Spec~ } \CC[t]$ and let $x, y$ be the homogeneous coordinates of $\PP^1$.
Consider the following non-splitting exact sequence, 
\begin{equation}\label{nsplit}
0\rightarrow\h{O}_{\PP^1}\stackrel{[x,y^{n-1}]}{\longrightarrow}\h{O}_{\PP^1}(1)\oplus\h{O}_{\PP^1}(n-1)\stackrel{[y^{n-1},-x]^t}{\longrightarrow}\h{O}_{\PP^1}(n)\longrightarrow 0.
\end{equation}
\vni
Since the exact sequence \eqref{nsplit} does not split, we let $0\neq\mathbf{e}\in\textrm{Ext}^1(\h{O}_{\PP^1}(n),\h{O}_{\PP^1})$
be the corresponding extension class of the exact sequence. Let $\PP^1_T$ be the projective line over $T$. The class $t\,\mathbf{e}\in\textrm{Ext}^1(\h{O}_{\PP^1_T}(n),\h{O}_{\PP^1_T})$ induces an extension $\h{E}$ of $\h{O}_{\PP^1_T}(n)$ by $\h{O}_{\PP^1_T}$, corresponding to the exact sequence
$$0\longrightarrow \h{O}_{\PP^1_T}\longrightarrow\h{E}\longrightarrow \h{O}_{\PP^1_T}(n)\longrightarrow 0.$$

\vni
The fiber of the vector bundle $\h{E}$ over $t\in T$ is given by  $$\h{E}_t\cong\h{O}_{\PP^1}(1)\oplus\h{O}_{\PP^1}(n-1) \textrm{ ~for ~} t\in T-\{0\}$$ and
$$\h{E}_0\cong\h{O}_{\PP^1}\oplus\h{O}_{\PP^1}(n)$$ by definition; since $t=0, t\,\mathbf{e}=\mathbf{0}\in \textrm{Ext}^1(\h{O}_{\PP^1_T}(n),\h{O}_{\PP^1_T})$ 
hence the corresponding exact sequence splits. Set 
$$Y=\PP(\h{E})\stackrel{\pi}{\rightarrow}\PP^1_T,$$ a $\PP^1$
bundle
over $\PP^1_T$.  Note that 
$$Y_t\cong\PP(\h{O}_{\PP^1}(1)\oplus\h{O}_{\PP^1}(n-1))\cong\mathbb{F}_{n-2} \textrm{ for } t\in T-\{0\}$$ and
$$Y_0\cong\PP(\h{O}_{\PP^1}\oplus\h{O}_{\PP^1}(n))\cong\mathbb{F}_{n}.$$ 
\vni
We have the following diagram 
\vni

\begin{equation*}
  \xymatrix@R+2em@C+2em{
  Y \ar[r]^-{\phi} \ar[d]_{\pi} & ~~~\PP^{n+1}_T \ar[d]^{p} \\
  \PP^1_T \ar[r]_-{p}& T
  }
 \end{equation*}
\vni
where $\phi$ is the morphism induced by the  tautological bundle $\h{L}$ on $\PP(\h{E})=Y$;
the image $S$ under $\phi$ is a family of surfaces in $\PP^{n+1}$ of degree $n$ over the affine line $T$  whose general fiber is a rational normal scroll in $\PP^{n+1}$ and the special fiber is a cone over a rational normal curve in $\PP^{n}$. This family is a flat family  because the Hilbert polynomial of the image of special and general fibers are the same.

\vni Let $X_0$ be a smooth connected curve of degree $d$ and genus $g$ on the cone $S_0\subset\PP^{n+1}$ over a rational normal curve in $\PP^{n}$. Let $\st{X_0}\subset Y_0$ denote the proper transformation of $X_0$ under the desingularization  $Y_0\rightarrow S_0$. Denoting $h$ and $f$ by the class of the pull back of $\h{O}_{S_0}(1)$ via $Y_0\rightarrow S_0$ and the class of a fiber 
$Y_0\rightarrow\PP^1$, we have 
$$\st{X_0}\equiv kh+(d-nk)f\in\mathrm{Pic}{ ~Y_0}\cong\mathbb{Z}^{\oplus 2},$$ where $k:=(\st{X_0}\cdot f)$.

\vni
Set 
$$\h{M}=\h{L}^{\otimes k}\otimes\pi^*\h{O}_{\PP^1_T}(d-nk).$$

\vni Then for any $t\in T$, the fiber $\h{M}_t$ over $t$ is an inversitble sheaf on $Y_t$, so that $$\h{M}_t\equiv kh+(d-nk)f\in\mathrm{Pic } ~Y_t$$ by definition of $\h{M}$; by abusing notation we use the same notaton  $f$ and $h$ on $\mathrm{Pic } ~Y_t$ denoting the fiber and the pull back of $\h{O}_{S_t}(1)$ via $Y_t\stackrel{\cong}{\rightarrow} S_t$ which is an isomorphism if $t\neq 0$. By \eqref{conevertex}, we have  $d=nk$ or $d=nk+1$.  We claim that for every $t\neq 0$ there exists smooth connected curve $X_t\subset S_t\subset\PP^{n+1}$ such that $\st{X_t}\in|\h{M}_t|$; here we use the notation $\st{X_t}$ as a curve on the surface $\mathbb{F}_{n-2}\cong Y_t$. To see this, note that the corresponding line bundle on $\mathbb{F}_{n-2}$ is very ample; if $d=nk$, then $\h{M}_t\equiv kh$ and since $h=C_0+(n-1)f$ is very ample, $\h{M}_t$ is very ample. If $d=nk+1$, then 
$$\h{M}_t\equiv kh+f=k(C_0+(n-1)f))+f=kC_0+((n-1)k+1)f$$ is still very ample by the well-known criteria for the very ampleness of line bundles on $\mathbb{F}_{n-2}\cong Y_t\cong S_t$ if $t\neq 0$; cf. \cite[V. Corollary 2.18]{Hartshorne}. Note that 
 $$h^0(Y_t,\h{M}_t)=g+2d-(n-2)k$$ for $t\neq 0$ by a  formula corresponding to \eqref{lsdimension} on a smooth scroll in $\PP^{n+1}$. 
For $t=0$, we also have $$h^0(Y_0,\h{M}_0)=g+2d-(n-2)k,$$
coming from the vanishing of the cohomology on $\FF_e$, $e\ge 0$; 
\begin{equation}\label{vc} 
h^i(\mathbb{F}_e, mh)=h^i(\mathbb{F}_e, mh+f)=0 \mathrm{~for~ } i=1,2 \mathrm{ ~and~ } m\ge -1.
\end{equation}
Recalling that $Y_0=\mathbb{F}_{n}$ and $\h{M}_0=kh+f$ or $kh$, we have
$$h^i(Y_0,\h{M}_0)=h^i(\mathbb{F}_{n},\st{X_0})=0$$ for $i=1,2$
by \eqref{vc}. By Riemann-Roch,
\begin{eqnarray*}h^0(Y_0,\h{M}_0)&=&h^0(Y_0,\st{X}_0)=\frac{1}{2}(\st{X}_0+K_{Y_0})\cdot\st{X}_0-\st{X}_0\cdot K_{Y_0}+\chi(\h{O}_{Y_0})\\
&=&\frac{1}{2}\deg K_{\st{X}_0}-\st{X}_0\cdot(-2h+(n-2)f)+1
\\&=&g+2d-(n-2)k.
\end{eqnarray*}


\vni
We may regard $Y$ as a $T$-schme via $$g:=\phi\circ p=p\circ\pi; Y\rightarrow T.$$ Note that by 
Grauert theorem $g_*\h{M}$ is locally free of rank $r=g+2d-(n-2)k$ on $T$ and hence $\PP(g_*\h{M})$ is an irreducible $\PP^{r-1}$ - bundle over $T$. \vni
We further note that a general elment of $\PP(g_*\h{M})$ corresponds to a smooth connected curve  which is conained in a general member of the family $S\subset\PP^{n+1}_T$. Therefore it finally follows that there is an open set $U\subset \PP(g_*\h{M})$ and a flat family of curves over $U$ whose general element is a curve on the surface $S_t\subset\PP^{n+1}$, $t\in T-\{0\}$, and whose special member is in $X_0\subset S_0\subset\PP^{n+1}$ where $S_0$ is a cone; take an open set $0\in V\subset T$. Then take $\st{U}=g^{-1}(V)\subset Y$. And then consider the restriction map $\st{g}=g_{|\st{U}}: \st{U}\rightarrow V\subset T$. Then take $U=\PP(\st{g}_*(\h{M}_{|\st{U}})).
$
\end{proof}

\section{A criteria for linear normality} 
\vni
The criteria Remark \ref{linearlynormalcriteria} (c)(iii) for the linear normality of curves lying on  certain rational surfaces is a consequence  of the following lemma due to Dolcetti and Pareschi \cite[Lemma 1.3]{DP}.
\begin{lem} \label{dpc}Let $C\subset\PP^n$ be a smooth, connected curve and $H$ be a hyperplane intersecting $C$ transversally. Let $Z\subset H$ be a curve. Assume that 

\textrm{(i)} $h^j(H,\mathcal{I}_{Z,H}(1))=0, j=0,1$

\textrm{(ii)} $C\cap Z=C\cap H$ as schemes.

\vni Then $X=C\cup Z$ satisfies $h^j(\PP^n,\h{I}_X(1))=0, j=0, 1$.
\end{lem}

\vni
\begin{cor}\label{appendixb}
Let $S\subset\PP^n$ be a smooth rational surface $\PP^2_s$ embedded by a very ample linear system $\h{Z}=(a; b_1, \cdots , b_s)$, $b_1\ge b_2\ge \cdots \ge b_s\ge 1$ with $\sum b_i<3a$. 
Let $\h{L}=(\st{a}; \st{b}_1, \cdots , \st{b}_s)$ be  a very ample linear system  such that 
$\st{a}\ge a, \st{b}_1\ge b_1, \cdots , \st{b}_s\ge b_s$. Assume that the linear system 
$\h{L}\otimes\h{Z}^{-1}$ contains a smooth connected curve $C\subset S$.
Then a general $\st{C}\in\h{L}$ is smooth, irreducible and linearly normal.
\end{cor}
\begin{proof} Let $Z=S\cap H\in\h{Z}$ be a general hyperplane section of $S$ which is  smooth,  irreducible and non-degenerate in $H$.  Note that
\begin{itemize}
\item[(i)]
$\deg Z=\deg S=Z^2=a^2-\sum b_i^2$
\item[(ii)]
$p_a(Z)=\frac{(a-1)(a-2)}{2}-\sum\frac{b_i(b_i-1)}{2}$ by adjunction
\item[(iii)]
$h^0(H,\h{I}_{Z,H}(1))=0$ since $Z$ is non-degenerate in $H$
\item[(iv)]
$h^1(H,\h{I}_{Z,H}(1))=0$: By $\sum b_i<3a$, $2p_a(Z)-2<\deg Z$ and we have 
$$n=h^0(H,\h{O}_H(1))=h^0(Z, \h{O}_Z(1))=\deg Z-p_a(Z)+1.$$ Thus $h^1(H,\h{I}_{Z,H}(1))=0$

\end{itemize}
The effective divisor $X=C+Z\in\h{L}$ satisfiles  $h^j(\PP^n,\h{I}_X(1))=0, j=0, 1$ by Lemma \ref{dpc}. Hence by semi-continuity, a general element of the very ample $\h{L}$ is a linearly normal, smooth and irreducible curve. 
\end{proof}

\begin{rmk} \vni
(i) For $a=3,  1\le s\le 8, b_1=\cdots =b_s=1$, a smooth curve on a del Pezzo  surface of degree $9-s$ in $\PP^{9-s}$ belonging to a very ample linear system $({a}';{b_1}',\cdots {b_s}')$ is linearly normal if $b'_s\ge 1$. For $s=6$, this is a special case of a stronger result by Kleppe \cite[Proposition 3.1.3]{kleppe}

\vni
(ii) If $a=4, s=10, b_i=1, 1\le i\le 10$, a smooth curve $C\in (a'; b'_1, \cdots , b'_s), b_1'\ge b'_2\ge \cdots \ge b'_s$ on a Bordiga surface $S\stackrel{(4;1^s)}{\hookrightarrow}\PP^4$  is linearly normal if $b'_s\ge 1$.
\end{rmk}
\section{Very ample linear series on a triple covering of elliptic curves}
\begin{lem}\label{triple}
 Let $C\stackrel{\eta}{\rightarrow} E$ be a triple covering of an elliptic curve $E$ of genus $g\ge 3\al+7$, $\al\ge3$. 
 For $g^{\al -1}_\al\in W^{\al -1}_\al(E)$, 
 we put $\h{E}:=|\eta^*(g^{\al -1}_\al )|$.
 Then 
 $\dim\h{E}=\al -1$ and $\h{E}^\vee =g^{g-2\al -2}_{2g-2-3\al }$ is very ample.

 \end{lem}
 \begin{proof} $\h{E}=|\eta^*(g^{\al -1}_\al )|$ is base-point-free since $g^{\al -1}_\al$ on $E$ is base-point-free. Set $\h{E}_0=\eta^*(g^{\al-1}_\al)\subset|\eta^*(g^{\al -1}_\al )|$.
 Suppose $\dim\h{E}=\beta\ge \al$.   Note that  the morphism induced by $\h{E}_0=g^{\al-1}_{3\al}$ --
 which is same as the triple cover $\eta$ -- factors through the morphism $C\stackrel{\phi}{\rightarrow}\PP^\beta$ induced by $\h{E}$ and $\deg\phi|\deg\eta=3$. If $\deg\phi =3$, then $\al\le\beta\le\deg\phi(C)=\frac{3\al}{\deg\phi}=\al$ and hence $\phi(C)$ is a rational curve, which is impossible by Castelnuovo-Severi inequality or by the fact that there is no morphism from a rational curve $\phi(C)$ onto an elliptic curve $E$. 
Hence the morphism $\phi$ is birationally very ample. However $g\le\pi(3\al,\al)=3\al+3$ which is not compatible with the genus assumption $g\ge3\al+7$. Thus we have $\dim\h{E}=\al -1$.
   
 Suppose $\h{E}^\vee=\h{D}$ is not very ample. Choose $\Delta\in C_2$ such that $\h{E}'=|\h{E}+\Delta|=g^s_{3\al +2}$, $\al\le s\le\al +1$. Since 
$|\h{E}'-\Delta|=\h{E}$, the triple covering $C\stackrel{\eta}{\rightarrow} E\subset\PP^{\al -1}$ onto $E$ factors through the morphism $C\stackrel{\zeta}{\rightarrow} F\subset\PP^s$ induced by the moving part $\w{\h{E}}$ of $\h{E}'$;
\[
\xymatrix{
C 
\ar[r]^{{\zeta}} \ar[dr]_{\eta} &{F}\ar[d]^{\nu} \\
  &E
}
\]

\vni
Note that $$\deg{\zeta}|\deg\eta =3,\,
\deg F\ge s\ge \al, \,3\al +1\le\deg\w{\h{E}}=\deg\zeta\cdot\deg F\le 3\al+2.$$ Hence $\deg{\zeta}=1$ and $\zeta$ is a birational morphism onto its image 
 $F\subset\PP^s$. By the Castelnuovo genus bound, 
$$g\le p_a(F)\le \pi (t, s) ~~~\textrm{with}~~~  3\al +1\le t\le 3\al+2, \al\le s\le\al +1$$ and we have the following cases;
\[  g\le p_a(F)\le\pi (t,s)=\begin{cases}
3\al& \textrm{if }~ (t,s)=(3\al+1,\al+1) \\
3\al+3& \textrm{if }~ (t,s)=(3\al+2,\al+1) \\
3\al +6& \textrm{if }~ (t,s)=(3\al+1,\al), \al\ge 4\\
3\al +7& \textrm{if }~ (t,s)=(3\al+1,\al), \al=3\\
3\al +9& \textrm{if } ~(t,s)=(3\al+2,\al),\al\ge 5\\
3\al +10& \textrm{if } ~(t,s)=(3\al+2,\al),\al=4\\
3\al +11& \textrm{if } ~(t,s)=(3\al+2,\al),\al=3

\end{cases}
\]
\vni
Three cases from the top can be eliminated by the assumption $g\ge 3\al +7$. For the fourth case $\al =3$, $(t,s)=(10,3)$, we have $g=3\al+7=16=\pi(10,3)$. Hence $F\subset\PP^3$ is an extremal curve 
and is $5$-gonal. By the Casetelnuovo-Severi inequality, a curve of genus $g\ge 12$ cannot be both triple covering of an elliptic curve and $5$-gonal.
For the remaining three cases, while we have $\pi (t,s)\gneq 3\al +7$, we also have $$\pi_1(3\al +2, \al)=3\al+6<g$$ hence $F\subset\PP^{\al}$ lies on a surface $S\subset\PP^\al$ of minimal degree $\al-1$. We assume $S$ is 
a smooth rational normal scroll or a Veronese if $\al =5$. If $\al =5$, $\deg F=t=3\al +2$ is odd and hence $S$ is cannot be a Veronese surface. Using the same notation as in Remark \ref{minimal},  we let $F\in|aH+bL|$. We find common integer solutions $(a,b)$ for the equations (\ref{scrolldegree}) 
and (\ref{scrollgenus}) in Remark \ref{minimal} by substituting $n=\al-1$, $d=t=3\al+2$,  $p_a(F)=3\al+7, 3\al+8,3\al+9$, $3\al+10$ (only when $\al =3,4)$ and  $3\al+11$ (only when $\al=3$). After elementary but tedious numerical computation, we arrive at
\vni
\[  (a,b)=\begin{cases}
(6,-1) & \textrm{if }~ p_a(F)=3\al+11;\textrm{ only ~when } \al=3\\
(5,-1) & \textrm{if }~ p_a(F)=3\al+10;\textrm{ only ~when } \al=4\\
(4,6-\al) & \textrm{if }~ p_a(F)=3\al+9 \textrm{ in general\quad  or } \\
(5,-3) & \textrm{if }~ p_a(F)=3\al+9  \textrm{ and } \al=5\\

(7,-3) & \textrm{if }~ p_a(F)=3\al+9  \textrm{ and } \al=3\\

(5,-5)& \textrm{if  }~ p_a(F)=3\al+8; \textrm{ only ~when } \al=6\\
(6,-4)& \textrm{if  }~ p_a(F)=3\al+8; \textrm{ only ~when } \al=4\\
(5,-7)& \textrm{if } ~p_a(F)=3\al+7; \textrm{ only ~when } \al=7.\\

\end{cases}
\]
\vni Note that  the ruling $|L|$ cut out a base-point-free $g^1_a$ on $F\in|aH+bL|$. Hence  by the Castelnuovo-Severi inequality, we have 
$$g\le (3-1)(a-1)+3\cdot 1=2a+1\le 15\le 3\al +6$$if $a\le 7$ and $\al\ge 3$,  contrary to the genus assumption  $g\ge 3\al +7$.

\vni
We assume that  $S\subset\PP^\al$ is a cone over a rational normal curve in $\PP^{\al-1}$. In this case we solve the equation (\ref{cone}) in Remark \ref{minimal} to find integer values $k$ by substituting
$n=\al-1$, $d=3\al+2$.  $p_a(F)=3\al+7, 3\al+8,3\al+9$, $3\al+10$ (only if $\al =3,4)$ and  $3\al+11$ (only if $\al=3$); 
\[ k=\begin{cases}
5,6 & \textrm{if }~ p_a(F)=3\al+11;\textrm{ only ~when } \al=3\\
5 & \textrm{if }~ p_a(F)=3\al+10;\textrm{ only ~when } \al=4\\
4 & \textrm{if }~ p_a(F)=3\al+9 \textrm{ in general or }\\
5 & \textrm{if }~ p_a(F)=3\al+9 \textrm{ only when } \al =5\\
7 & \textrm{if }~ p_a(F)=3\al+9 \textrm{ only when } \al =3\\
5 & \textrm{if  }~ p_a(F)=3\al+8; \textrm{ only ~when } \al=6\\
5 & \textrm{if } ~p_a(F)=3\al+7; \textrm{ only ~when } \al=7.\\

\end{cases}
\]
Recall that $k=\w{F}\cdot f$ on the Hirzebruch surface $\widetilde{S}\cong\mathbb{F}_{\al-1}\rightarrow S$, where $\w{F}$ is the strict transformation of the reduced curve $F\subset S\subset\PP^\al$ under the desingularization $\widetilde{S}\cong\mathbb{F}_{\al-1}\rightarrow S$ of the cone $S$; cf. Remark \ref{minimal}\,(iii). Hence $C\cong \w{F}$ has $g^1_k$ and we arrive at the same contradiction  as the smooth scroll case by Castelnuovo-Severi inequality.
 \end{proof}
\bibliographystyle{spmpsci} 

\end{document}